\documentclass{article}
\usepackage{amssymb}
\usepackage{amsmath}
\usepackage{amsthm}
\usepackage[english]{babel}
\usepackage{authblk}
\usepackage{graphicx}
\usepackage{xcolor}
\newtheorem{thm}{Theorem}[section]

\newtheorem{defn}{Definition}[section]

\newtheorem{rem}{Remark}[section]

\newcommand{\R}{\mathbb{R}}

\newcommand{\e}{\textup{e}}\renewcommand{\i}{\textup{i}}\newcommand{\bff}{\mathbf f}\newcommand{\bfh}{\mathbf h}\newcommand{\bfu}{\mathbf u}\newcommand{\bfy}{\mathbf y}\newcommand{\bfzero}{\mathbf 0}\newcommand{\bfuno}{\mathbf 1}

\newcommand{\bfnn}{{\boldsymbol n}}\newcommand{\bfff}{{\boldsymbol f}}\newcommand{\bfii}{{\boldsymbol i}}\newcommand{\bfjj}{{\boldsymbol j}}\newcommand{\bfkk}{{\boldsymbol k}}
\newcommand{\bftheta}{{\boldsymbol\theta}}

\DeclareMathSymbol{\shortminus}{\mathbin}{AMSa}{"39}

\providecommand{\keywords}[1]{\textit{Keywords } #1}

\begin{document}

\title{Can the a.c.s. notion and the GLT theory handle approximated PDEs/FDEs with either moving or unbounded domains?}

\author[1]{Andrea Adriani\thanks{andrea.adriani@uninsubria.it}}

\author[2]{Alec Jacopo Almo Schiavoni-Piazza\thanks{aschiavo@sissa.it}}

\author[1,3]{Stefano Serra-Capizzano\thanks{s.serracapizzano@uninsubria.it, stefano.serra@it.uu.se}}

\author[4]{Cristina Tablino-Possio\thanks{cristina.tablinopossio@unimib.it}}

\affil[1]{Department of Science and High Technology, Division of Mathematics,  University of Insubria,  Via Valleggio 11, 22100 Como, Italy}

\affil[2]{Scuola Internazionale Superiore di Studi Avanzati, Via Bonomea 265, 34136 Trieste, Italy}

\affil[3]{Department of Information Technology, Division of Scientific Computing, Uppsala University, Box 337, SE-751 05,  Uppsala, Sweden}

\affil[4]{Department of Mathematics and Applications, University of Milano - Bicocca, via Cozzi, 53,  20125  Milano, Italy}

\date{}

\maketitle

\section*{Abstract}

In the current note we consider matrix-sequences $\{B_{n,t}\}_n$ of increasing sizes depending on $n$ and equipped with a parameter $t>0$. For every fixed $t>0$, we assume that each $\{B_{n,t}\}_n$ possesses a canonical spectral/singular values symbol $f_t$ defined on $D_t\subset \R^{d}$ of finite measure, $d\ge 1$. Furthermore, we assume that $ \{ \{ B_{n,t}\} : \, t > 0 \} $ is an approximating class of sequences (a.c.s.) for $ \{ A_n \} $ and that $ \bigcup_{t > 0} D_t = D $ with $ D_{t + 1} \supset D_t $. Under such assumptions and via the notion of a.c.s, we prove results on the canonical distributions of $ \{ A_n \} $, whose symbol, when it exists, can be defined on the, possibly unbounded, domain $D$ of finite or even infinite measure.

We then extend the concept of a.c.s. to the case where the approximating sequence $ \{ B_{n,t}\}_n $ has possibly a different dimension than the one of $ \{ A_n\} $. This concept seems to be particularly natural when dealing, e.g., with the approximation both of a partial differential equation (PDE) and of its (possibly unbounded, or moving) domain $D$, using an exhausting sequence of domains $\{ D_t \}$.

%The main tool in concrete applications is the theory of GLT matrix-sequences and reduced GLT matrix-sequences to which usually all the basic matrix-sequences $\{B_{n,t}\}_n$ belong for every $t>0$.

Examples coming from approximated PDEs/FDEs with either moving or unbounded domains are presented in connection with the classical and the new notion of a.c.s., while numerical tests and a list of open questions conclude the present work.

\ \\
\ \\
\medskip

%\noindent
%$(1)$ \, DICATAM,  Università degli studi di Brescia,  Via Branze 38, 25123 Brescia, Italy\\
%(andrea.adriani@unibs.it ORCID ID 0000-0003-3390-7891)\\
%$(2)$\, Scuola Internazionale Superiore di Studi Avanzati, Via Bonomea 265, 34136 Trieste, Italy\\ (aschiavo@sissa.it ORCID ID: 0000-0003-2652-7081)\\
%$(3)$ \, Department of Science and High Technology, Division of Mathematics,  University of Insubria,  Via Valleggio 11, 22100 Como, Italy\\
%(s.serracapizzano@uninsubria.it ORCID ID 0000-0001-9477-109X)\\
%%%(c)\, Department of Humanities and Innovation,   University of Insubria,  Via Sant’Abbondio 12, 22100 Como, Italy\\
% %% (stefano.serrac@uninsubria.it)\\
%$(4)$\, Department of Information Technology, Division of Scientific Computing, Uppsala University, Box 337, SE-751 05,  Uppsala, Sweden \\ (stefano.serra@it.uu.se)\\
%$(5)$\, Department of Mathematics and Applications, University of Milano - Bicocca, via Cozzi, 53,  20125  Milano, Italy\\
%(cristina.tablinopossio@unimib.it ORCID ID 0000-0003-1424-2767)
%\ \\
%\ \\
%\medskip
%\noindent
\keywords{Discretization of PDEs and FDEs, Unbounded domains, Spectral distribution of matrix-sequences, Approximating class of sequences, GLT theory.}
\section{Introduction}

Partial Differential Equations (PDEs) and more recently Fractional Differential Equations (FDEs) represent standard tools employed for modeling real-world problems in Applied Sciences and Engineering. In particular, the notion of FDE can be considered as a generalization of that of PDE, in which fractional order derivatives are used describing anomalous diffusion processes stemming from concrete applications. The price to pay is the nonlocal nature of the underlying operators, which implies the dense character of the approximated equations and hence a higher computational cost.

In general, when considering either PDEs or FDEs, analytical solutions are not generally known in close form and when they are known via proper representation formulae, it happens that the related computation is costly. Indeed, while the analysis plays a fundamental role in establishing the well-posedness of a given PDE/FDE, numerical methods are crucial for an efficient computation of a numerical solution approximating the solution to the infinite dimensional problem within a certain error.

When using a linear numerical method and the given PDE/FDE is of linear type $Lu=g$ on some domain $\Omega\subset \mathbb{R}^\nu$, $\nu \ge 1$, we end up with a usually large linear system
\begin{equation}\label{basic-discrete}
A_n\bfu_n=\mathbf g_n,
\end{equation}
where $\mathbf g_n$ incorporate the approximation of the known term and the given boundary conditions and $A_n$ is a square matrix of size $d_n$, $d_k< d_{k+1}$, $k\in \mathbb N$. In this way, as $n$ tends to infinity, i.e. the matrix-size $d_n$ tends to infinity, the approximated solution  $\bfu_n$ converges to the solution of the continuous problem $Lu=g$ in a given topology, depending on the continuous problem and on the given numerical method.

Looking at (\ref{basic-discrete}) collectively for every $n$, we observe that we are considering a whole sequence of linear systems with increasing matrix-size $d_n$. Hence it becomes useful to study the collective behavior of the sequence $\{ A_n \}_n$ of coefficient matrices.

%usually of large size if a high accuracy is required, or simply when the dimensionality $\nu$ of the domain domain $\Omega$ is larger than $1$.
%The size of the discretization matrix $L_n$, say $d_n$, increases with $n$ and tends to red infinity as $n\to\infty$. We are then in the presence of a %sequence of linear systems with increasing size.

What is often observed in practice is that the sequence of matrices $\{ A_n \}_n$ enjoys an asymptotic spectral distribution, which is somehow connected to the spectrum of the linear differential operator $L$ associated either with the given PDE or with the given FDE. More in detail, for a large set of test functions $F$, usually for all continuous functions with bounded support or just continuous functions if the spectral norm of $A_n$ is uniformly bounded by a constant independent of $n$, a weak-$*$ convergence stands. More specifically, for every test function $F$, we have
\begin{equation}\label{symbol-relation}
\lim_{n\to\infty}\frac1{d_n}\sum_{j=1}^{d_n}F(\lambda_j(A_n))=\frac1{\mu_m(D)}\int_D
\frac{\sum_{i=1}^p F(\lambda_i(\bff(\bfy)))}p\,{\rm d}\bfy,
\end{equation}
where $\lambda_j(A_n)$, $j=1,\ldots,d_n$, are the eigenvalues of $A_n$, $ \mu_m(\cdot)$ is the Lebesgue measure in $\mathbb R^m$, and $\lambda_i(\bff(\bfy))$, $i=1,\ldots,p$, are the eigenvalues of a certain matrix-valued function
\begin{equation*}\label{symbol-form}
\bff:D\subset\mathbb R^m\to\mathbb C^{p\times p}
\end{equation*}
with $\mu_m(D)\in (0,\infty)$.
The function $\bff$ is referred to as the spectral symbol of the sequence of matrices $\{A_n\}$.

Informally speaking, relation \eqref{symbol-relation} says that, for $n$ large enough,
the spectrum of $A_n$ can be subdivided into $p$ different subsets (or ``branches'') of approximately the same cardinality $d_n/p$, and the $i$-th branch is approximately a uniform sampling over $D$ of the $i$-th eigenvalue function $\lambda_i(\bff(\bfy))$, $i=1,\ldots,p$. In particular, the number $r$ coincides with the number of ``branches'' that compose the spectrum of $A_n$. For instance, if $k=1$, $d_n=np$, and $D=[a,b]$, then up to few possible outliers the eigenvalues of $A_n$ are approximately equal to
\[ \lambda_i\Bigl(\bff\Bigl(a+j\,\frac{b-a}n\Bigr)\Bigr),\qquad j=1,\ldots,n,\qquad i=1,\ldots,p; \]
if $k=2$, $d_n=n^2p$, and $D=[a_1,b_1]\times[a_2,b_2]$, then again up to few possible outliers the eigenvalues of $A_n$ are approximately equal to
\[ \lambda_i\Bigl(\bff\Bigl(a_1+j_1\,\frac{b_1-a_1}n,\,a_2+j_2\,\frac{b_2-a_2}n\Bigr)\Bigr),\qquad j_1,j_2=1,\ldots,n,\qquad i=1,\ldots,p; \]
and so on in a $m$-dimensional setting, with $m>2$ integer.

A complementary notion is that of singular value symbol and singular value distribution. For this further notion, instead of (\ref{symbol-relation}), for every test function $F$, we have
\begin{equation*}\label{sv-symbol-relation}
\lim_{n\to\infty}\frac1{d_n}\sum_{j=1}^{d_n}F(\sigma_j(A_n))=\frac1{\mu_m(D)}\int_D
\frac{\sum_{i=1}^p F(\sigma_i(\bff(\bfy)))}p\,{\rm d}\bfy,
\end{equation*}
where $\sigma_j(A_n)$, $j=1,\ldots,d_n$, are the singular values of $A_n$, $\sigma_i(\bff(\bfy))$, $i=1,\ldots,p$, are the singular values of a certain matrix-valued function $\bff:D\subset\mathbb R^m\to\mathbb C^{p\times p}$ with $\mu_m(D)\in (0,\infty)$ and where the informal meaning exactly mirrors that described above for the spectral symbol.
The function $\bff$ is referred to as the singular value symbol of the sequence of matrices $\{A_n\}$.

It is then clear that the spectral symbol (singular value symbol) $\bff$ provides a ``compact'' and quite accurate description of the spectrum (singular values) of the discretization matrices $A_n$. The identification and the study of the symbol are consequently two important steps in the analysis of $A_n$ and, as a consequence, in the analysis and in the design of fast iterative solvers for the linear systems in (\ref{basic-discrete}), especially for large matrix-sizes $d_n$.

We remind that this type of eigenvalue distribution results are studied since the beginning of the last century in the context of the sequences of Toeplitz matrices generated by real-valued (or Hermitian-valued) functions, as the reader can check in the papers \cite{Tillinota,ty-1,tyrtL1} or in the books \cite{BS,glt-book-1,glt-book-2}, taking into account related references therein. Wide generalizations to the $*$-algebras of GLT matrix-sequences are considered in \cite{glt-book-1,glt-book-2,GLT-block1D,GLT-blockdD} with specific applications to the approximation via local numerical methods of (systems of) PDEs/FDEs also with nonsmooth variable coefficients and irregular but bounded domains/manifolds \cite{GLT-PDE-manifold,cmame1,mc-colloc,nobranches-FV,doro1,branches-DG,immersed NLAA-1,branches-IgA-k-l,applications-axioms,ratn1,ratn2}.

For giving a global picture, despite the great variety of numerical techniques, the parameters in formula (\ref{symbol-relation}) in the GLT analysis depend on a combination of the continuous PDE problem and of the approximation technique. For instance, the typical $D$ is $\Omega \times [-\pi,\pi]^\nu$, if $\Omega$ is a domain in $\mathbb{R}^\nu$ of positive and finite Lebesgue measure, where the PDE under consideration is defined; see the early Locally Toeplitz and GLT papers \cite{glt-laa,glt-Fourier,Tilliloc}. As a consequence, the typical $m$ equals $2\nu$ while, and this is natural, in the case of a submanifold of co-dimension $c$, we have $m=2(\nu -c)$, $1\le c\le \nu-1$. The parameter $p=s \alpha(d)$ ($d=\nu$ or $\nu-c$ in the submanifold setting), where $s$ is the size of the vector function $g$ in the continuous equation $Lu=g$: in the case of a vector PDE obviously we have $s>1$.  More interestingly $\alpha(d)$ depends very much on the approximation scheme, in general of local type, such as Finite Elements \cite{Cia}, Discontinuous Galerkin \cite{DG-book}, Finite Differences \cite{FD}, Finite Volumes \cite{FV-book}, Isogeometric Analysis \cite{IgA} etc. More precisely, in the case of Finite Differences \cite{glt-laa}, Finite Volumes \cite{nobranches-FV}, Isogometric Analysis of degree $k$ and maximal regularity $k-1$ \cite{mc-colloc}, Finite Elements of degree $1$ \cite{BSer}, we have $\alpha(d)=1$, while for Isogeometric Analysis of intermediate regularity $l$ with $l<k-1$, we observe $\alpha(d)=(k-l)^d$ \cite{branches-IgA-k-l}. In the case of Finite Elements of higher order $k$, we have $\alpha(d)=k^d$ \cite{branches-FEM,P_k-elements}, while when using Discontinuous Galerkin $\alpha(d)=(k+l)^d$ \cite{branches-DG}. Notice that the latter two formulas are a special instance of $\alpha(d)=(k-l)^d$, since $l=0$ in the Finite Elements of high order and $l=-1$ in the DG setting, because of the global discontinuity. We also remark that the exponential growth of the type $\alpha(d)=(k-l)^d$ is a drawback from a spectral viewpoint because only one branch of the spectrum is acoustic, that is related to the spectrum of the continuous operator, while the remaining very numerous branches are optical, in the sense that they behave like a pathology introduced by the numerical method (see e.g. \cite{IgA-vibrations} and references therein), and in addition they represent a challenge for designing fast iterative solvers, due to the very involved structure of the spectrum
(see the applications in \cite{GLT-block1D,GLT-blockdD,applications-axioms,glt-book-1,glt-book-2,branches-IgA-k-l}).

The paper is organized as follows. Sections \ref{sec:prel} and \ref{sec:Toeplitz} are devoted to introduce some notation and basic tools, such as the concept of spectral distribution, approximating class of sequences and Toeplitz matrix-sequences. Section \ref{sec:Main} contains the main approximation results. In particular, in Subsection \ref{ssec:gacs}, we introduce the new tool of generalized approximating class of sequences (g.a.c.s.) and we give an approximation result which highlights its usefulness. {In Section \ref{sec:applied-numerical} we give few application for constant coefficient PDEs, complemented by numerical experiments and related visualizations, which are then extended to the case of variable coefficients in Section \ref{sec:varcoeff}}. Finally, Section \ref{sec:end} is devoted to concluding remarks and to mention few open problems and future directions of research.

\section{Spectral distribution and Approximating class of sequences}\label{sec:prel}

In this section we introduce the tools of spectral distribution for a matrix sequence and of approximating class of sequences (a.c.s.), which is a fruitful concept in the theory of numerical approximation and asymptotic spectral analysis of matrix-sequences. We start with the definitions and then recall some useful results connecting the two concepts.

\begin{defn}\label{def:distributions}
Let $A_n\}_n$ be a matrix-sequence and let $f : D\to\mathbb C^{p_1\times p_2}$ be a measurable Hermitian matrix-valued function defined on the measurable (bounded) set $D\subset\mathbb R^m$, with $0<\mu_m(D)<\infty$. We write that $\{A_{n}\}_n$ is distributed as $f$ in the sense of singular values in $ D $ and we write $\{A_{n}\}_n \sim_\sigma (f,D)$, if
\begin{equation*}
\lim_{n \to \infty} \sum_{j=1}^{d_n \wedge d_n'} \frac{F(\sigma_j(A_n))}{d_n \wedge d_n'}  = \frac1{\mu(D)} \int_D \sum_{k=1}^p  \frac{F\left(\sigma_k(f(s))\right)}{p}ds,\qquad\forall F\in C_c(\mathbb{R}),
\end{equation*}
where $\sigma_1(f(s)), \ldots,\sigma_p(f(s))$ are the singular values of $f(s)$, $p=p_1 \wedge p_2$. We call $f$ the singular value symbol of $\{A_{n}\}_n$.

With the same notation as before, imposing $d_n=d_n'$, we write that $\{A_{n}\}_n$ is distributed as $f$ in the sense of eigenvalues in $ D $ with necessarily $p_1=p_2=p$ and we write $\{A_{n}\}_n \sim_\lambda (f,D)$, if
\begin{equation*}
\lim_{n \to \infty} \sum_{j=1}^{d_n} \frac{F(\lambda_j(A_n))}{d_n }  = \frac1{\mu(D)} \int_D \sum_{k=1}^p  \frac{F\left(\lambda_k(f(s))\right)}{p}ds,\qquad\forall F\in C_c(\mathbb{R}),
\end{equation*}
where $\lambda_1(f(s)), \ldots,\lambda_p(f(s))$ are the eigenvalues of $f(s)$.
We call $f$ the spectral symbol of $\{A_{n}\}_n$.
\end{defn}

\begin{defn}
Let $\{A_{n}\}_n$ be a matrix-sequence of size $d_n \times d_n'$ with $ d_n $ and $ d_n' $ monotonically increasing integer sequences and let $\left\{\{B_{n,t}\}_n\right\}_{t}$ be a sequence of matrix-sequences of the same size $d_n \times d_n'$. We say that $\left\{\{B_{n,t}\}_n\right\}_{t}$ is an approximating class of sequences (a.c.s.) for $\{A_{n}\}_n$ %and we write
%	$$
%	\{B_{n,t}\}_\bfnn \to\{A_{n}\}_n \qquad  \mbox{a.c.s.},
%	$$
if the following condition is met: for every $t$ there exists $n_t$ such that, for $n>n_t$,
\begin{equation*}
A_n =  B_{n,t} + R_{n,t} + N_{n,t}, \qquad \textnormal{rank}\left(R_{n,t}\right) \leq c(t) d_n \wedge d_n',
\end{equation*}
$$
\left\|N_{n,t}\right\|\leq \omega(t),
$$
where $n_t, c(t), \omega(t)$ depend only on $t$, and
$$
\lim_{t \to \infty} c(t) = \lim_{t \to \infty} \omega(t) =0.
$$
\end{defn}

The following theorem constitutes a link between a.c.s. and spectral distribution.

 \begin{thm}\label{thm:main}
Let $\{A_{n}\}_n$ be a matrix-sequence of size $d_n \times d_n'$ with $ d_n $ and $ d_n' $ monotonically increasing integer sequences and let $\left\{\{B_{n,t}\}_n\right\}_{t}$ be an a.c.s. for $ \{ A_n \} $. Suppose that $ \left\{ \{ B_{n,t}\}_n\right\}_{t} \sim_{\sigma} (f_t, D) $ and $ f_t $ converges in measure to $ f $, then $ \{ A_n \}_n \sim_{\sigma} (f,D) $.
Furthermore, if $d_n = d_n'$, all the involved matrices are Hermitian, and $ \left\{ \{ B_{n,t}\}_n\right\}_{t} \sim_{\lambda} (f_t, D) $ and $ f_t $ converges in measure to $ f $, then $ \{ A_n \}_n \sim_{\lambda} (f,D) $
\end{thm}

The tool is quite powerful and it was introduced in the seminal work by Tilli \cite{Tilliloc} on locally Toeplitz matrix-sequences. Then in \cite{acs-laa} the terminology was provided and further results were given.
 An account of the general theory on generalized Locally Toeplitz (GLT) matrix-sequences, based on the a.c.s. notion, can be found in \cite{GLT-block1D,GLT-blockdD,glt-book-1,glt-book-2,branches-IgA-k-l} and references therein.

In these books and long research/exposition papers one can find also examples of applications ranging from approximated integro-differential equations, approximated partial differential equations, approximated fractional differential equations with any kind of methods (finite elements, fine differences, isogeometric analysis, finite volumes etc) and with very mild assumptions: only Riemann integrability in the case of variable coefficients, only Peano-Jordan measurability of the domain (see e.g. \cite{PeJo} for the latter two notions), grids approximated by a given function applied on uniform grids. In particular the richness in the potential domains is obtained via the notion of reduced GLT matrix sequences: see \cite{glt-laa}[pp. 395–399] for a detailed example, \cite{glt-Fourier}[Section 3.1.4] for a initial proposal and related terminology, and the long dense paper \cite{Barb} for a systematic treatment of the subject and for a complete theoretical development.

However, in all cases the idea is the immersion of the given Peano-Jordan measurable domain into a cube or rectangle in the appropriate number of dimensions, where this idea is related to the immersed methods \cite{immersed bible,immersed book} and to the less recent idea of fictitious domains \cite{kuznetsov2,kuznetsov1}. The reader is referred also to \cite{immersed NLAA-1,immersed NLAA-2} for an asymptotical analysis and numerical methods regarding standard PDEs and fractional PDEs using the immersion idea and the reduced GLT tools.

However, the previous techniques are not applicable in the case of unbounded domain. The theorems in the present paper and the applications in Section \ref{sec:applied-numerical} {and \ref{sec:varcoeff}} fill the gap and complete the picture, by considering moving and unbounded domains with either finite or infinite Lebesgue measure. Here, the basic a.c.s. and the new g.a.c.s. notions are the main tools which allow us to go beyond the GLT machinery. The new g.a.c.s. concept allows to deal with matrix-sequences defined by different sequences of dimensions and this may happen e.g. in several approximation schemes for differential equations. The idea is already present in \cite{ratn1} in Theorem 4.3, Corollary 4.4 and it is reminiscent of the extradimensional approach proposed by Tyrtyshnikov more than two decades ago (see again \cite{ratn1}, beginning of Section 4.2).

\section{Multi-index notation, Toeplitz and multi-level Toeplitz matrices}\label{sec:Toeplitz}

In this section, we first introduce a multi-index notation that we use hereafter. Given an integer $d\geq 1$, a $d$-index $\bfkk$ is an element of $\mathbb{Z}^d$, that is, $\bfkk=\left(k_1, \ldots, k_d\right)$ with $k_r \in \mathbb{Z}$ for every $r=1,\ldots, d$. We intend $\mathbb{Z}^d$ equipped with the standard lexicographic ordering, that is, given two $d$-indices $\bfii=(i_1,\ldots,i_d)$, $\bfjj=(j_1,\ldots,j_d)$, we write $\bfii \vartriangleleft \bfjj$ if $i_r<j_r$ for the first $r=1,2,\ldots,d$ such that $i_r\neq j_r$. The relations $\trianglelefteq,\vartriangleright,\trianglerighteq$ are defined accordingly.
	
	Given two $d$-indices $\bfii,\bfjj$, we write $\bfii < \bfjj$ if $i_r< j_r$ for every $r=1,\ldots,d$. The relations $\leq,>,\geq$ are defined accordingly.
	
We use bold letters for vectors and vector/matrix-valued functions. We indicate with $\bf{0},\bf{1},\bf{2},\ldots$,  the $d$-dimensional constant vectors $\left(0,0,\ldots,0\right)$, $\left(1,1,\ldots,1\right)$, $\left(2,2,\ldots,2\right),\ldots$, respectively. With the notation $\frac{\bfii}{\bfnn}$ we mean the element-wise division of vectors, i.e., $\frac{\bfii}{\bfnn}= \left(\frac{i_1}{n_1},\ldots,\frac{i_d}{n_d} \right)$. We write $|\bfii|$ for the vector $\left(|i_1|, \ldots, |i_d|\right)$. Finally, given a $d$-index $\bfnn$, we write $\bfnn \to \infty$ meaning that $\min_{r=1,\ldots,d}\{n_r\}\to \infty$.\\

A Toeplitz matrix of order $n$ is characterized by the fact that all the diagonals are constant: $ (T_n)_{i,j} =t_{i-j} $ for $ i,j=1,...,n $ and some coefficients $ t_k$, $k=1-n,...,n-1$:
$$
T_n=\left(
\begin{array}{cccc}
\textsf{t}_0      &   \textsf{t}_{-1}  & \cdots & \textsf{t}_{1-n}\\
\textsf{t}_1    &    \textsf{t}_0   & \ddots & \vdots \\
\vdots & \ddots & \ddots & \textsf{t}_{-1} \\
\textsf{t}_{n-1} & \cdots & \textsf{t}_1     & \textsf{t}_0\\
\end{array}
\right).
$$
When every term $ t_k $ is a matrix of fixed size $ p_1\times p_2 $, we say that $T_n$ is of block Toeplitz type. The definition of $d$-level Toeplitz matrices is more involved and it is based on the following recursive idea: a $d$-level Toeplitz matrix is a Toeplitz matrix where each coefficient $ t_k $ is a $(d-1)$-level Toeplitz matrix and so on. Using standard multi-index notation, we can give a more detailed definition as follows: a $d$-level Toeplitz matrix is a matrix $ T_{\mathbf{n}} $ such that
$$T_\bfnn=\left(\textsf{t}_{\bfii-\bfjj}\right)_{\bfii,\bfjj=\bfuno}^{\bfnn}\in\mathbb C^{(n_1\cdots n_d)\times(n_1\cdots n_d)},$$
with the multi-index $\bfnn$ such that $\bfzero<\bfnn=(n_1,\ldots,n_d)$ and $\textsf{t}_\bfkk\in\mathbb C$,  $-(\bfnn-\bfuno)\trianglelefteq \bfkk\trianglelefteq \bfnn-\bfuno$. If the basic element $ t_{\bfkk} $ is a block of fixed size $p_1\times p_2$, $\max\{p_1,p_2\} \geq 2 $, we write that the matrix is a $d$-level block Toeplitz matrix and we denote it by $ T_{\bfnn,p_1,p_2} $:
$$
T_{\bfnn,p_1,p_2}=\left(\textsf{t}_{\bfii-\bfjj}\right)_{\bfii,\bfjj=\bfuno}^{\bfnn}\in\mathbb C^{(n_1\cdots n_d p_1)\times(n_1\cdots n_d p_2)}, \qquad \textsf{t}_\bfkk\in\mathbb{C}^{p_1\times p_2}.
$$
Given now a function $ \bfff : [-\pi,\pi]^d \to \mathbb{C}^{p_1\times p_2} $ in $ {L}^1([-\pi,\pi]^d) $, we denote its Fourier coefficients by
\begin{equation*}
\hat{\bfff}_\bfkk=\frac1{(2\pi)^d}\int_{[-\pi,\pi]^d}\bfff(\bftheta)\e^{-\i\,\bfkk\cdot\bftheta}d\bftheta\in\mathbb C^{p_1\times p_2},\qquad\bfkk\in\mathbb Z^d, \quad \bfkk\cdot\bftheta=\sum_{r=1}^d k_r \theta_r,
\end{equation*}
(the integrals are done component-wise), and we associate to $\bfff$ the family of $d$-level block Toeplitz matrices
\begin{equation*}
T_{\bfnn,p_1,p_2}(\bfff):=\left(\hat{\bfff}_{\bfii-\bfjj}\right)_{\bfii,\bfjj=\bfuno}^{\bfnn},\qquad\bfnn\in\mathbb N^d.
\end{equation*}
In this context, $\bfff$ is called the generating function of the sequence $ \{ T_{\bfnn,p_1,p_2} \} $.\\
The following result links the definition of symbol function and generating function for Toeplitz matrix-sequences.
\begin{thm}[\cite{Tillinota,Tillicomplex,DNS}]\label{thm:symbol_d-block-toeplitz}
Let $\bfff:[-\pi,\pi]^d\to\mathbb C^{p_1\times p_2}$ be a function belonging to ${L}^1([-\pi,\pi]^d)$, $p_1,p_2,d\ge 1$. Then
$$
\left\{T_{\bfnn,p_1,p_2}(\bfff)  \right\}_\bfnn \sim_\sigma \bfff,
$$
that is the generating function of $\left\{T_{\bfnn,p_1,p_2}(\bfff)  \right\}_\bfnn$ coincides with its singular value symbol.
When $p_1=p_2=p$ and $\bfff$ is Hermitian-valued almost everywhere or belongs to the Tilli class, i.e. $\bfff$ is essentially bounded, the closure of its range has empty interior, and the range does not disconnect the complex field (see \cite{Tillicomplex} when $p=1$ and \cite{DNS} when $p>1$), we have
\[
\left\{T_{\bfnn,p,p}(\bfff)  \right\}_\bfnn \sim_\lambda \bfff,
\]
that is the generating function of $\left\{T_{\bfnn,p,p}(\bfff)  \right\}_\bfnn$ coincides with its spectral symbol. Here, when $p_1=p_2=p$, for range of $\bfff$ we mean the union of the ranges of the $p$ eigenvalue functions of $\bfff$.
\end{thm}

\section{Main results}\label{sec:Main}

We start this section by stating and proving a theorem which extends the practical purposes of the a.c.s. notion to the case of unbounded domains, possibly of infinite measure, using exhaustions by bounded sets of finite measure.

\begin{thm}\label{main_sv}
Let $\{A_n\}$ be a matrix-sequence with $A_n$ of size $ d_n \times d_n' $ with $ d_n $ and $ d_n' $ monotonically increasing integer sequences. Let $ D $ be a measurable set in $ \R^{d} $ for some $ d \geq 1 $, possibly unbounded and of infinite measure
and let $f:D\rightarrow \mathbb{C}^{p_1\times p_2}$, $p_1,p_2\ge 1$, be a measurable function with $p=p_1 \wedge p_2$.
Let $ D_t $ be an exhaustion of $ D $, that is, $ D_{t + 1} \supset D_t $ for every $ t > 0 $ and $ \bigcup_{t > 0} D_t = D $. Assume that $ \exists \{ B_{n,t}\} $ such that $ \{ \{ B_{n,t}\} : \, t > 0 \} $ is an a.c.s. for $ \{ A_n \} $ with $\{ B_{n,t} \} \sim_{\sigma} (f_t, D_t)$, i.e.,
\begin{equation*}
\lim_{n \to \infty} \sum_{j=1}^{d_n \wedge d_n'} \frac{F(\sigma_j(B_{n,t})}{d_n \wedge d_n'} = \frac{1}{\mu(D_t)} \int_{D_t} \sum_{k=1}^{p} \frac{F(\sigma_k(f_t(s)))}{p} ds = \alpha_t(F)
\end{equation*}
for every $ F \in C_c(\R) $.
Assume that, for every $ F \in C_c(\R) $,
\begin{equation*}
\lim_{t \to \infty} \alpha_t(F) = \alpha(F).
\end{equation*}
Then
\begin{equation}\label{limit}
\lim_{n \to \infty} \sum_{j=1}^{d_n \wedge d_n'} \frac{F(\sigma_j(A_n)}{d_n \wedge d_n'} = \alpha(F).
\end{equation}
If in addition we have
\begin{equation}\label{limit_func}
\lim_{t \to \infty} f_t^E = f
\end{equation}
in $D$ almost everywhere, with
\begin{equation*}
f_t^E= \begin{cases}
f_t(s) & \text{ if } s \in D_t \\
0 & \text{ otherwise},
\end{cases}
\end{equation*}
then we write $ \{ A_n \} \sim_{\sigma, \text{moving}} (f,D) $. Moreover, if $ \mu(D) < \infty$, we have
\begin{equation*}
\alpha(F) = \frac{1}{\mu(D)} \int_{D} \sum_{k=1}^{p} \frac{F(\sigma_k(f(s)))}{p} ds.
\end{equation*}
\end{thm}

\begin{proof}
By the assumption $\{ B_{n,t}\} \sim (f_t, D_t) $ we know that $ \mu(D_t) >0 $ and hence, since $ D \supset D_t $ for every $ t >0 $, we deduce that $ \mu(D) >0 $.\\
First, we note that relation \eqref{limit} is true if and only if it is true for any $ F $ in $ C^{1} $ with bounded support by simple functional approximation techniques. At this point, taking $ F=F^\uparrow +  F^\downarrow $, with $ F^\uparrow(x) = \int_{0}^{x} (F')^+(t) dt $ and $ F^\downarrow (x) = \int_{0}^{x} (F')^-(t) dt$, where $ g^+ =\max(g, 0) $ and $ g^-=\max(-g,0) $, we deduce by linearity that \eqref{limit} is true if and only if it holds for every $ G $ continuous, differentiable (with $ G' $ nonnegative), being equal to $ 0 $ for every $ x \leq c_1 $ and being constant (equal to $ \|G\|_{\infty} $) for every $ x \geq c_2 $.
Now, since $ \{ \{ B_{n,t}\} \, : \, t > 0 \} $ is an a.c.s. for $ \{ A_n \} $, we deduce that $ \exists s(n,t) \leq \alpha(t) (d_n \wedge d_n') $ for which
\begin{equation}\label{ineq_singval}
\sigma_{j+\alpha(t)(d_n \wedge d_n')}(B_{n,t}) - \frac{1}{t} \leq \sigma_j(A_n) \leq \sigma_{j-\alpha(t)(d_n \wedge d_n')} (B_{n,t}) + \frac{1}{t},
\end{equation}
where $ \lim_{t \to \infty} \alpha(t) = 0 $ and $ j-\alpha(t) (d_n \wedge d_n') $ has to be intended equal to $1$ if $j-\alpha(t) (d_n \wedge d_n')\leq 1$ and it has to be intended equal to $ d_n \wedge d_n' $ if $ j-\alpha(t) (d_n \wedge d_n') \geq d_n \wedge d_n' $.
Inequalities in \eqref{ineq_singval}, together with the features of $ G $, lead to
\begin{align*}
\frac{1}{d_n \wedge d_n'} \sum_{j=1}^{d_n \wedge d_n'} G \bigg( \sigma_j(B_{n,t}) &-\frac{1}{t} \bigg) - \| G \|_{\infty} \alpha(t) \leq \\
&\leq \frac{1}{d_n \wedge d_n'} \sum_{j=1}^{d_n \wedge d_n'} G(\sigma_j(A_n)) \leq \frac{1}{d_n \wedge d_n'} \sum_{j=1}^{d_n \wedge d_n'} G \left( \sigma_j(B_{n,t})+\frac{1}{t} \right) + \| G \|_{\infty} \alpha(t).
\end{align*}
Since
\begin{equation*}
G\left(x+\frac{1}{t}\right) \leq G(x)+ \|G'\|_{\infty}\frac{1}{t} \quad \text{ and }
G\left(x-\frac{1}{t}\right) \geq G(x)- \|G'\|_{\infty}\frac{1}{t}
\end{equation*}
we conclude that
\begin{align*}
\frac{1}{d_n \wedge d_n'} \sum_{j=1}^{d_n \wedge d_n'} G ( &\sigma_j(B_{n,t})) -\frac{1}{t} \|G'\|_{\infty} - \| G \|_{\infty} \alpha(t) \leq \\
&\leq \frac{1}{d_n \wedge d_n'} \sum_{j=1}^{d_n \wedge d_n'} G(\sigma_j(A_n)) \leq \frac{1}{d_n \wedge d_n'} \sum_{j=1}^{d_n \wedge d_n'} G(\sigma_j(B_{n,t}))+\frac{1}{t} \|G'\|_{\infty} + \| G \|_{\infty} \alpha(t).
\end{align*}
Since
\begin{equation*}
\lim_{t \to \infty} \frac{1}{t} \| G' \|_{\infty} \pm \alpha(t) \| G\|_{\infty} = 0,
\end{equation*}
taking any $ \varepsilon > 0$ there exists $ n_{\varepsilon} = n_{\varepsilon,t,G} $ such that for every $ n \geq n_{\varepsilon} $ we have
\begin{equation*}
\frac{1}{d_n \wedge d_n'} \sum_{j=1}^{d_n \wedge d_n'} G(\sigma_j(A_n)) \in \left[ \alpha_t(G)-\frac{\varepsilon}{2}, \alpha_t(G) + \frac{\varepsilon}{2} \right] \subseteq \left[ \alpha(G)-\varepsilon, \alpha(G) + \varepsilon \right].
\end{equation*}
This concludes the first part of the theorem.\\
For the in addition part, we observe that under the assumption $ \mu(D) < \infty $, we have $ \lim_{t \to \infty} \mu(D_t)=\mu(D) $ and $\lim_{t \to \infty} \mu ( D\setminus D_t ) =0$. Using Equation \eqref{limit_func}, it follows that
\begin{align*}
\lim_{t \to \infty} \int_{D_t} \frac{\sum_{k=1}^p F(\sigma_k(f_t(s)))}{p} ds &= \lim_{t \to \infty} \left( \int_{D} \frac{\sum_{k=1}^p F(\sigma_k(f_t^E(s)))}{p} ds - F(0) \mu(D \setminus D_t) \right)\\
&=  \int_{D} \frac{\sum_{k=1}^p F(\sigma_k(f(s)))}{p} ds,
\end{align*}
ending the proof.
\end{proof}

Under the hypothesis that the sequences are constituted by Hermitian matrices, we can extend the previous result to the case of eigenvalues by obtaining the theorem below. The proof is an eigenvalue version of the proof of Theorem \ref{main_sv}: the steps are identical and we leave them to the interested reader.

\begin{thm}\label{main_eig}
Let $\{A_n\}$ be a matrix-sequence with $A_n$ of order $ d_n$ such that $ A_n = A_n^*$ with $ d_n $ and monotonically increasing integer sequence. Let $ D $ be a measurable set in $ \R^{d} $ for some $ d \geq 1 $, possibly unbounded and of infinite measure
and let $f:D\rightarrow \mathbb{C}^{p\times p}$, $p\ge 1$, be a measurable function.
Let $ D_t $ be an exhaustion of $ D $, that is, $ D_{t + 1} \supset D_t $ for every $ t > 0 $ and $ \bigcup_{t > 0} D_t = D $. Assume that $ \exists \{ B_{n,t}\} $ such that $B_{n,t}=B_{n,t}^* $ for every $ n,t $ and that $ \{ \{ B_{n,t}\} : \, t > 0 \} $ is an a.c.s. for $ \{ A_n \} $ with $\{ B_{n,t} \} \sim_{\lambda} (f_t, D_t)$, i.e.,
\begin{equation*}
\lim_{n \to \infty} \sum_{j=1}^{d_n } \frac{F(\lambda_j(B_{n,t})}{d_n } = \frac{1}{\mu(D_t)} \int_{D_t} \sum_{k=1}^{p} \frac{F(\lambda_k(f_t(s)))}{p} ds = \alpha_t(F)
\end{equation*}
for every $ F \in C_c(\R) $.
Assume that, for every $ F \in C_c(\R) $,
\begin{equation*}
\lim_{t \to \infty} \alpha_t(F) = \alpha(F).
\end{equation*}
Then
\begin{equation}\label{limit}
\lim_{n \to \infty} \sum_{j=1}^{d_n } \frac{F(\lambda_j(A_n)}{d_n } = \alpha(F).
\end{equation}
If in addition we have
\begin{equation}\label{limit_func}
\lim_{t \to \infty} f_t^E = f
\end{equation}
in $D$ almost everywhere, with
\begin{equation*}
f_t^E= \begin{cases}
f_t(s) & \text{ if } s \in D_t \\
0 & \text{ otherwise},
\end{cases}
\end{equation*}
then we write $ \{ A_n \} \sim_{\lambda, \text{moving}} (f,D) $. Moreover, if $ \mu(D) < \infty$, we have
\begin{equation*}
\alpha(F) = \frac{1}{\mu(D)} \int_{D} \sum_{k=1}^{p} \frac{F(\lambda_k(f(s)))}{p} ds.
\end{equation*}
\end{thm}

\subsection{Generalized Approximating class of sequences}\label{ssec:gacs}

In this section, we introduce a new tool inspired by the a.c.s. theory and by the extradimensional
approach (see \cite{ratn1}, Section 4.1 and Section 4.2), but more flexible and able to handle all the parameters needed in order to study discretizations on unbounded domains of finite measure, or moving domains. As for the classical a.c.s. tool, also the notion of generalized a.c.s. carries information about spectral and singular value distribution, which is what we prove in Theorem \ref{gacs_sing_val}. Hereafter, we focus on domains with finite measure, which will be the ones treated in the numerical tests.

\begin{defn}\label{def_gacs}
Let $\{A_{n}\}_n$ be a matrix-sequence of size $d_n \times d_n'$ with $ d_n $ and $ d_n' $ monotonically increasing integer sequences and let $\left\{\{B_{n,t}\}_{n}\right\}_{t}$ be a sequence of matrix-sequences of size $d_{n,t} \times d_{n,t}'$. We say that $\left\{\{B_{n,t}\}_{n}\right\}_{t}$ is a generalized approximating class of sequences (g.a.c.s.) for $\{A_{n}\}_n$
if, for every $t$, there exists $n_t$ such that, for $n>n_t$,
\begin{equation*}
	A_n =  U_{n,t} \left( B_{n,t}\oplus0_{n,t}\right)V_{n,t} + R_{n,t} + N_{n,t},
\end{equation*}
where $0_{n,t}$ is the null matrix of size $\left(d_n-d_{n,t}\right) \times \left(d_n'-d_{n,t}\right)$, $U_{n,t}$ and $V_{n,t}$ are two unitary matrices of order $d_n \times d_n$ and $d_n' \times d_n'$ respectively and $R_{n,t}, N_{n,t}$ are matrices of the same size of $A_n$, satisfying:
\begin{equation*}
	\textnormal{rank}\left(R_{n,t}\right) \leq c(t) d_n \wedge d_n',
\end{equation*}
\begin{equation*}
	\left\|N_{n,t}\right\|\leq \omega(t),
\end{equation*}
\begin{equation*}
	d_n \wedge d_n'-d_{n,t} \wedge d_{n,t}'=:m_{n,t} \leq m(t)d_n \wedge d_n',
\end{equation*}
\begin{equation*}
	\lim_{t \to \infty} c(t) =\lim_{t \to \infty} \omega(t) = \lim_{t \to \infty} m(t) =0.
\end{equation*}
If we have a sequence $\{A_{n}\}_n$	of square Hermitian matrices, we also ask $\left\{\{B_{n,t}\}_{n}\right\}_{t}$ to be square Hermitian and $V_{n,t}=U_{n,t}^*$ for all $n$ and $t$.
\end{defn}

\begin{rem}
The notion of g.a.c.s. is intended as a generalization of the idea of a.c.s. being more natural when dealing with the approximation of infinite dimensional operators over moving or unbounded domains. In particular, in the case of a discretization of a (fractional) differential operator  over moving or unbounded domains we usually end up with a sequence of domains (either a precompact exhaustion of the unbounded domain or the sequence of moving domains or a combination of both). Using the same approximation procedure (Finite Differences, Finite Elements, Isogeometric Analysis, Finite Volumes etc), it is natural to obtain matrices of different dimensions and g.a.c.s. is intended as a tool for dealing with this difficulty in the spirit of the extradimensional approach used in Sections 4.1, 4.2 in \cite{ratn1}.
\end{rem}

\begin{thm}\label{gacs_sing_val}
	Let $\{A_n\}$ be a matrix-sequence of size $ d_n \times d_n' $ with $ d_n $ and $ d_n' $ monotonically increasing integer sequences. Let $\left\{\{B_{n,t}\}_{n}\right\}_{t}$ be a g.a.c.s. for $\{A_n\}$. If, $\forall t$, $\exists\left(f_t,D_t\right)$ such that: \\
	\begin{itemize}
		\item $\left\{\{B_{n,t}\}_{n}\right\}_{t} \sim_{\sigma} (f_t,D_t)$,\\
		\item $D_t \subset D_{t+1}$, $\forall t$,\\
		\item $D := \bigcup\limits_{t>0} D_t$ of finite measure,\\
		\item $\exists f: D\to\mathbb{C}^{p_1 \times p_2}
		$, $p_1,p_2\ge 1$, measurable such that $f^{E}_t\to f$ in measure, $t\to +\infty$, with $p=p_1 \wedge p_2$ and
	\end{itemize}
	\begin{equation*}
	f_t^E= \begin{cases}
	f_t(s) & \text{if } s \in D_t \\
	0 & \text{if } s \in D\setminus D_t,
	\end{cases}
	\end{equation*}
	then $\{A_n\} \sim_{\sigma} (f,D)$, i.e.
	\begin{equation}\label{limit_sing}
		\lim_{n \to \infty}\sum_{j=1}^{d_n \wedge d_n'} \frac{F(\sigma_j(A_n))}{d_n \wedge d_n'} = \frac{1}{\mu(D)} \int_{D} \sum_{k=1}^{p} \frac{F(\sigma_k(f(s)))}{p} ds, \qquad\forall F\in C_c(\mathbb{R})
	\end{equation}
\end{thm}

\begin{proof}
	First, we note that relation \eqref{limit_sing} is true if and only if it is true for any $ F $ in $ C^{1} $ with bounded support by simple functional approximation techniques. At this point, taking $ F=F^\uparrow +  F^\downarrow $, with $ F^\uparrow(x) = \int_{-\infty}^{x} (F')^+(t) dt $ and $ F^\downarrow (x) = \int_{-\infty}^{x} (F')^-(t) dt$, where $ g^+ =\max(g, 0) $ and $ g^-=\max(-g,0) $, we deduce by linearity that \eqref{limit_sing} is true if and only if it holds for every $ G $ continuous, differentiable (with $ G' $ nonnegative), being equal to $ 0 $ for every $ x \leq c_1 $ and being constant (equal to $ \|G\|_{\infty} $) for every $ x \geq c_2 $.
	
	Since, by hypotesis, $\left\{\{B_{n,t}\}_{n}\right\}_{t} \sim_{\sigma} (f_t,D_t)$, $D_t\subset D_{t+1}$ and $D := \bigcup\limits_{t>0} D_t$ is of finite measure, we have:
	\begin{equation*}
		0 < \mu(D_t)\leq\mu(D_{t+1})\to\mu(D)<+\infty, \qquad t\to+\infty
	\end{equation*}
	and, since 	$f^{E}_t\to f$ in measure, $t\to +\infty$, we also have
	\begin{equation*}
		\sigma_k(f^{E}_t) \to \sigma_k(f) \quad\text{in measure } \qquad\forall 1\leq k\leq p
	\end{equation*}	
	hence,
	\begin{equation*}
		\lim_{t \to \infty}\frac{1}{\mu(D_t)}\int_{D_t} \sum_{k=1}^{p} \frac{G(\sigma_k(f_t(s)))}{p} ds = \lim_{t \to \infty}\frac{\mu(D)}{\mu(D_t)}\frac{1}{\mu(D)}\big(\int_{D} \sum_{k=1}^{p} \frac{G(\sigma_k(f^{E}_t(s)))}{p} ds - G(0)\mu(D\setminus D_t)\big) =
	\end{equation*}
	\begin{equation}\label{cont_side}
		 = \frac{1}{\mu(D)} \int_{D} \sum_{k=1}^{p} \frac{G(\sigma_k(f(s)))}{p} ds
	\end{equation}\\
	Now, since $\forall n>n_t$ $A_n =  U_{n,t} \left( B_{n,t}\oplus0_{n,t}\right)V_{n,t} + R_{n,t} + N_{n,t}$, we set $C_{n,t}=U_{n,t} \left( B_{n,t}\oplus 0_{n,t}\right)V_{n,t}$ and we notice that the singular values of $C_{n,t}$ are the singular values of $B_{n,t}$ with $m_{n,t} = d_n \wedge d_n'-d_{n,t} \wedge d_{n,t}'$ additional singular values equal to zero. By classic results of interlacing of singular values under rank and norm corrections, we have
	\begin{equation*}
		\sigma_j(A_n) = \sigma_j(C_{n,t}+ R_{n,t} + N_{n,t})\leq\sigma_j(C_{n,t}+ R_{n,t}) + \omega(t)\leq\sigma_{j+(d_n \wedge d_n')c(t)}(C_{n,t}) + \omega(t)
	\end{equation*}\\
	where $\sigma_{j}(C_{n,t}):=+\infty$ for $j > d_n \wedge d_n'$.
	On the basis of the initial discussion, we take $ G $ differentiable, monotone non-decreasing, positive and bounded (so $\left\|G\right\|_{\infty} = G(+\infty)$) and we have:
	\begin{align}\label{part1}
		\sum_{j=1}^{d_n \wedge d_n'} \frac{G(\sigma_j(A_n))}{d_n \wedge d_n'}=\sum_{j=1}^{d_n \wedge d_n'} \frac{G(\sigma_j(C_{n,t}+ R_{n,t} + N_{n,t}))}{d_n \wedge d_n'}\leq
	\end{align}
	\begin{align*}
		\leq\sum_{j=1}^{d_n \wedge d_n'} \frac{G(\sigma_{j+(d_n \wedge d_n')c(t)}(C_{n,t}) + \omega(t))}{d_n \wedge d_n'}\leq\sum_{j=1}^{d_n \wedge d_n'} \frac{G(\sigma_{j+(d_n \wedge d_n')c(t)}(C_{n,t}))}{d_n \wedge d_n'}+\omega(t)\left\|G'\right\|_{\infty}\leq\\
		\leq\sum_{j=1+(d_n \wedge d_n')c(t)}^{d_n \wedge d_n'} \frac{G(\sigma_{j}(C_{n,t}))}{d_n \wedge d_n'}+c(t)\left\|G\right\|_{\infty}+\omega(t)\left\|G'\right\|_{\infty}\leq\sum_{j=1}^{d_n \wedge d_n'} \frac{G(\sigma_{j}(C_{n,t}))}{d_n \wedge d_n'}+c(t)\left\|G\right\|_{\infty}+\omega(t)\left\|G'\right\|_{\infty}
	\end{align*}\\
	Taking into account the information about the singular values of $C_{n,t}$, it follows that:\\
	\begin{equation}\label{part2}
		\sum_{j=1}^{d_n \wedge d_n'} \frac{G(\sigma_{j}(C_{n,t}))}{d_n \wedge d_n'}=\frac{d_{n,t} \wedge d_{n,t}'}{d_n \wedge d_n'}\sum_{j=1}^{d_{n,t} \wedge d_{n,t}'} \frac{G(\sigma_{j}(B_{n,t}))}{d_{n,t} \wedge d_{n,t}'} + \frac{m_{n,t}}{d_n \wedge d_n'}G(0)\leq
	\end{equation}
	\begin{equation*}
		\leq\sum_{j=1}^{d_{n,t} \wedge d_{n,t}'} \frac{G(\sigma_{j}(B_{n,t}))}{d_{n,t} \wedge d_{n,t}'} + m(t)\left\|G\right\|_{\infty}
	\end{equation*}
	Now, given $\epsilon>0$, we choose $t$ large enough such that $\omega(t),c(t),m(t)<\epsilon$ and also, by (\ref{cont_side}):
	\begin{equation}\label{part3}
		\frac{1}{\mu(D_t)}\int_{D_t} \sum_{k=1}^{p} \frac{G(\sigma_k(f_t(s)))}{p} ds\leq \frac{1}{\mu(D)} \int_{D} \sum_{k=1}^{p} \frac{G(\sigma_k(f(s)))}{p} ds + \epsilon
	\end{equation}
	Once $t$ is fixed, using that $\left\{\{B_{n,t}\}_{n}\right\}_{t} \sim_{\sigma} (f_t,D_t)$, we can find $N_{\epsilon}>n_t$ such that $\forall n>N_{\epsilon}$ we have:
	\begin{equation}\label{part4}
		\sum_{j=1}^{d_{n,t} \wedge d_{n,t}'} \frac{G(\sigma_{j}(B_{n,t}))}{d_{n,t} \wedge d_{n,t}'}\leq\frac{1}{\mu(D_t)}\int_{D_t} \sum_{k=1}^{p} \frac{G(\sigma_k(f_t(s)))}{p} ds + \epsilon
	\end{equation}
	Finally, using the derivations in (\ref{part1}),(\ref{part2}),(\ref{part3}) and (\ref{part4}), $\forall n>N_{\epsilon}$ we get:
	\begin{equation*}
		\sum_{j=1}^{d_n \wedge d_n'} \frac{G(\sigma_j(A_n))}{d_n \wedge d_n'}\leq\sum_{j=1}^{d_n \wedge d_n'} \frac{G(\sigma_{j}(C_{n,t}))}{d_n \wedge d_n'}+c(t)\left\|G\right\|_{\infty}+\omega(t)\left\|G'\right\|_{\infty}\leq
	\end{equation*}
	\begin{equation*}
		\leq\sum_{j=1}^{d_{n,t} \wedge d_{n,t}'} \frac{G(\sigma_{j}(B_{n,t}))}{d_{n,t} \wedge d_{n,t}'} + (m(t)+c(t))\left\|G\right\|_{\infty}+\omega(t)\left\|G'\right\|_{\infty}\leq
	\end{equation*}
	\begin{equation*}
		\leq\frac{1}{\mu(D_t)}\int_{D_t} \sum_{k=1}^{p} \frac{G(\sigma_k(f_t(s)))}{p} ds + (1+2\left\|G\right\|_{\infty}+\left\|G'\right\|_{\infty})\epsilon\leq
	\end{equation*}
	\begin{equation*}
		\leq\frac{1}{\mu(D)} \int_{D} \sum_{k=1}^{p} \frac{G(\sigma_k(f(s)))}{p} ds + (2+2\left\|G\right\|_{\infty}+\left\|G'\right\|_{\infty})\epsilon
	\end{equation*}
	Since $\epsilon$ is arbitrary, we get
	\begin{equation*}
		\limsup_{n \to \infty}\sum_{j=1}^{d_n \wedge d_n'} \frac{G(\sigma_j(A_n))}{d_n \wedge d_n'}\leq\frac{1}{\mu(D)} \int_{D} \sum_{k=1}^{p} \frac{G(\sigma_k(f(s)))}{p} ds
	\end{equation*}
	In a similar we can obtain also that
	\begin{equation*}
		\liminf_{n \to \infty}\sum_{j=1}^{d_n \wedge d_n'} \frac{G(\sigma_j(A_n))}{d_n \wedge d_n'}\geq\frac{1}{\mu(D)} \int_{D} \sum_{k=1}^{p} \frac{G(\sigma_k(f(s)))}{p} ds
	\end{equation*}
	we just need to be to be a bit more careful in two of the passages, when, using that G is positive, we added terms in the sum in (\ref{part1}) and used that $d_{n,t}\wedge d_{n,t}'\leq d_n\wedge d_n'$ in (\ref{part2}). Nevertheless, those two additional terms of correction depend on $G$, $c(t)$ and $m(t)$ and can be controlled similarly to those coming from limsup derivations. This ends the proof.
	
\end{proof}

Under the hypothesis that the sequences are Hermitian, we can extend the previous result to the case of eigenvalues (with essentially the same proof), getting the following theorem.

\begin{thm}\label{gacs_eigen}
	Let $\{A_n\}$ be a sequence of Hermitian matrices of size $ d_n$ with $ d_n $ a monotonically increasing integer sequence. Let $\left\{\{B_{n,t}\}_{n}\right\}_{t}$ be a g.a.c.s. for $\{A_n\}$ (in the Hermitian sense). If, $\forall t$, $\exists\left(f_t,D_t\right)$ such that: \\
	\begin{itemize}
		\item $\left\{\{B_{n,t}\}_{n}\right\}_{t} \sim_{\lambda} (f_t,D_t)$,\\
		\item $D_t\subset D_{t+1}$, $\forall t$,\\
		\item $D := \bigcup\limits_{t>0}D_t$ of finite measure,\\
		\item $\exists f:D\to\mathbb{C}^{p \times p}$ measurable such that $f^{E}_t\to f$ in measure, $t\to +\infty$, with
	\end{itemize}
	\begin{equation*}
	f_t^E= \begin{cases}
	f_t(s) & \text{if } s \in D_t \\
	0 & \text{if } s \in D\setminus D_t,
	\end{cases}
	\end{equation*}
	then $\{A_n\} \sim_{\lambda} (f,D)$.
\end{thm}

\begin{rem}
In the case of domains of infinite measure we expect that a similar extension as the one obtained for a.c.s. can be obtained. We do not go into detail at the moment as we want to focus mainly on unbounded domains with finite measure, leaving this possible extension for future developments.
At any rate, as Theorem \ref{main_sv} and Theorem \ref{main_eig} show, the infinite measure setting can be handled: however, we still need to make a fine tuning of the new concepts in order to work in the best possible way.
\end{rem}

\section{Applications, examples, numerical evidences}\label{sec:applied-numerical}

In this section we show a few basic numerical examples for proving the validity of the theory developed in the previous section. We start with a kind of trivial one in a bounded setting and then explore a more involved one where the domain is unbounded with finite measure and hence, of course, of non-Cartesian type: in this case we consider both $P_1$ and $P_2$ {Finite Elements} and as operator both the constant coefficient and the variable coefficient Laplacian.

\subsection{Approximating by Finite Elements over an exhaustion of (0,1)}

We consider the following model problem

\begin{equation*}
\begin{cases}
-\Delta u = f & \text{ on } (0,1)\\
u(0)=u(1)=0.
\end{cases}
\end{equation*}

It is probably the most studied test example and it is already known that its $P_1$ Finite Elements approximation \cite{Cia} on a uniform grid with a proper scaling leads to the solution of a linear system whose coefficient matrix is $A_n=T_n(2-2\cos(\theta))$. As widely known related sequence $ \{ A_n \}_n $ has a spectral distribution given by the function $ 2-2\cos(\theta)$ in $ [-\pi,\pi] $: it is indeed a special case of Theorem \ref{thm:symbol_d-block-toeplitz} since its symbol is real-valued. We can see this also by applying linear Finite Elements to the problem

\begin{equation*}
\begin{cases}
-\Delta u = f & \text{ on } (0,1-1/t)\\
u(0)=u(1-1/t)=0.
\end{cases}
\end{equation*}

In this case the spectral distribution of the associated matrix sequence $ \{B_{n,t}\} $ is given by the symbol function $ (1-1/t)(2-2\cos(\theta)) $ for every $ t > 0 $. Clearly, all the assumption of Theorem \ref{main_eig} are satisfied, hence the sequence $ \{ A_n\}_n $ has spectral distribution with symbol given by $ \lim_{t \to \infty} (1-1/t)(2-2\cos(\theta)) = 2 - 2 \cos(\theta) $.

\begin{rem}
We note that one of the main aspects that allows us to infer the spectral distribution of the sequence $ \{ A_n \}_n $ is the ``convergence" of the points in the discretization  defining $ \{\{ B_{n,t} \}\} $ (and of the related domain) to those defining $ \{ A_n \} $ (and of the limit domain) as $ t \to \infty $. In fact, this is what makes $ \{\{ B_{n,t} \}\} $ an a.c.s. for $ \{ A_n\} $.
\end{rem}

\subsection{Quadratic Finite Elements for the model problem on a two-dimensional unbounded set of finite measure}

We consider the following problem

\begin{equation}\label{prob_unbounded}
\begin{cases}
-\Delta u = v & \text{ on } \Omega\\
u=0 & \text{ on } \partial \Omega,
\end{cases}
\end{equation}
where $ \Omega = \{ (x,y) \in \mathbb{R} \, : \, x>0, y > 0 \text{ and } y < g(x) \} $ with
\begin{equation*}
g(x)= \begin{cases}
1 & x < 1\\
\frac{1}{x^2} & x \geq 1.
\end{cases}
\end{equation*}
It is clear that $ \mu(\Omega) = \int_{0}^{\infty} g(x) dx=2 < \infty $.  For the discretization, we consider as space-step $ h=\frac{1}{n+1} $ and the nodes $ (x_i,y_j)=(ih,jh) $. We consider the maximum value of $ i $,  $ \bar{i} $, for which there exists $ j $ such that $ (x_{\bar{i}},y_j) \in \Omega $. We consider all the points in the rectangle $ (0,\bar{i}) \times (0,1) $. Note that the index $ \bar{i} = \lfloor \sqrt{n+1} \rfloor $.\\
On these points we apply quadratic $P_2$ Finite Elements which leads to a two-level Toeplitz matrix $ T_{\mathbf{n}}(f) $ with $ \mathbf{n}=(n,n \lfloor \sqrt{n+1} \rfloor) $ and $f=f_{{P}_2}: [-\pi,\pi]^2 \longrightarrow \mathbb{C}^{4\times 4}$ (see \cite{P_k-elements}) with
\begin{equation}\label{eq:P2symbol}
f=f_{{P}_2}(\theta_{1},\theta_{2})  =  \left [
\begin{array}{rr|rr}
\alpha                 & -\beta(1+e^{\i \theta_{1}})  & -\beta(1+e^{\i \theta_{2}})  & 0 \\
-\beta(1+e^{-\i \theta_{1}}) & \alpha                 & 0                      & -\beta(1+e^{\i \theta_{2}})\\
\hline
-\beta(1+e^{-\i \theta_{2}}) & 0                     & \alpha                 & -\beta(1+e^{\i \theta_{1}})\\
0                      & -\beta(1+e^{-\i \theta_{2}}) & -\beta(1+e^{-\i \theta_{1}}) & \gamma +\frac{\beta}{2} (\cos(\theta_{1})+\cos(\theta_{2}))\\
\end{array}
\right ]
\end{equation}
where $\alpha={16}/{3}$, $\beta={4}/{3}$, and $\gamma=4$.\\

 Clearly, this sequence has spectral distribution given by $ f $ on the domain $ [-\pi, \pi]^2 $ by Theorem \ref{thm:symbol_d-block-toeplitz} since its symbol is Hermitian-valued: indeed given the even nature of the symbol the domain can be reduce to the subdomain $ [0, \pi]^2 $.
 We now consider the points $ (x_i, y_j) \in \Omega $ and the points $ (x_i,y_j) \in \Omega_t= \Omega \cap B(0,t)^{\|\cdot\|_{\infty}} $.
 %Clearly, since $ g(x) \to 0 $ as $ x \to \infty $,  for a fixed space-step, we have the points $(x_i,y_j) $ are all in $ \Omega_{t}= \Omega \cap B(0,t)^{\|\cdot\|_{\infty}} $ for some great enough $ t $.
 As $ h \to 0 $ we obtain uniform discretizations both for the set $ \Omega $ and for the sets $ \Omega_t $, which constitute an exhaustion of the domain $ \Omega $. Finally, we can obtain a discrete version $ A_{\mathbf{n}} $ of the Laplacian on $ \Omega $ by simply cutting from the matrix $ T_{\mathbf{n}}(f) $ all the rows and columns corresponding to indices $ (i,j) $ such that $ (x_i,y_j) $ are not in $ \Omega $. We can do the same on $ \Omega_t $ obtaining another sequence of matrices $ \{\{ C_{\mathbf{n},t} \} \} $. In general $ \dim(A_{\mathbf{n}}) \neq \dim(C_{{\mathbf{n}},t}) $. We then consider the sequence $\{ \{ B_{\mathbf{n},t} \}\} $, where $ B_{\mathbf{n},t}=C_{\mathbf{n},t} \oplus 0_{m}$ and $ m $ is such that $ \dim(B_{\mathbf{n},t})=\dim(A_\mathbf{n}) $. We want to prove that
\begin{itemize}
\item[a)] there exists $ f_t $ such that $ \{ B_{\mathbf{n},t} \} \sim_{\lambda, \sigma} (f_t, \Omega \times [-\pi,\pi]^2) $.
\item[b)] $ \{ B_{\mathbf{n},t} \} $ is an a.c.s. for a matrix-sequence $ \{ \tilde{A}_{\bfnn} \} $ where $ \tilde{A}_{\bfnn} $ is similar to $ A_\mathbf{n} $ for every $ \bfnn $ using a unique permutation matrix.
\item[c)] $\exists \, \lim_{t \to \infty} f_t$ in $ \Omega \times [-\pi,\pi]^2 $.
\end{itemize}
Using Theorem \ref{main_sv} and Theorem \ref{main_eig}, we can conclude that $ \{ A_\mathbf{n} \} \sim_{\lambda, \sigma} (f,\Omega \times [-\pi,\pi]^2) $.\\
For item a) we can simply use the theory of reduced GLT in \cite{Barb}, obtaining that $ \{B_{\mathbf{n},t}\} \sim_{\lambda,\sigma} (f_t,\Omega\times [-\pi,\pi]^2)$, with $ f_t(\mathbf{x},\theta_1,\theta_2)= \chi_{\Omega_t}(\mathbf{x}) \cdot f $, with $f$ depending on the used approximation of the considered PDE and $\mathbf{x}=(x,y)$.\\
For item b) we note that $ \tilde{A}_\mathbf{n} = B_{\mathbf{n},t} + R_{\mathbf{n},t} $ with $ \text{rank} (R_{\mathbf{n},t}) \leq (2-\mu(\Omega_t))n^2 + c(t) n $, where $ c(t) \to 0 $ as $ t \to \infty $. Since $ \mu(\Omega_t) \to 2 $ as $ t \to \infty $, we find that $ \{\{ B_{\mathbf{n},t} \}\} $ is an a.c.s. for $ \{ \tilde{A}_\mathbf{n} \} $.\\
For item c) we have
\begin{equation*}
\lim_{t \to \infty} f_t = f
\end{equation*}
for every $ (\mathbf{x},\theta_1,\theta_2) \in  \Omega \times [-\pi,\pi]^2$.
Using Theorem \ref{main_sv} and Theorem \ref{main_eig}, recalling Definition \ref{def:distributions}, we conclude that the sequence $\{ A_n \}$ has a distribution both in the eigenvalue sense and in the spectral value sense given by $f$ on $\Omega \times [-\pi,\pi]^2$ or equivalently on
$\Omega \times [0,\pi]^2$.

\begin{rem}\label{rem:actors}
Equivalently, we can choose as $ B_{\bfnn,t} $ the matrix defined by $ D_{\bfnn}(\chi_{\Omega_t}) A_\bfnn D_{\bfnn}(\chi_{\Omega_t}) $ and it is possible to see that  the rank of $A_{\bfnn}-B_{\bfnn,t}$ satisfies the same upper estimate as above and $\{B_{\bfnn,t}\}$ is still distributed as $(f_t,\Omega\times [-\pi,\pi]^2)$. Furthermore the set $[-\pi,\pi]^2$ can be replaced by $[0,\pi]^2$ in all cases in which the function is even in
$\theta_1$, $\theta_2$, separately. As already observed in \cite{glt-laa}[pp. 375-378], it is worth recalling that the obtained spectral symbol depends on three actors: the operator order, the coefficients/physical domain, the approximation technique. For instance in our setting the underlying operator has order two and this can be read in the minimal eigenvalue of the symbol which is a asymptotic to $2-2\cos(\theta)$ in one variable and to $4-2\cos(\theta_1)-2\cos(\theta_2)$ in two variables. In both cases the order of the zero is two and this decides the conditioning of the resulting matrices which will grow as $N^{2\over d}$ with $N$ being the matrix-size and $d$ the dimensionality of the domain $\Omega$.
However, there are also other features that can be recovered. Here we mention three of them:
\begin{itemize}
\item the fact that the zero is at zero informs that the subspace related to low frequencies is the one associated with the small eigenvalues,
\item an unbounded diffusion coefficient $a(x,y)$ e.g. with a unique pole at $(\hat x,\hat y)$ of order $\gamma$ will be seen in the maximal eigenvalues exploding as $N^{\gamma\over d}$ and the related subspaces that can be easily identified as a function of the point $(\hat x,\hat y)$: in this case the overall conditioning will be asymptotic to $N^{(2+\gamma)\over d}$.
\item The approximation technique plays a role in the structure on the underlaying matrices and hence in the complexity of the associated matrix-vector product: the more the method is precise the larger is the bandwidth in a multilevel sense, while the number of levels is decided by the dimensionality of the physical domain $\Omega$, and the presence of blocks and their size is exactly the gap between the degree of the polynomials used in the Finite Elements and the global continuity which is imposed (see also the discussion at the end of the introduction).
\end{itemize}
\end{rem}

The remark below concerns the approximation by $Q_2$ Finite Elements, i.e., using rectangles as basic geometric elements instead of triangles. Notice that the curve $y={1 \over x^2}$ for $x\ge 1$ defining the domain is better approximated using triangles instead of rectangles that would give an unpleasant staircase.

\begin{rem}\label{rem:Q2}
Consider the quadratic Finite Elements approximation of the considered one-dimensional model problem. The resulting stiffness matrix is essentially of block Toeplitz type with blocks
$$K_0=\frac13\left[\begin{array}{rr}16 & -8\\ -8 & 14\end{array}\right],\qquad K_1=\frac13\left[\begin{array}{rr}0 & -8\\ 0 & 1\end{array}\right].$$
Similarly, according to \cite{branches-FEM}, the corresponding blocks of the mass matrix are
$$M_0=\frac1{15}\left[\begin{array}{rr}8 & 1\\ 1 & 4\end{array}\right],\qquad M_1=\frac1{30}\left[\begin{array}{rr}0 & 2\\ 0 & -1\end{array}\right].$$
Furthermore, the symbol of the sequence of the stiffness matrices is the function
$$\bff_2(\theta)=\frac13\left[\begin{array}{cc}16 & -8-8\e^{\i\theta}\\ -8-8\e^{-\i\theta} & 14+2\cos\theta\end{array}\right],\quad\theta\in[-\pi,\pi],$$
while the symbol of the sequence of the mass matrices is the function
$$\bfh_2(\theta)=\frac1{15}\left[\begin{array}{cc}8 & 1+\e^{\i\theta}\\ 1+\e^{-\i\theta} & 4-\cos\theta\end{array}\right],\quad\theta\in[-\pi,\pi],$$
again in accordance to \cite{branches-FEM}.
The eigenvalues of $\bff_2(\theta)$,
%,shown in Figure \ref{A.f2,f3},
are
\begin{align*}
\lambda_1(\bff_2(\theta))&=5+\frac13\cos\theta+\frac13\sqrt{129+126\cos\theta+\cos^2\theta},\\[5pt]
\lambda_2(\bff_2(\theta))&=5+\frac13\cos\theta-\frac13\sqrt{129+126\cos\theta+\cos^2\theta}=\frac{16}3\,\frac{2-2\cos\theta}{\lambda_1(\bff_2(\theta))}.
\end{align*}
 Similarly, the eigenvalues of $\bfh_2(\theta)$ are
\begin{align*}
\lambda_1(\bfh_2(\theta))&=\frac25-\frac1{30}\cos\theta+\frac1{30}\sqrt{24+16\cos\theta+\cos^2\theta},\\[5pt]
\lambda_2(\bfh_2(\theta))&=\frac25-\frac1{30}\cos\theta-\frac1{30}\sqrt{24+16\cos\theta+\cos^2\theta}.
\end{align*}
Since the eigenvalue functions $\lambda_i(\bff_2(\theta)),\ i=1,2,$ are even, it follows from Definition \ref{def:distributions} that $\bff_2(\theta)$ restricted to $[0,\pi]$ is still a symbol for $\{K_n^{(2)}\}_n$. As already recalled in the Introduction, the latter is equivalent to write that a suitable ordering of the eigenvalues $\lambda_j(K_n^{(2)}),\ j=1,\ldots,2n-1$, assigned in correspondence with an equispaced grid on $[0,\pi]$, reconstructs approximately the graphs of the eigenvalue functions $\lambda_i(\bff_2(\theta)),\ i=1,2$.
\begin{figure}
\centering
\includegraphics[scale=0.35]{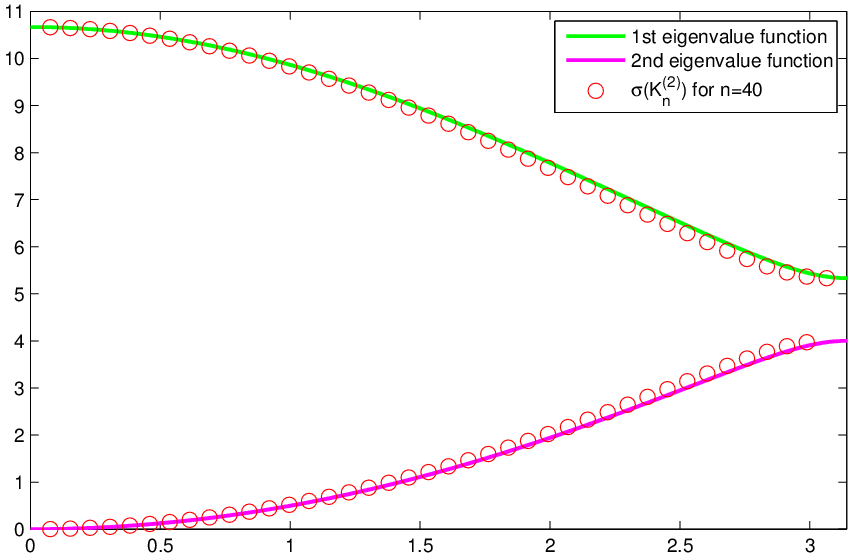}
\caption{Graph over $[0,\pi]$ of the eigenvalue functions $\lambda_i(\bff_2(\theta)),\ i=1,2,$ and of the pairs $\left(\frac{k\pi}{n+1},\lambda_k(K_n^{(2)})\right),\ k=1,\ldots,n,$ and $\left(\frac{(2n-k)\pi}{n+1},\lambda_k(K_n^{(2)}) \right),\ k=n+1,\ldots,2n-1,$ for $n=40$.}
\label{eig_funs_eigs_stiff}
\end{figure}
Setting $K_n^{(2)}$ the stiffness matrix in one dimension, this is observed in Figure \ref{eig_funs_eigs_stiff}, where we fixed the equispaced grid $\frac{k\pi}{n+1},\ k=1,\ldots,n,$ in $[0,\pi]$ and we plotted the eigenvalue functions $\lambda_i(\bff_2(\theta)),\ i=1,2,$ as well as the pairs $\left(\frac{k\pi}{n+1},\lambda_k(K_n^{(2)})\right),\ k=1,\ldots,n,$ and $\left(\frac{(2n-k)\pi}{n+1},\lambda_k(K_n^{(2)})\right),\ k=n+1,\ldots,2n-1,$ for $n=40$. We clearly see from Figure~\ref{eig_funs_eigs_stiff} that the eigenvalues of $K_n^{(2)}$ can be split into two subsets (or branches) of approximately the same cardinality; and the $i$-th branch is approximately given by a uniform sampling over $[0,\pi]$ of the $i$-th eigenvalue function $\lambda_i(\bff_2(\bftheta))$, $i=1,2$.

The spectral symbol in the two-dimensional case is then
\[
f=f_{{Q}_2}(\theta_1,\theta_2)= \bff_2(\theta_1)\bfh_2(\theta_2)+\bfh_2(\theta_1)\bff_2(\theta_2)
 \]
 according to formula (5.1) in \cite{branches-FEM} and with $\bff_2, \bfh_2$ as before. Interestingly enough the dimensionality of the two symbols
$f_{{P}_2},f_{{Q}_2}$ is the same: two variables and matrix-size equal to four. Furthermore in both cases three eigenvalues are structly positive and only one of them is asymptotic to the symbol of the standard discrete Laplacian by Finite Differences or $P_1$ Finite Elements, i.e., $4-2\cos(\theta_1)-2\cos(\theta_2)$: compare \cite{P_k-elements} and \cite{branches-FEM}.

Indeed, in the case of a general PDE with variable coefficients in a $d$-dimensional domains with standard Finite Elements the number of variables is $2d$ which reduces to $d$ if the coefficients of the differential operator are constants and the dimensionality is $k^d$ where $k$ is the degree of the used polynomials in the Finite Elements. For diminishing the presence of an exponential number of branches, as discussed in the introduction, the solution is the Isogeometric analysis \cite{branches-IgA-k-l}.
%
%\begin{eqnarray}
%\bff_2(\theta) & = & \frac13\left[\begin{array}{cc}16 & -8-8\e^{\i\theta}\\ -8-8\e^{-\i\theta} & 14+2\cos\theta\end{array}\right],\quad\theta\in
%[-\pi,\pi], \\
%\bfh_2(\theta) & = &\frac1{15}\left[\begin{array}{cc}8 & 1+\e^{\i\theta}\\ 1+\e^{-\i\theta} & 4-\cos\theta\end{array}\right],\quad\theta\in
%[-\pi,\pi].
%\end{eqnarray}
\end{rem}

Finally we consider a variable coefficient version of the previous PDE in (\ref{prob_unbounded}) expressed as follows

\begin{equation}\label{prob_unbounded_var}
\begin{cases}
\text{div}(-a\nabla u) = v & \text{ on } \Omega\\
u=0 & \text{ on } \partial \Omega,
\end{cases}
\end{equation}
where $a(x,y)$ is a positive non-degenerate variable coefficient on $ \Omega = \{ (x,y) \in \mathbb{R} \, : \, x>0, y > 0 \text{ and } y < g(x) \} $ with
\begin{equation*}
g(x)= \begin{cases}
1 & x < 1\\
\frac{1}{x^2} & x \geq 1.
\end{cases}
\end{equation*}
Notice that for $a\equiv 1$ problem (\ref{prob_unbounded_var}) reduces to (\ref{prob_unbounded}).
As in the constant coefficient case we opt for basic $P_1$ Finite Elements. According to the theory reported in \cite{branches-FEM,P_k-elements} the symbol is
\[
f(\theta_1,\theta_2,x,y)=(4-2\cos \theta_1-2\cos \theta_2)a(x,y),
\]
with four variables, $(x,y)$ in the physical domain, $(\theta_1,\theta_2)$ in the Fourier domain, and dimensionality $1$ since the degree of the polynomials used in the finite element is $1$.
\subsection{Numerical evidences}

In this subsection we show numerical tests and visualizations corroborating the analysis conducted in the previous subsection. We focus on essentially few discretization techniques for the approximation of \eqref{prob_unbounded} and \eqref{prob_unbounded_var}, namely standard centered Finite Differences of order two, $P_1$ Finite Elements,
$Q_1$ Finite Elements, and $P_2$ Finite Elements.

%\begin{rem}\label{rem:minimal eig - etc}
With reference to Remark \ref{rem:actors}, we observed that the conditioning of the resulting matrices grows as $N^{2\over d}$ with $N$ being the matrix-size and $d$ being the dimensionality. This is due to the minimal eigenvalue which converges to zero as $N^{-{2\over d}}$
(see \cite{SerraBIT96,SerraLAA98,BGLAA98} for the pure Toeplitz setting and \cite{minimal eig var 1,minimal eig var 2} for the variable coefficient case) and the related property can be seen in Figures \ref{FD_A}, \ref{Q1_A}, \ref{P2_A} where the univariate unique nondecreasing rearrangement of the symbol for $d=2$ has a positive bounded derivative so that the minimal eigenvalue tends to zero as $N^{-1}$. Things do not change, as expected, in the variable coefficient setting since the diffusion coefficient $a(x,y)$ is positive and bounded: see Figure \ref{FDcoeff_A}.
For the rest there is nothing much to comment given the very strong agreement of the spectral behavior of the global matrix-sequences and of the corresponding approximations $ \{\{ B_{\mathbf{n},t} \}\} $,  $ \{\{ C_{\mathbf{n},t} \}\} $ as $t$ grows. The very striking fact is that the convergence is appreciated already for moderate values of $t$ giving an evidence of the practical use of the used tools i.e. the notion of a.c.s and that, which is indeed more natural, of the generalized a.c.s.: for the details on the approximating matrices refer to the figures with even numbering from 4 to 16, from 20 to 32, from 36 to 48, from 52 to 64, and compare with Figures \ref{FD_A}, \ref{Q1_A}, \ref{P2_A}, \ref{FDcoeff_A} regarding the various matrix-sequences $ \{A_{\mathbf{n}} \} $, respectively. Furthermore, in the figures with odd numbering from 3 to 65, the errors in the eigenvalue predictions are reported: the pseudo-random behavior can be attributed to the non-decreasing (univariate) rearrangement of the spectral symbol of $ \{A_{\mathbf{n}} \} $. When maintaining the complete number of variables a much smoother surface is expected. 
In addition, when looking at the figures related to the $P_2$ approximation we observe 4 points where the rearranged symbol 
is not differentiable and this corresponds to the 4 branches of the spectra since the symbol is $4\times 4$ Hermitian-valued. Finally, 
when looking at the figures with even numbering from 50 to 64, we observe smoother curves and wider ranges and this is due 
to the variable coefficient $a(x,y)$, since any of the four eigenvalue functions in the constant coefficient case 
is multiplied by $a(x,y)$.
%\end{rem}

%///////////////////////////////////////////////////////////
%///////////////////////////////////////////////////////////

\subsubsection{Finite Differences/$P_1$ - $f(\theta_1,\theta_2)=(2-2\cos \theta_1)+(2-2\cos \theta_2)$}
Finite Differences/$P_1$: $f(\theta_1,\theta_2)=(2-2\cos \theta_1)+(2-2\cos \theta_2)$ \\

%-------------------------------------------------------------- %
%-------------------------------------------------------------- %
\begin{figure}
\centering
  \includegraphics[width=\textwidth]{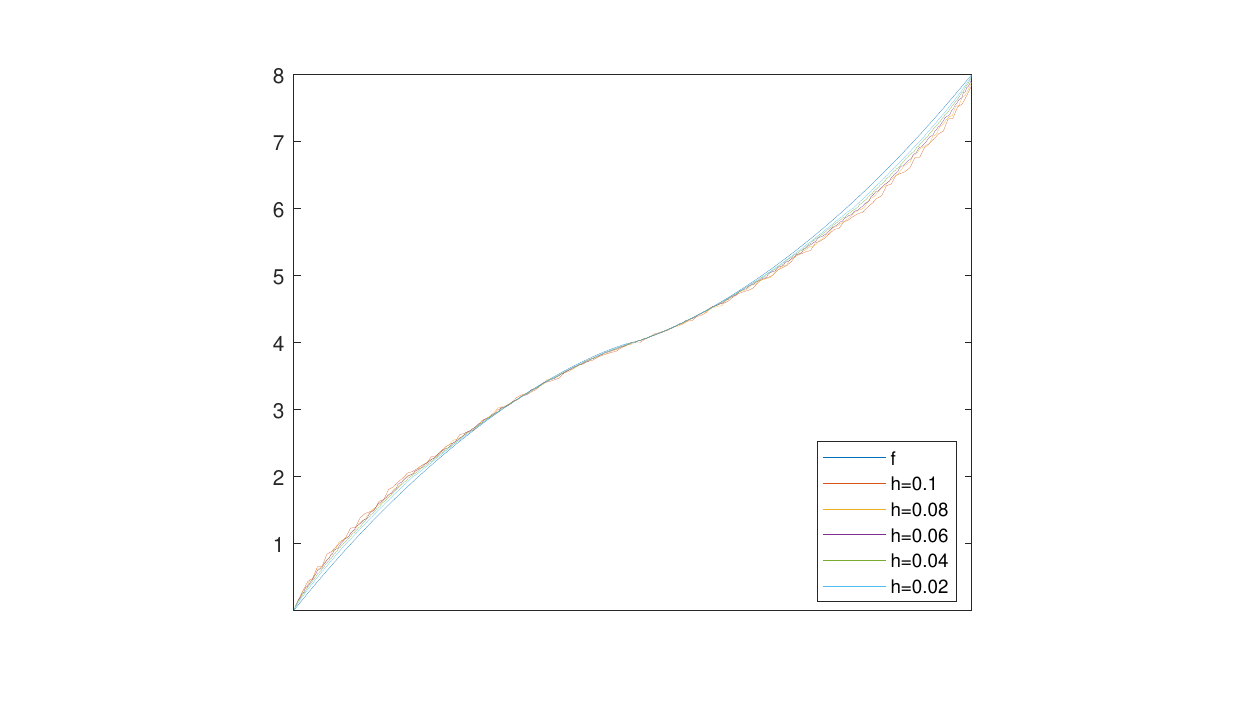} \vskip -0.5cm
  \caption{Eigenvalues distribution of $A_n$ for different $h$ values  together with the sampling of $f(\theta_1,\theta_2)=(2-2\cos \theta_1)+(2-2\cos \theta_2)$.}
\label{FD_A}
\end{figure}
\begin{figure}
\centering
  \includegraphics[width=\textwidth]{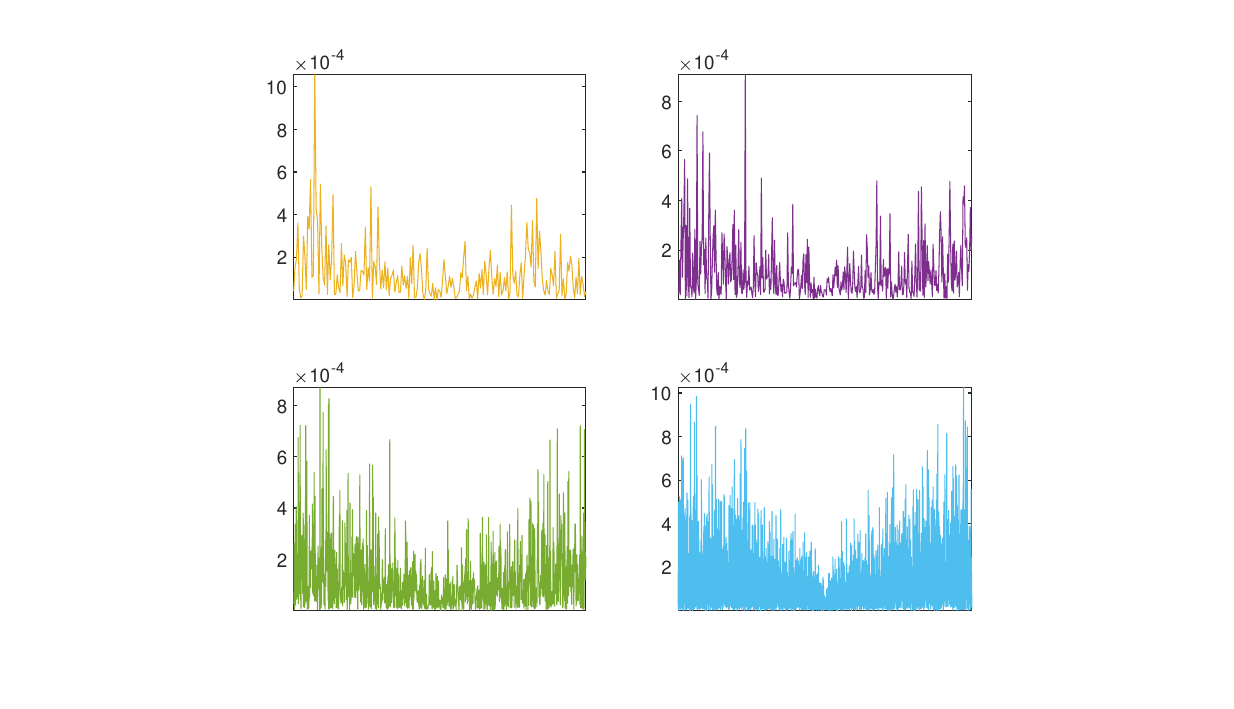} \vskip -0.5cm
  \caption{Minimal distance of eigenvalues of $A_n$ from $f(\theta_1,\theta_2)=(2-2\cos \theta_1)+(2-2\cos \theta_2)$ for different $h$ values.}
\label{FD_A_errore}
\end{figure}
%
%-------------------------------------------------------------- %
%-------------------------------------------------------------- %
\begin{figure}
\centering
  \includegraphics[width=\textwidth]{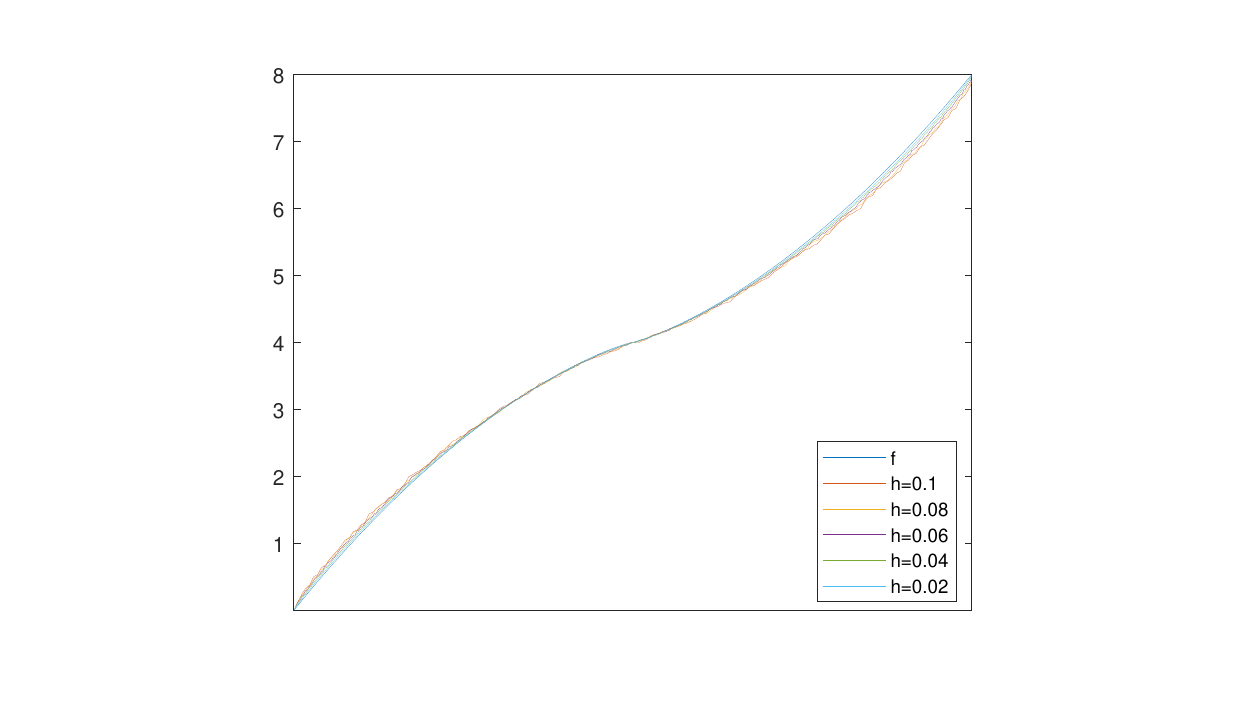} \vskip -0.5cm
  \caption{Eigenvalues distribution of $C_{n,t}$ for different $h$ values  together with the sampling of $f(\theta_1,\theta_2)=(2-2\cos \theta_1)+(2-2\cos \theta_2)$ and $t=2$.}
\label{FD_Bnt_small_t2}
\end{figure}
\begin{figure}
\centering
  \includegraphics[width=\textwidth]{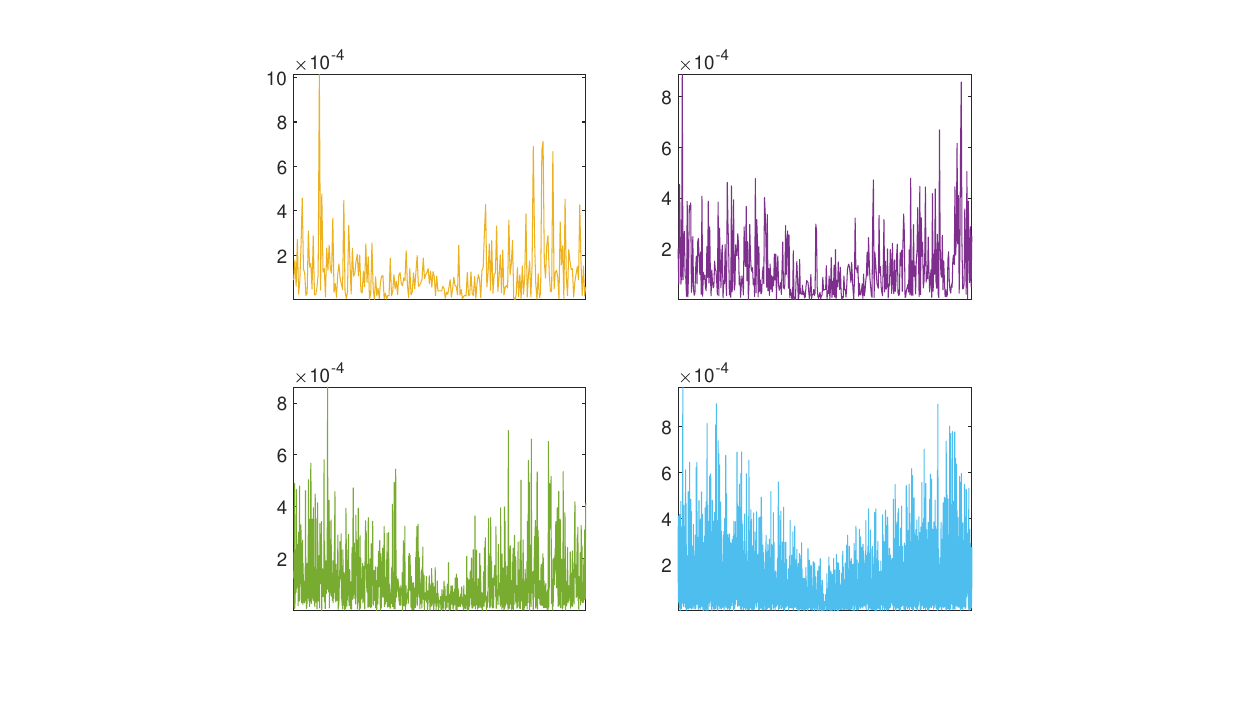} \vskip -0.5cm
  \caption{Minimal distance of eigenvalues of $C_{n,t}$ from $f(\theta_1,\theta_2)=(2-2\cos \theta_1)+(2-2\cos \theta_2)$ and $t=2$ for different $h$ values.}
\label{FD_Bnt_small_t2_errore}
\end{figure}
%--------------------------------------------------------------\begin{figure}
\begin{figure}
\centering
  \includegraphics[width=\textwidth]{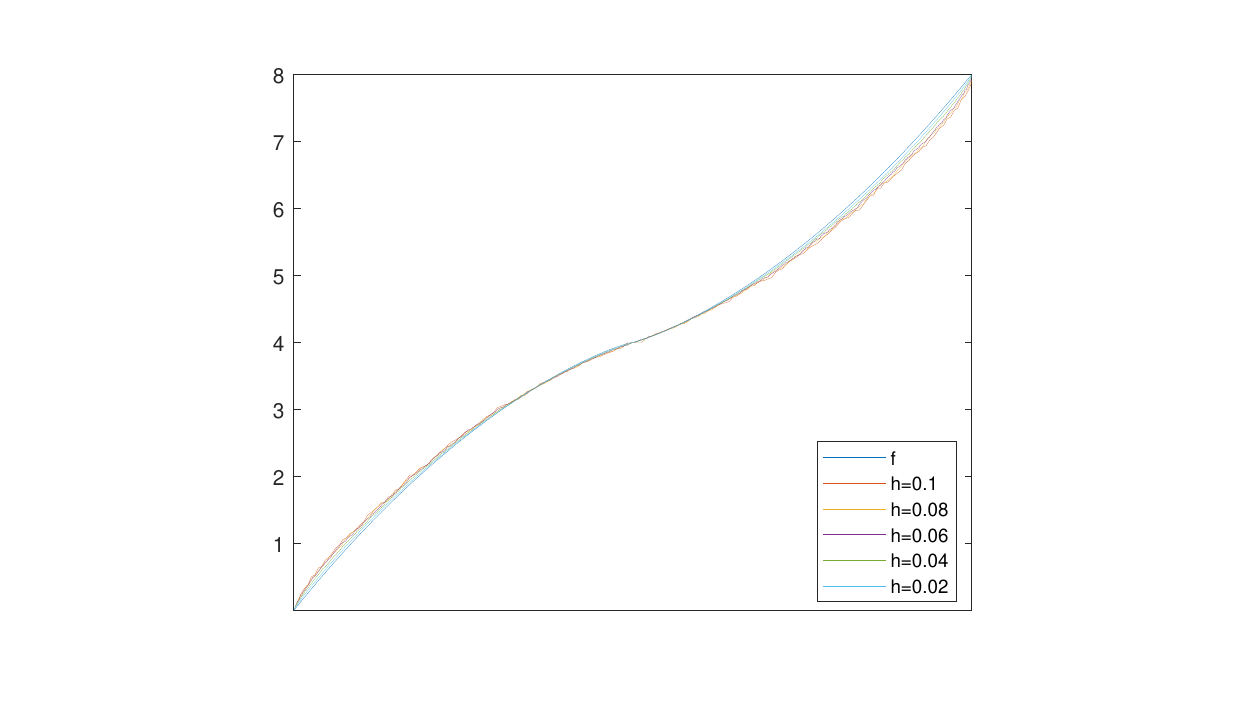} \vskip -0.5cm
  \caption{Eigenvalues distribution of $C_{n,t}$ for different $h$ values  together with the sampling of $f(\theta_1,\theta_2)=(2-2\cos \theta_1)+(2-2\cos \theta_2)$ and $t=4$.}
\label{FD_Bnt_small_t4}
\end{figure}
\begin{figure}
\centering
  \includegraphics[width=\textwidth]{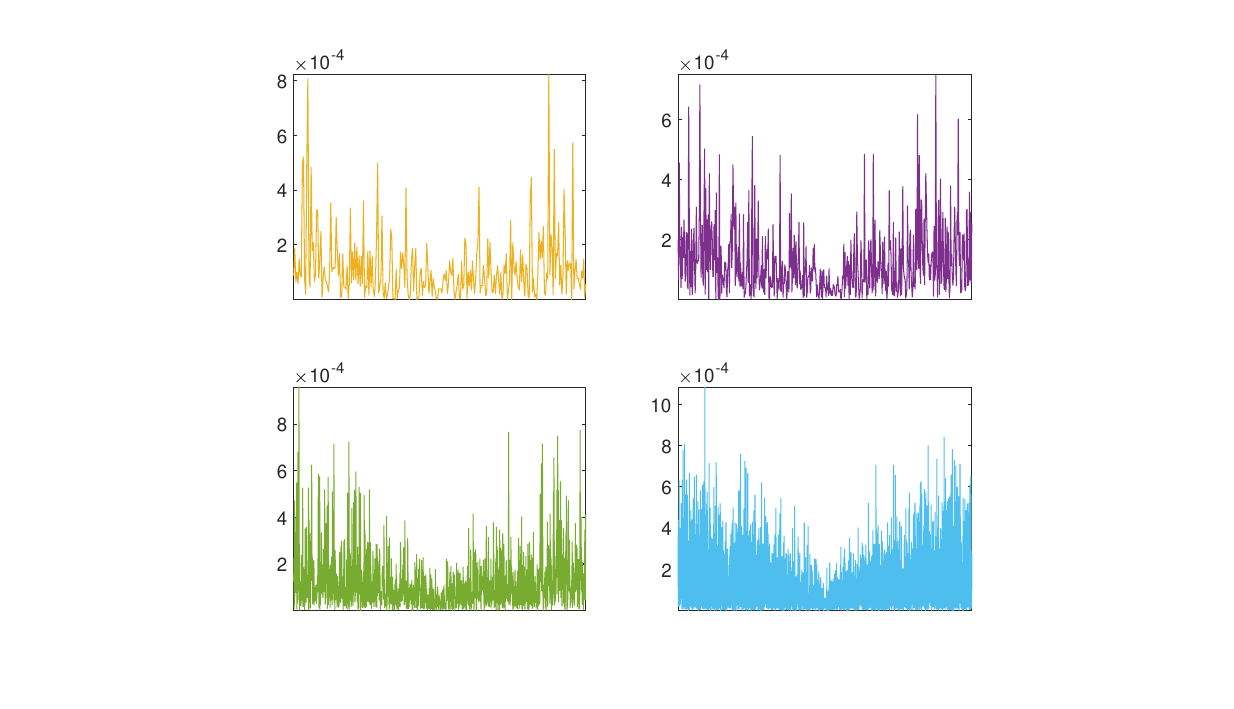} \vskip -0.5cm
  \caption{Minimal distance of eigenvalues of $C_{n,t}$ from $f(\theta_1,\theta_2)=(2-2\cos \theta_1)+(2-2\cos \theta_2)$ and $t=4$ for different $h$ values.}
\label{FD_Bnt_small_t4_errore}
\end{figure}
%--------------------------------------------------------------
\begin{figure}
\centering
  \includegraphics[width=\textwidth]{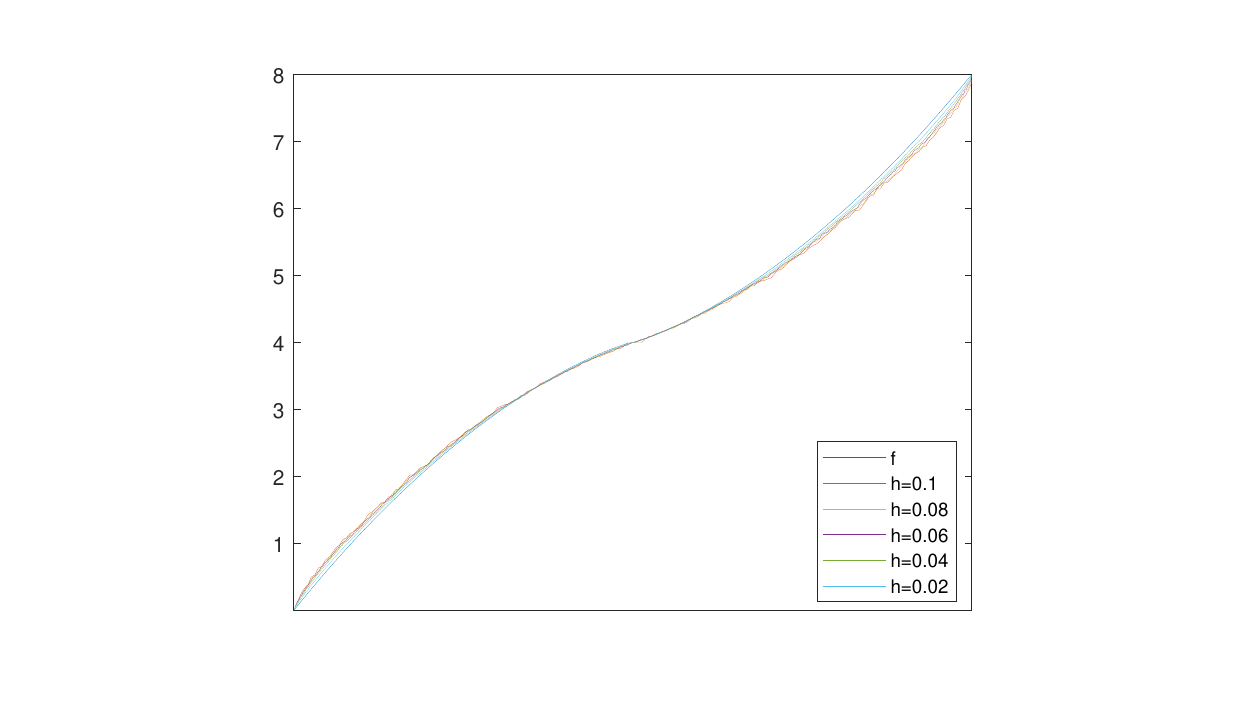} \vskip -0.5cm
  \caption{Eigenvalues distribution of $C_{n,t}$ for different $h$ values  together with the sampling of $f(\theta_1,\theta_2)=(2-2\cos \theta_1)+(2-2\cos \theta_2)$ and $t=6$.}
\label{FD_Bnt_small_t6}
\end{figure}
\begin{figure}
\centering
  \includegraphics[width=\textwidth]{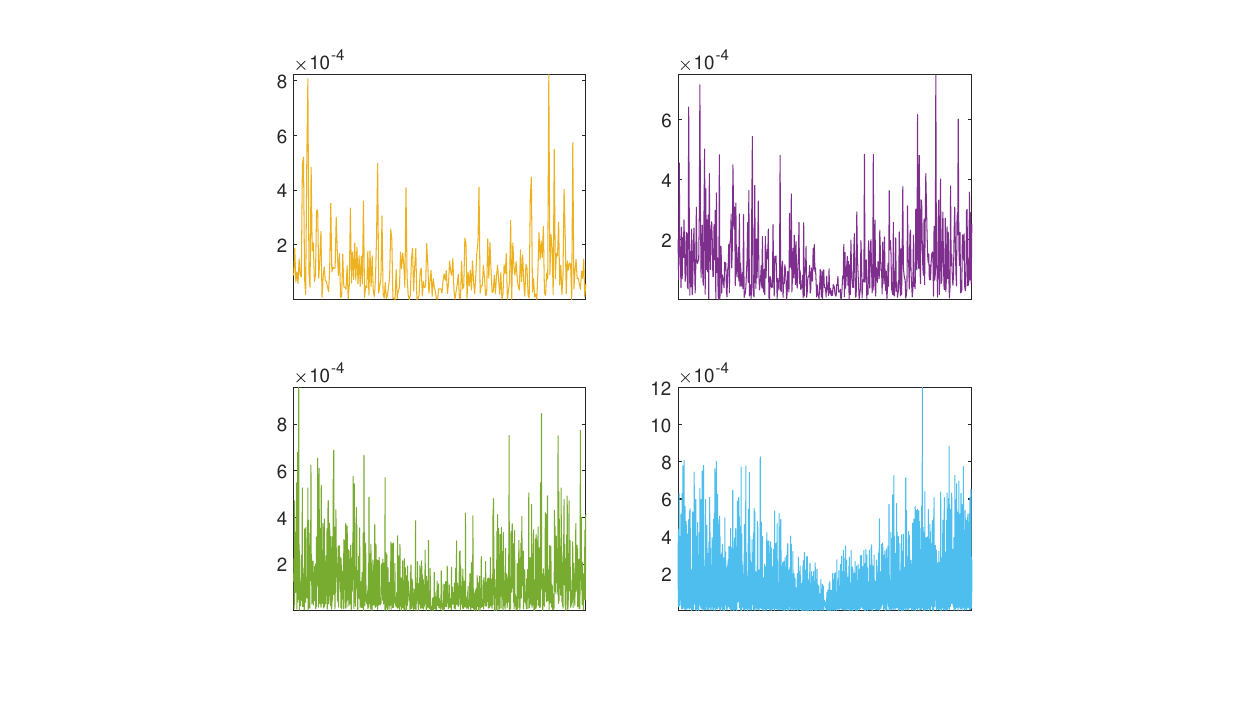} \vskip -0.5cm
  \caption{Minimal distance of eigenvalues of $C_{n,t}$ from $f(\theta_1,\theta_2)=(2-2\cos \theta_1)+(2-2\cos \theta_2)$ and $t=6$ for different $h$ values.}
\label{FD_Bnt_small_t6_errore}
\end{figure}
%--------------------------------------------------------------
%-------------------------------------------------------------- %
\begin{figure}
\centering
  \includegraphics[width=\textwidth]{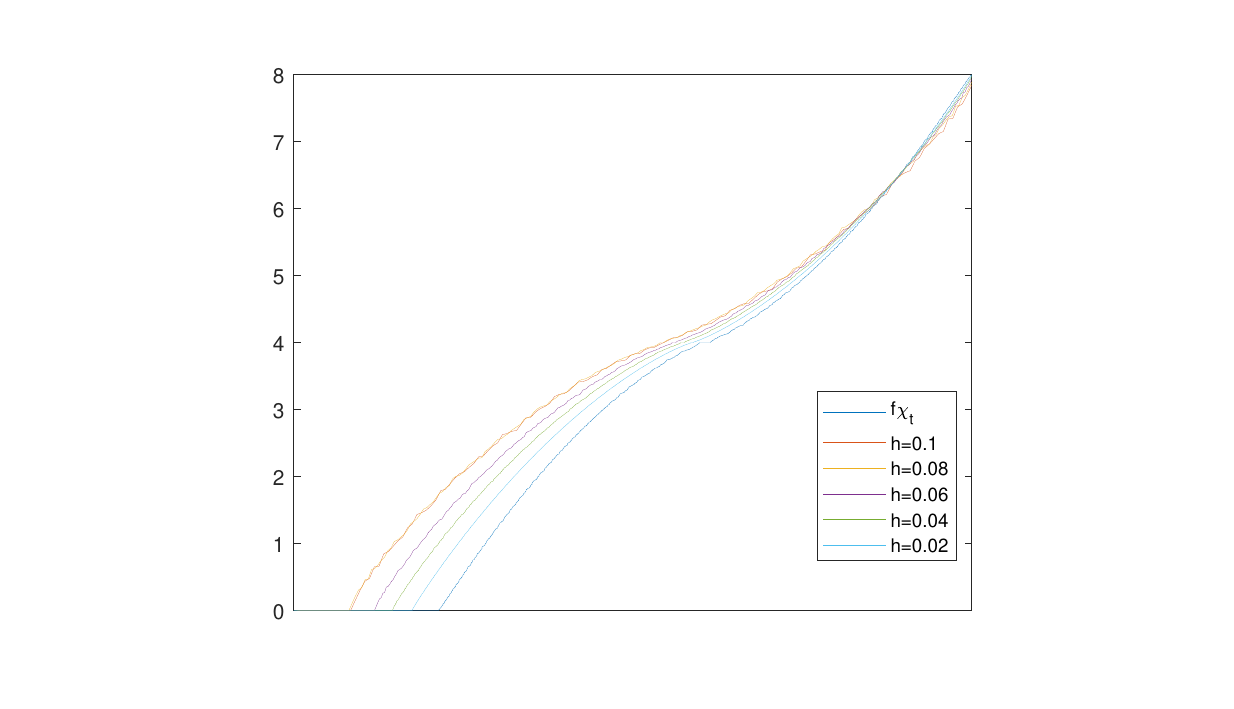} \vskip -0.5cm
  \caption{Eigenvalues distribution of $B_{n,t}$ for different $h$ values  together with the sampling of $f(\theta_1,\theta_2)=(2-2\cos \theta_1)+(2-2\cos \theta_2)$ and $t=2$.}
\label{FD_Bnt_full_t2}
\end{figure}
\begin{figure}
\centering
  \includegraphics[width=\textwidth]{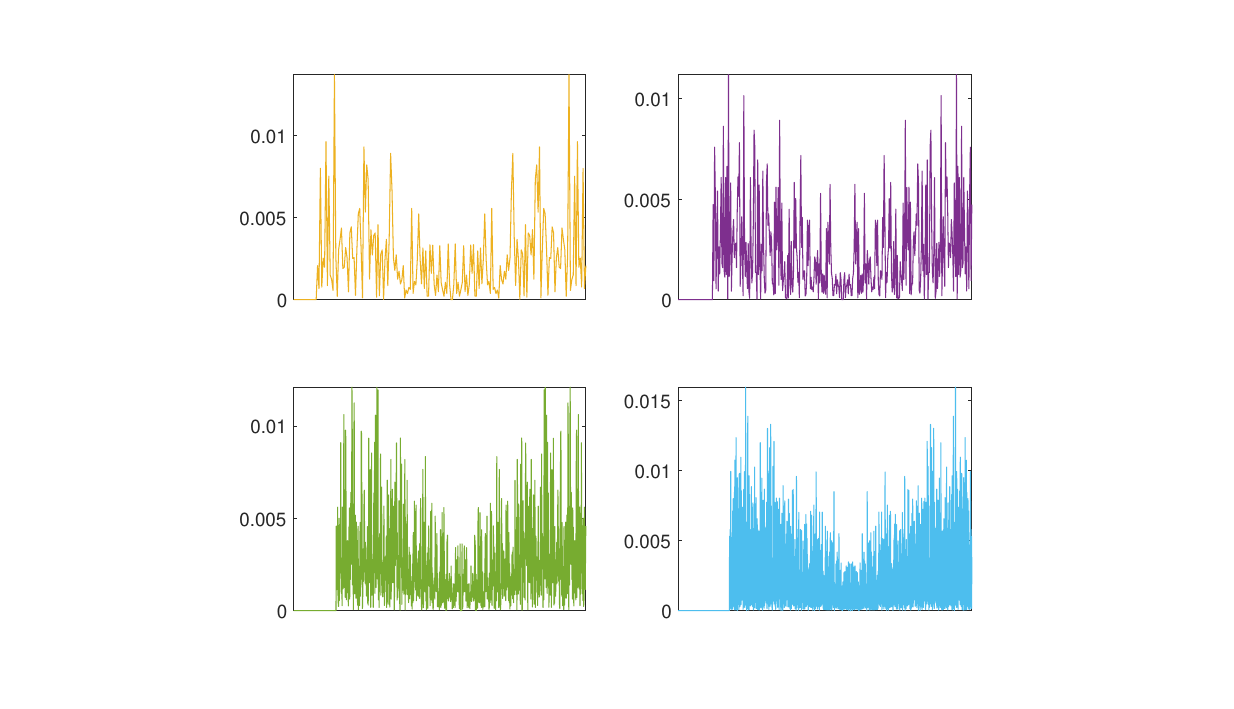} \vskip -0.5cm
  \caption{Minimal distance of eigenvalues of $B_{n,t}$ from $f(\theta_1,\theta_2)=(2-2\cos \theta_1)+(2-2\cos \theta_2)$ and $t=2$ for different $h$ values.}
\label{FD_Bnt_full_t2_errore}
\end{figure}
%--------------------------------------------------------------\begin{figure}
\begin{figure}
\centering
  \includegraphics[width=\textwidth]{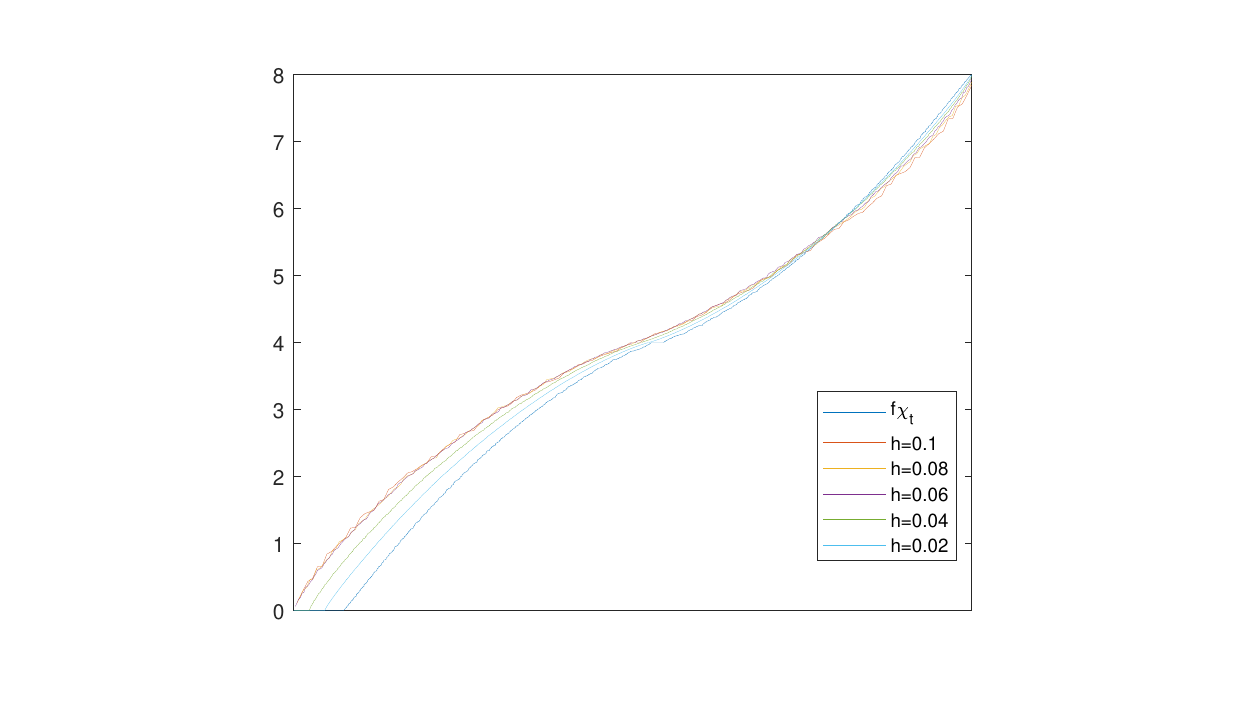} \vskip -0.5cm
  \caption{Eigenvalues distribution of $B_{n,t}$ for different $h$ values  together with the sampling of $f(\theta_1,\theta_2)=(2-2\cos \theta_1)+(2-2\cos \theta_2)$ and $t=4$.}
\label{FD_Bnt_full_t4}
\end{figure}
\begin{figure}
\centering
  \includegraphics[width=\textwidth]{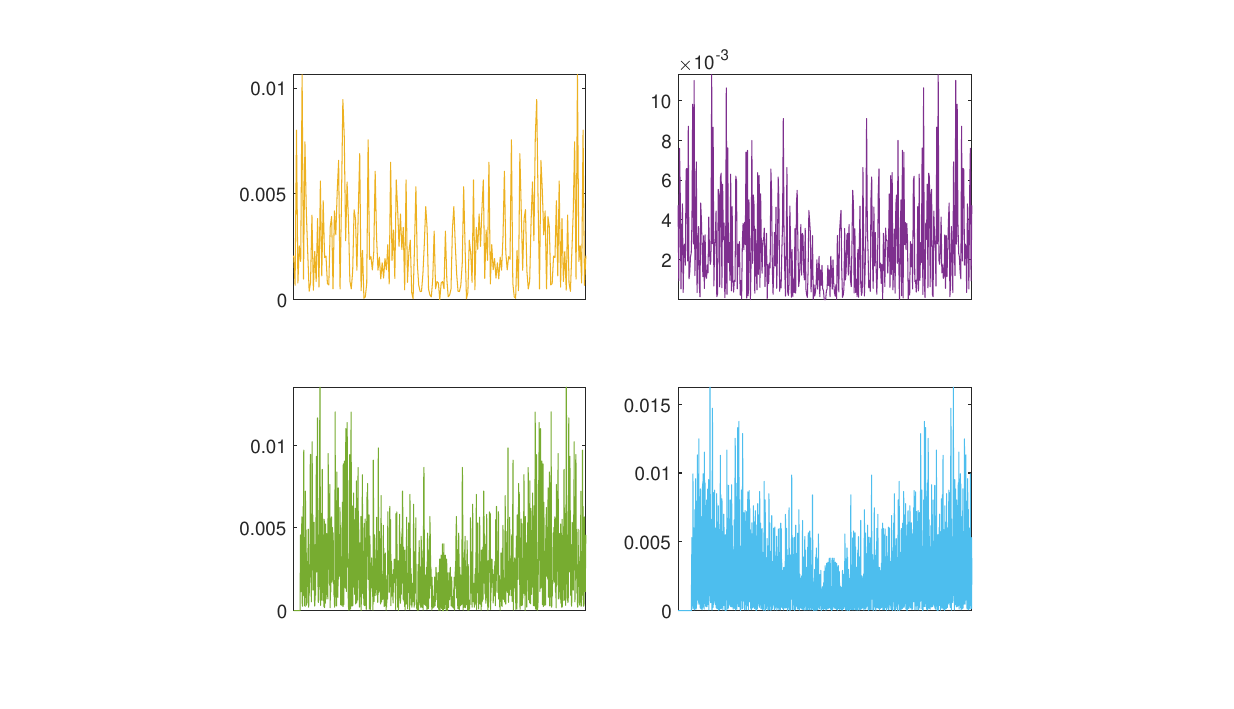} \vskip -0.5cm
  \caption{Minimal distance of eigenvalues of $B_{n,t}$ from $f(\theta_1,\theta_2)=(2-2\cos \theta_1)+(2-2\cos \theta_2)$ and $t=4$ for different $h$ values.}
\label{FD_Bnt_full_t4_errore}
\end{figure}
%--------------------------------------------------------------
\begin{figure}
\centering
  \includegraphics[width=\textwidth]{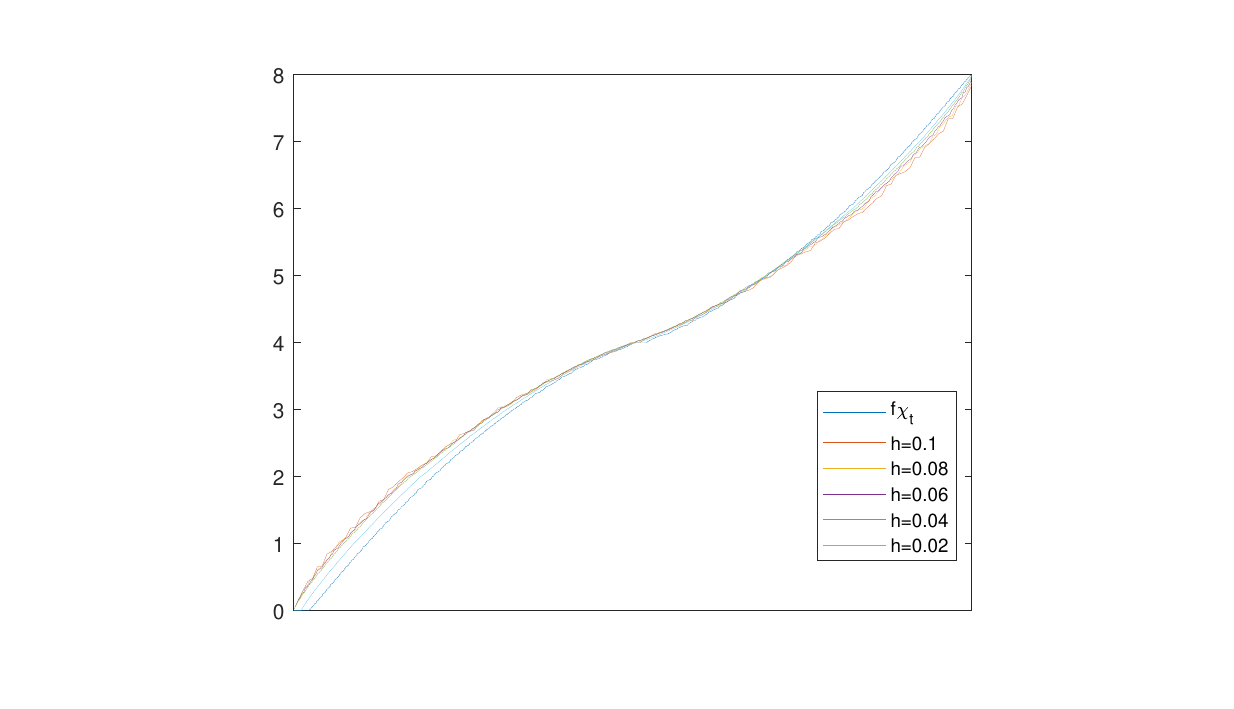} \vskip -0.5cm
  \caption{Eigenvalues distribution of $B_{n,t}$ for different $h$ values  together with the sampling of $f(\theta_1,\theta_2)=(2-2\cos \theta_1)+(2-2\cos \theta_2)$ and $t=6$.}
\label{FD_Bnt_full_t6}
\end{figure}
\begin{figure}
\centering
  \includegraphics[width=\textwidth]{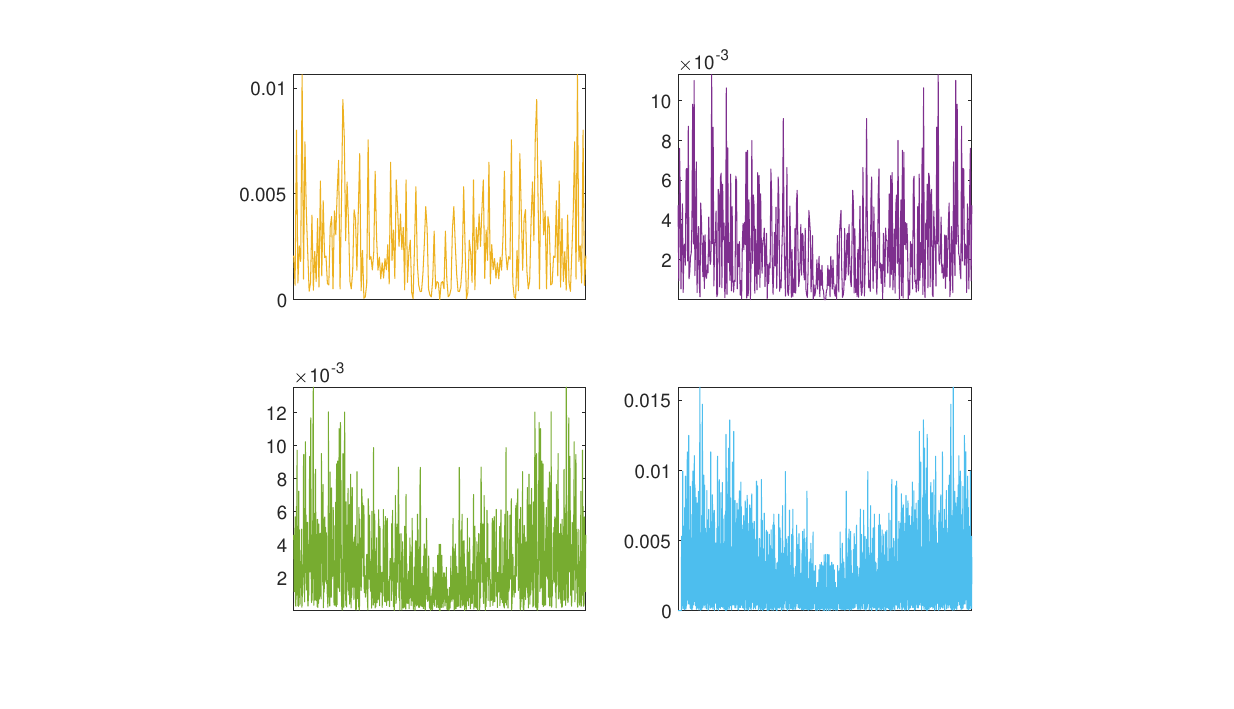} \vskip -0.5cm
  \caption{Minimal distance of eigenvalues of $B_{n,t}$ from $f(\theta_1,\theta_2)=(2-2\cos \theta_1)+(2-2\cos \theta_2)$ and $t=6$ for different $h$ values.}
\label{FD_Bnt_full_t6_errore}
\end{figure}
%--------------------------------------------------------------
\begin{figure}
\centering
  \includegraphics[width=\textwidth]{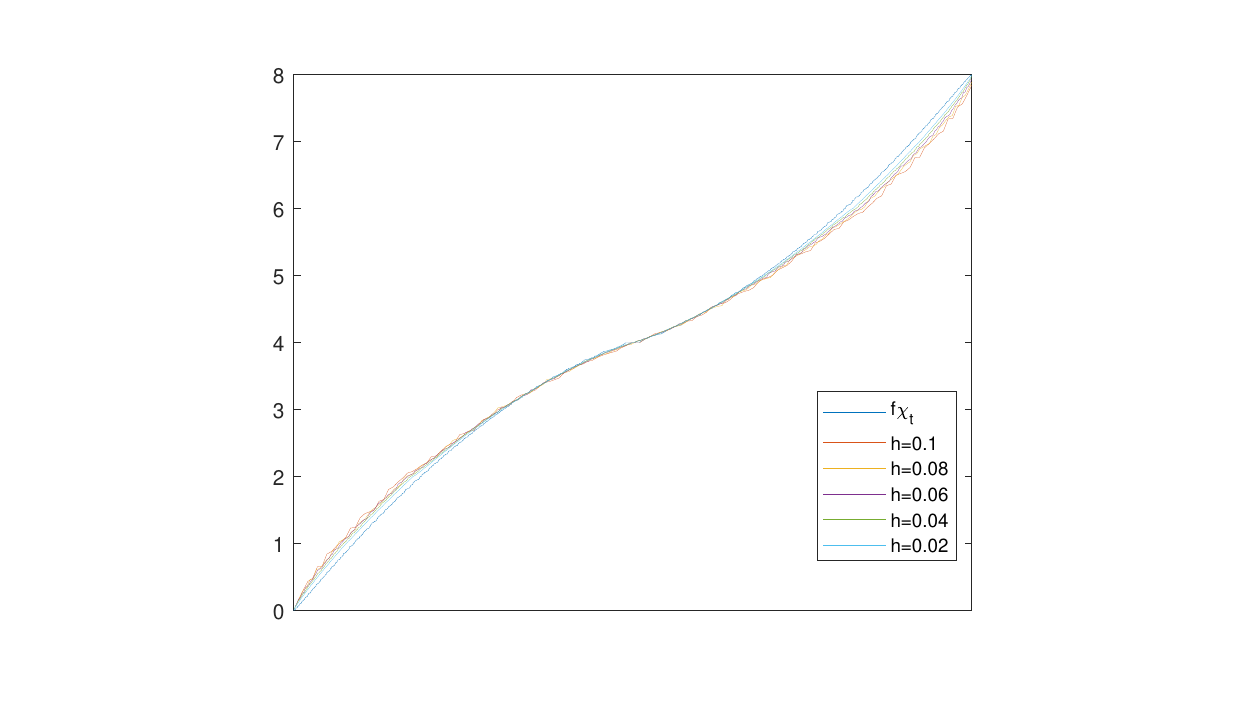} \vskip -0.5cm
  \caption{Eigenvalues distribution of $B_{n,t}$ for different $h$ values  together with the sampling of $f(\theta_1,\theta_2)=(2-2\cos \theta_1)+(2-2\cos \theta_2)$ and $t=8$.}
\label{FD_Bnt_full_t8}
\end{figure}
\begin{figure}
\centering
  \includegraphics[width=\textwidth]{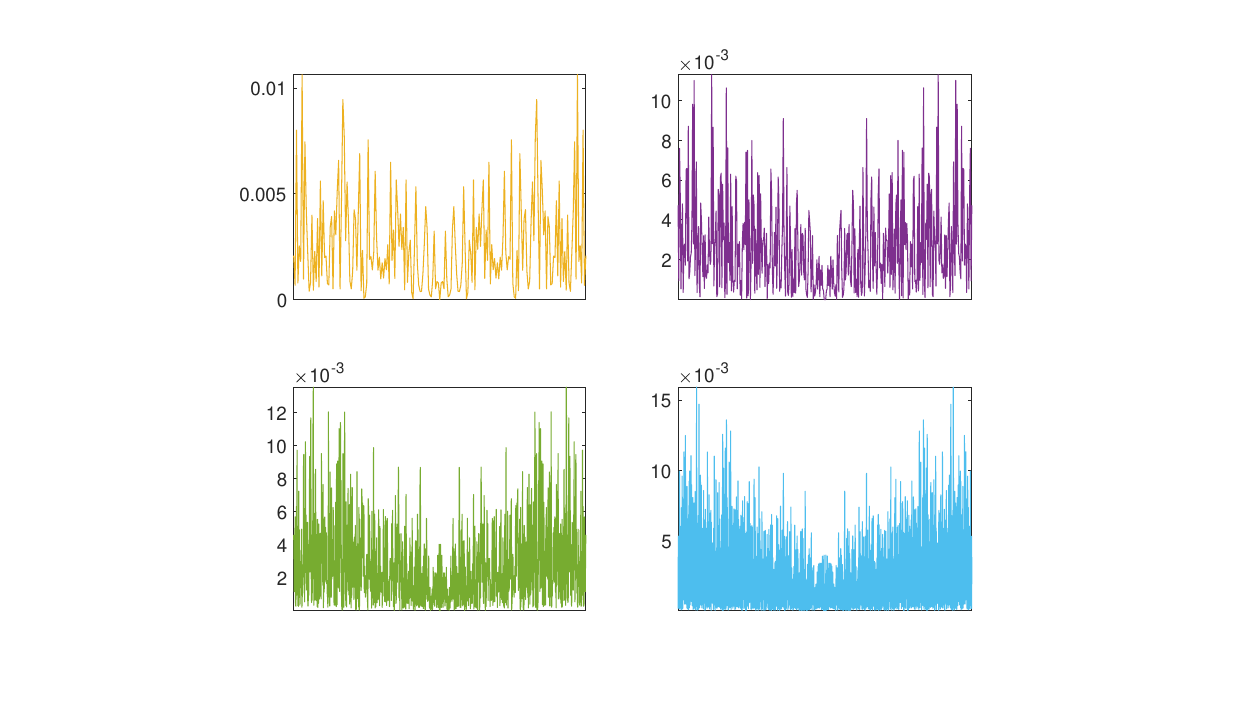} \vskip -0.5cm
  \caption{Minimal distance of eigenvalues of $B_{n,t}$ from $f(\theta_1,\theta_2)=(2-2\cos \theta_1)+(2-2\cos \theta_2)$ and $t=8$ for different $h$ values.}
\label{FD_Bnt_full_t8_errore}
\end{figure}
%--------------------------------------------------------------
%-------------------------------------------------------------- %

%///////////////////////////////////////////////////////////
\subsubsection{$Q_1$ - $f(\theta_1,\theta_2)=(8 -2\cos\theta_1-2\cos\theta_2-4\cos\theta_1 \cos\theta_2)/3$}
$Q_1$: $f(\theta_1,\theta_2)=(8 -2\cos\theta_1-2\cos\theta_2-4\cos\theta_1 \cos\theta_2)/3$\\

%-------------------------------------------------------------- %
%-------------------------------------------------------------- %
\begin{figure}
\centering
  \includegraphics[width=\textwidth]{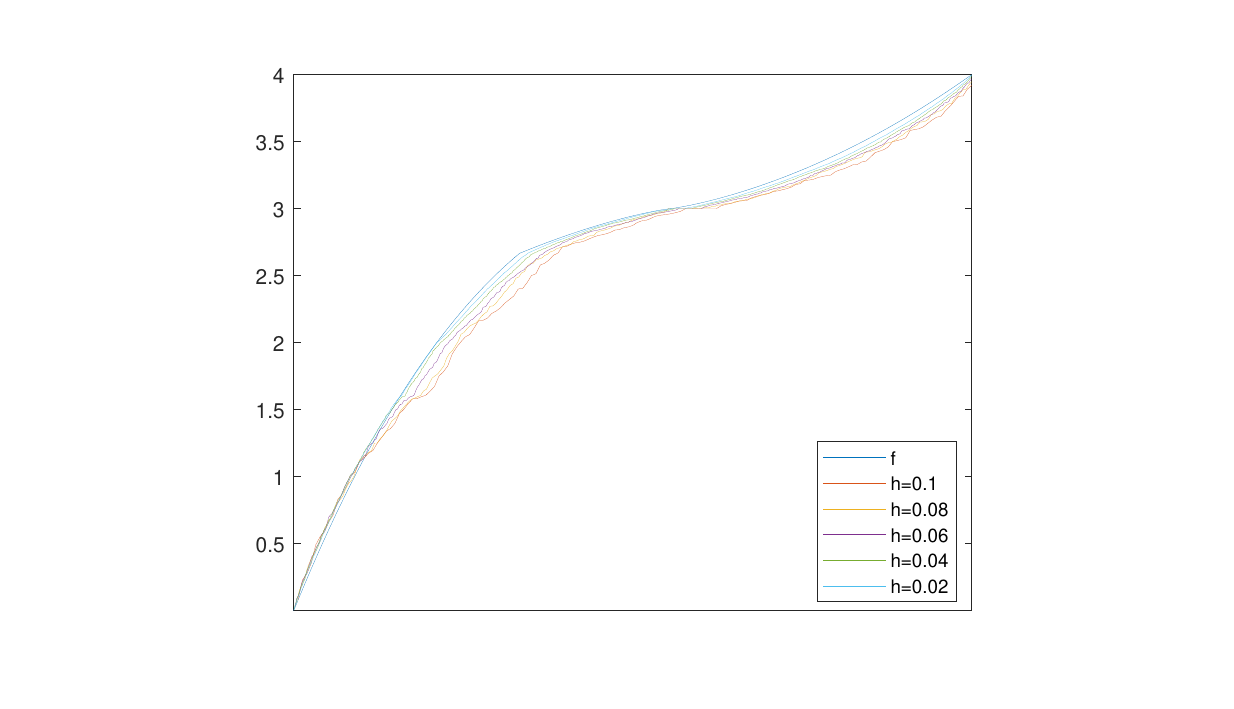} \vskip -0.5cm
  \caption{Eigenvalues distribution of $A_n$ for different $h$ values  together with the sampling of $f(\theta_1,\theta_2)=(8 -2\cos\theta_1-2\cos\theta_2-4\cos\theta_1 \cos\theta_2)/3$.}
\label{Q1_A}
\end{figure}
\begin{figure}
\centering
  \includegraphics[width=\textwidth]{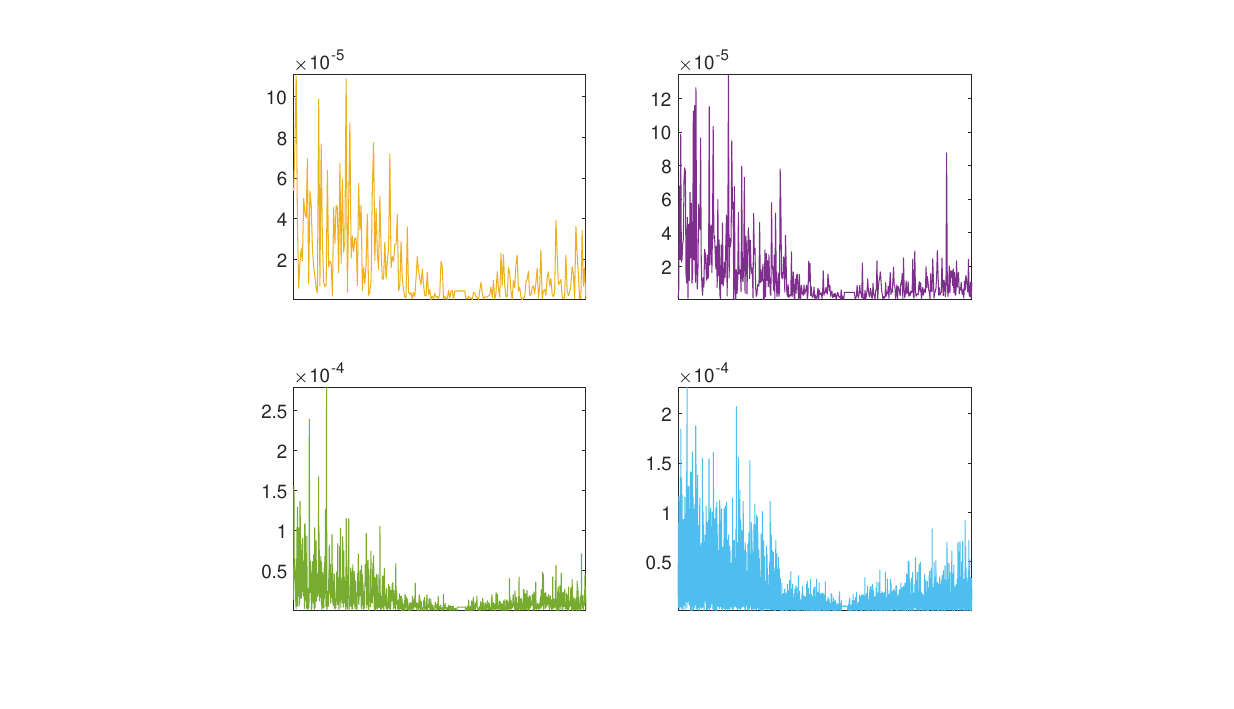} \vskip -0.5cm
  \caption{Minimal distance of eigenvalues of $A_n$ from $f(\theta_1,\theta_2)=(8 -2\cos\theta_1-2\cos\theta_2-4\cos\theta_1 \cos\theta_2)/3$ for different $h$ values.}
\label{Q1_A_errore}
\end{figure}
%-------------------------------------------------------------- %
%-------------------------------------------------------------- %
\begin{figure}
\centering
  \includegraphics[width=\textwidth]{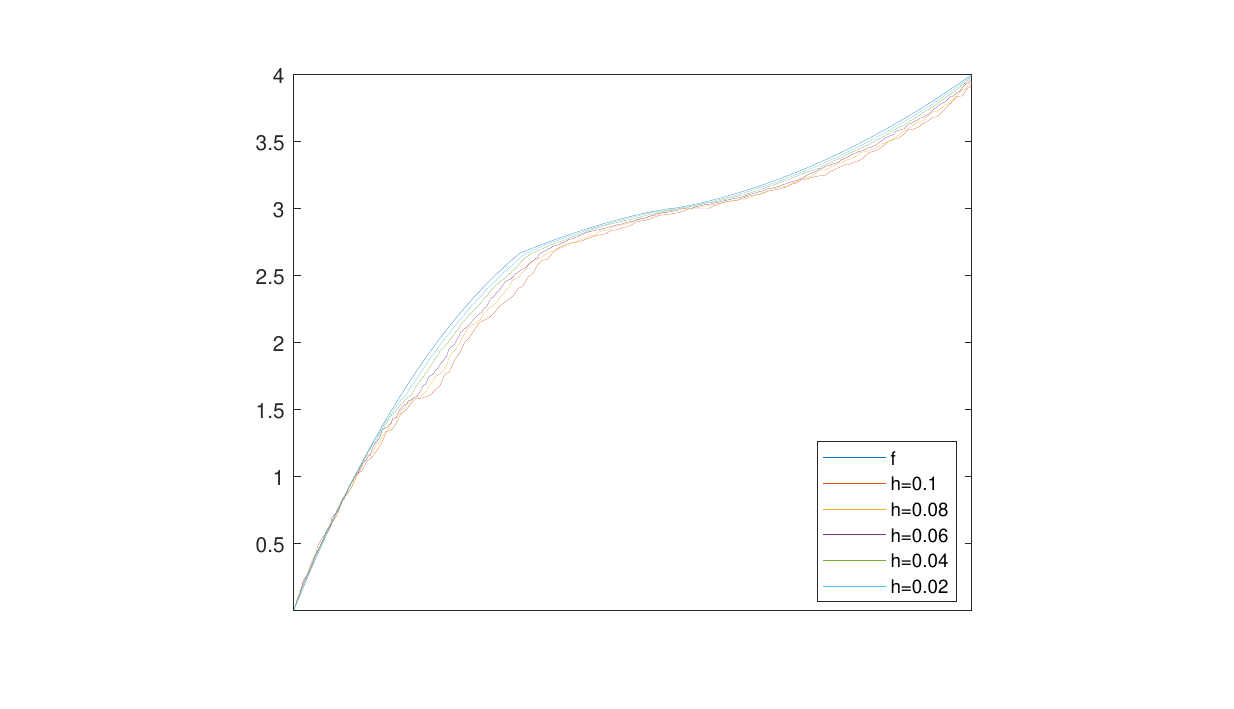} \vskip -0.5cm
  \caption{Eigenvalues distribution of $C_{n,t}$ for different $h$ values  together with the sampling of $f(\theta_1,\theta_2)=(8 -2\cos\theta_1-2\cos\theta_2-4\cos\theta_1 \cos\theta_2)/3$ and $t=2$.}
\label{Q1_Bnt_small_t2}
\end{figure}
\begin{figure}
\centering
  \includegraphics[width=\textwidth]{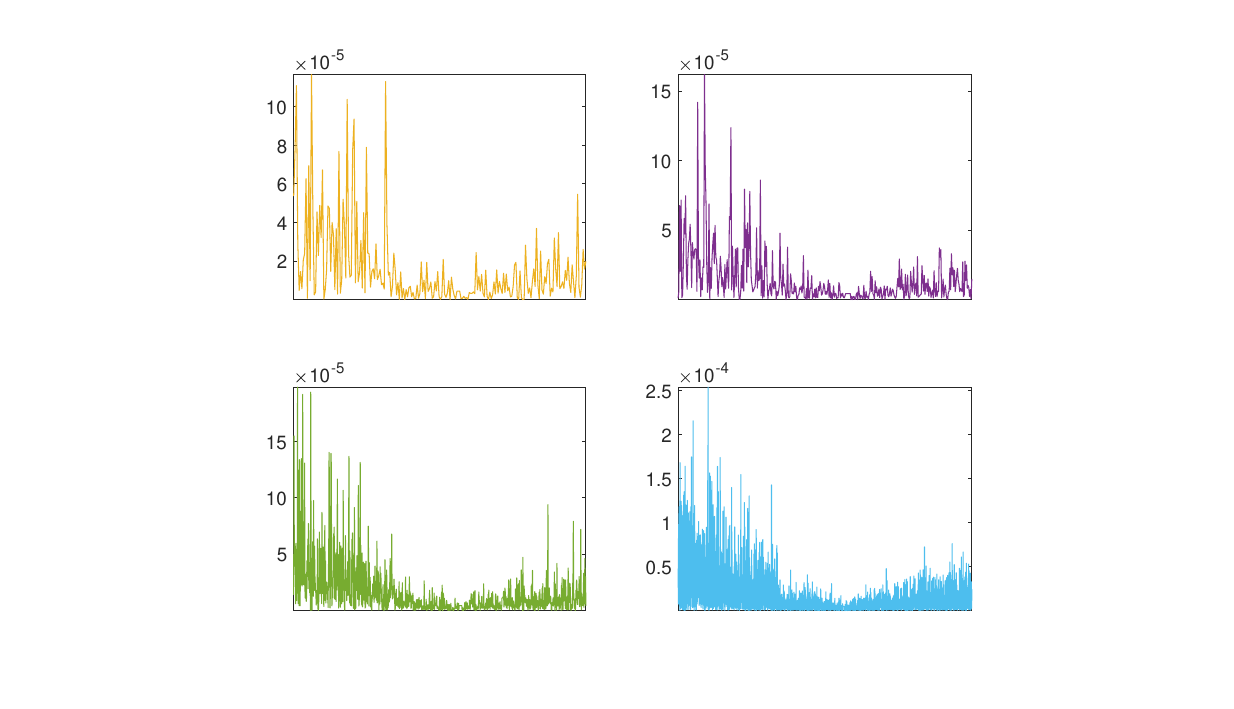} \vskip -0.5cm
  \caption{Minimal distance of eigenvalues of $C_{n,t}$ from $f(\theta_1,\theta_2)=(8 -2\cos\theta_1-2\cos\theta_2-4\cos\theta_1 \cos\theta_2)/3$ and $t=2$ for different $h$ values.}
\label{Q1_Bnt_small_t2_errore}
\end{figure}
%--------------------------------------------------------------\begin{figure}
\begin{figure}
\centering
  \includegraphics[width=\textwidth]{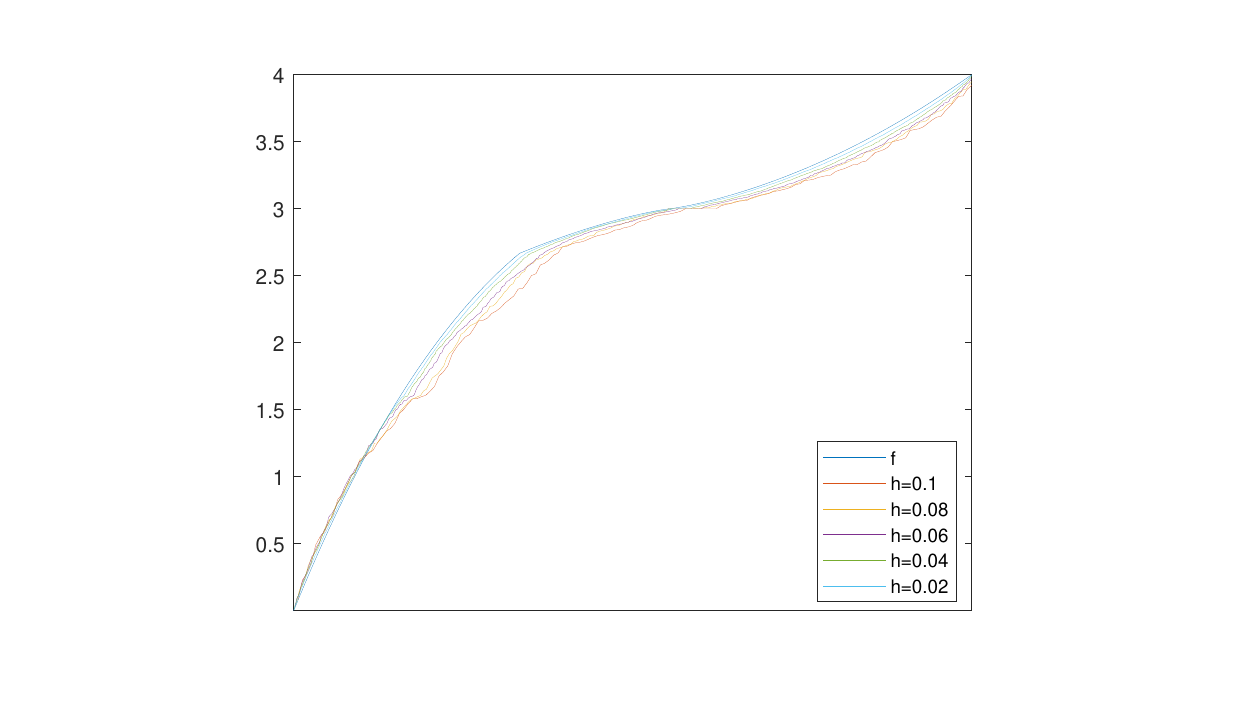} \vskip -0.5cm
  \caption{Eigenvalues distribution of $C_{n,t}$ for different $h$ values  together with the sampling of $f(\theta_1,\theta_2)=(8 -2\cos\theta_1-2\cos\theta_2-4\cos\theta_1 \cos\theta_2)/3$ and $t=4$.}
\label{Q1_Bnt_small_t4}
\end{figure}
\begin{figure}
\centering
  \includegraphics[width=\textwidth]{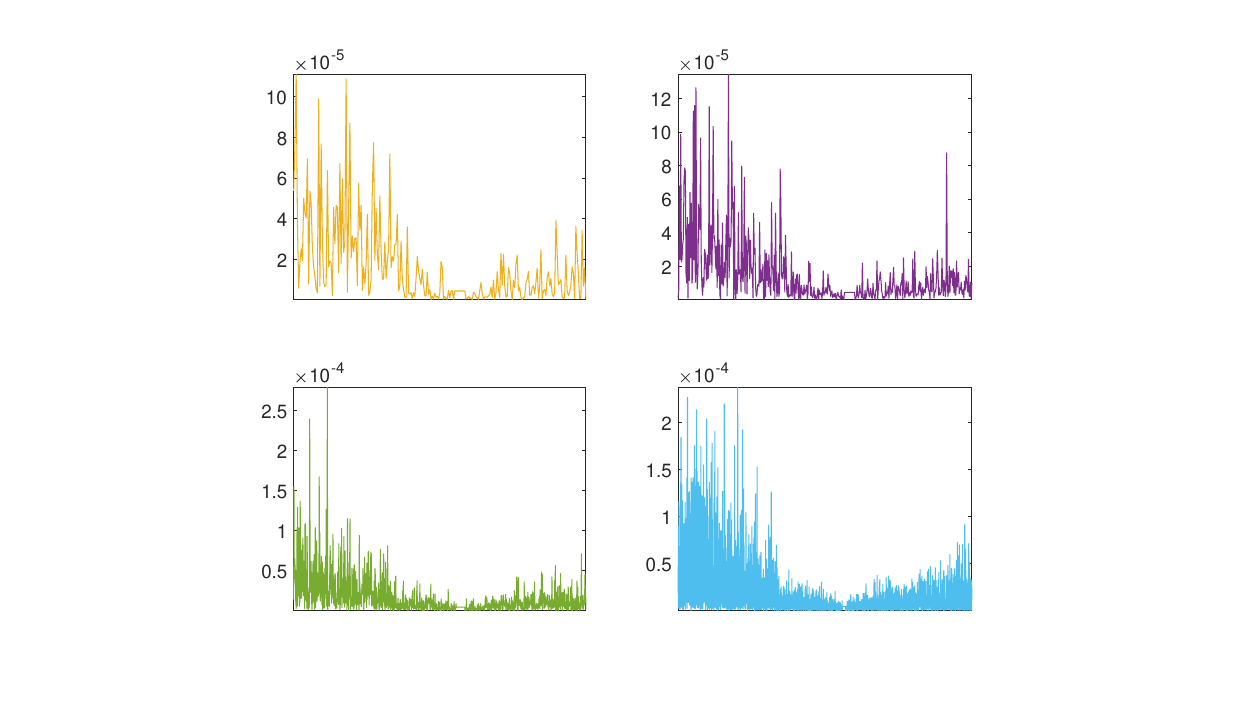} \vskip -0.5cm
  \caption{Minimal distance of eigenvalues of $C_{n,t}$ from $f(\theta_1,\theta_2)=(8 -2\cos\theta_1-2\cos\theta_2-4\cos\theta_1 \cos\theta_2)/3$ and $t=4$ for different $h$ values.}
\label{Q1_Bnt_small_t4_errore}
\end{figure}
%--------------------------------------------------------------
\begin{figure}
\centering
  \includegraphics[width=\textwidth]{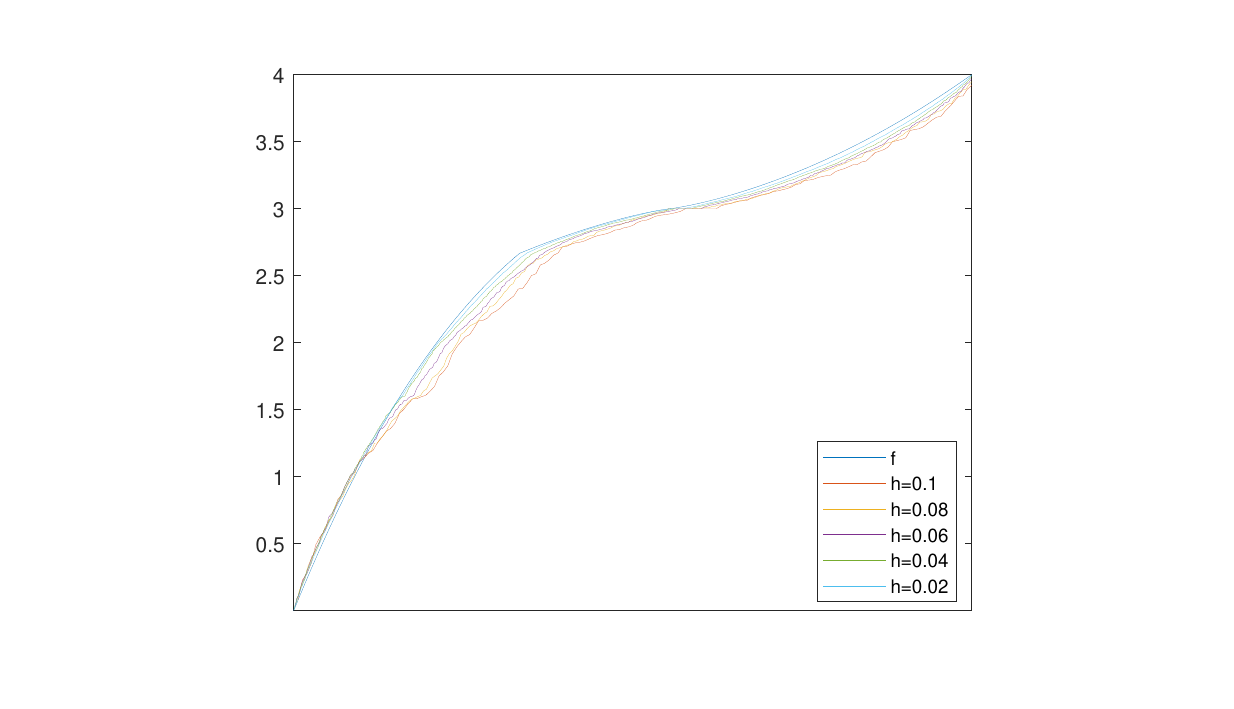} \vskip -0.5cm
  \caption{Eigenvalues distribution of $C_{n,t}$ for different $h$ values  together with the sampling of $f(\theta_1,\theta_2)=(8 -2\cos\theta_1-2\cos\theta_2-4\cos\theta_1 \cos\theta_2)/3$ and $t=6$.}
\label{Q1_Bnt_small_t6}
\end{figure}
\begin{figure}
\centering
  \includegraphics[width=\textwidth]{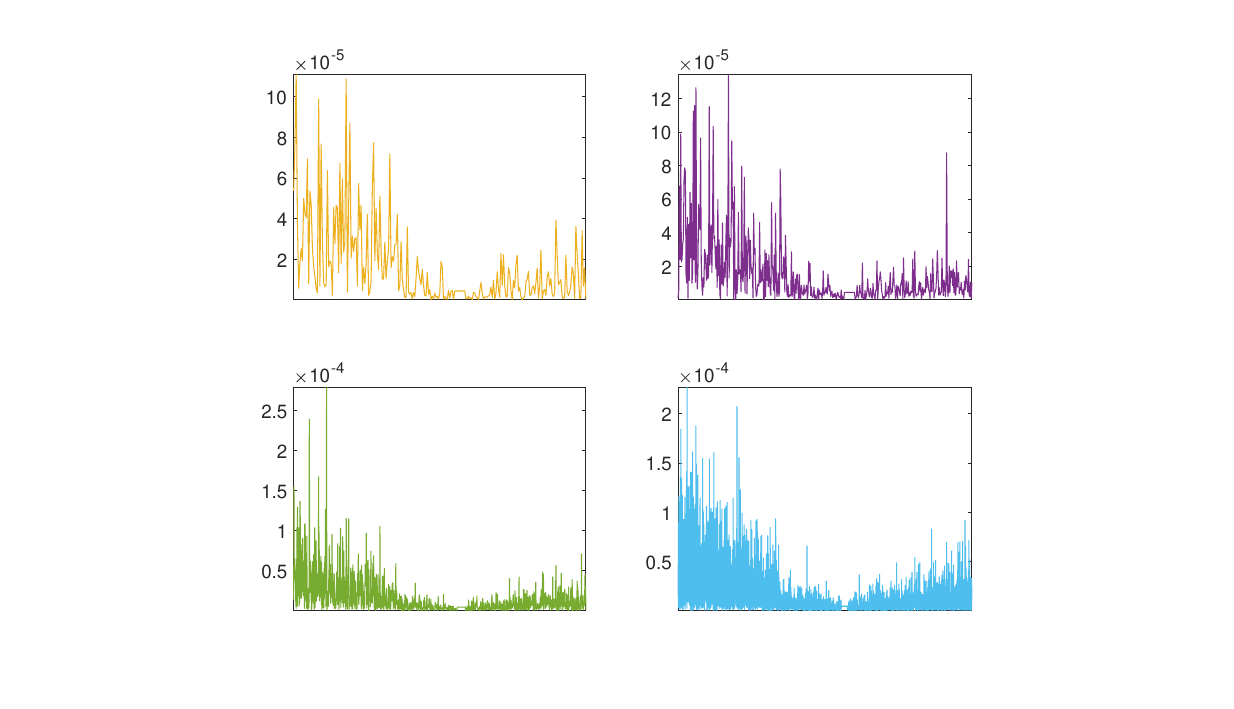} \vskip -0.5cm
  \caption{Minimal distance of eigenvalues of $C_{n,t}$ from $f(\theta_1,\theta_2)=(8 -2\cos\theta_1-2\cos\theta_2-4\cos\theta_1 \cos\theta_2)/3$ and $t=6$ for different $h$ values.}
\label{Q1_Bnt_small_t6_errore}
\end{figure}
%--------------------------------------------------------------
%-------------------------------------------------------------- %
\begin{figure}
\centering
  \includegraphics[width=\textwidth]{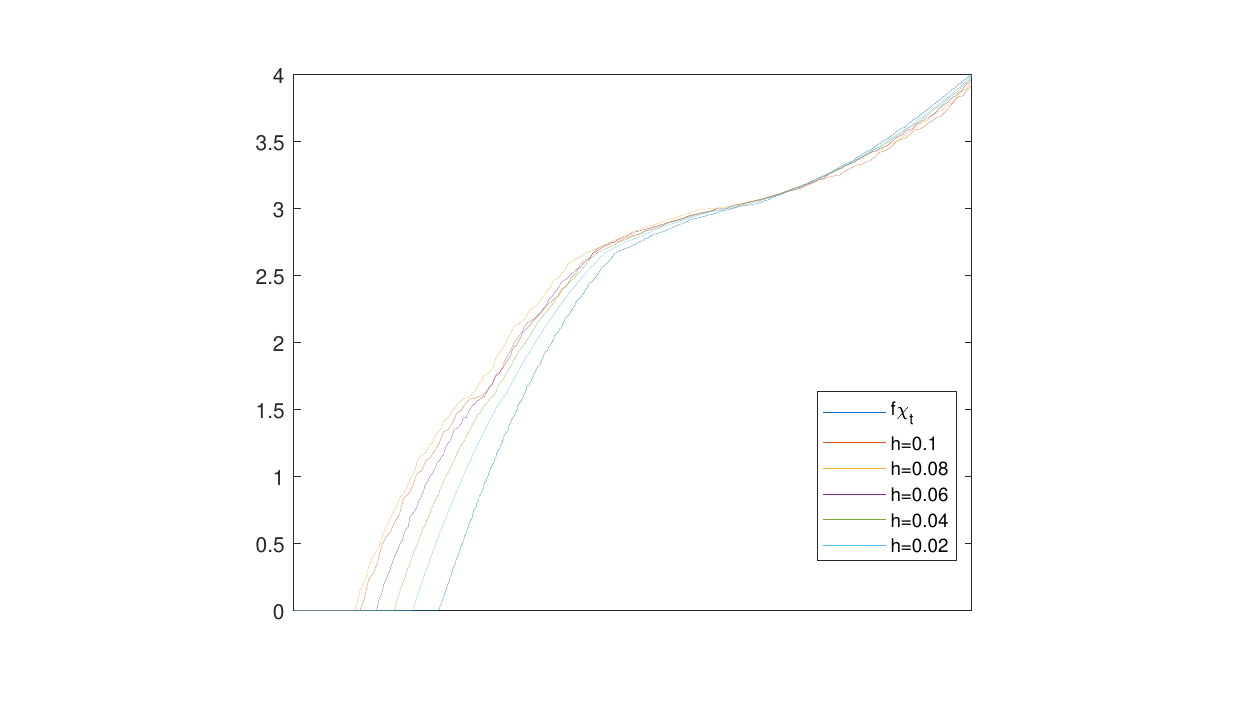} \vskip -0.5cm
  \caption{Eigenvalues distribution of $B_{n,t}$ for different $h$ values  together with the sampling of $f(\theta_1,\theta_2)=(8 -2\cos\theta_1-2\cos\theta_2-4\cos\theta_1 \cos\theta_2)/3$ and $t=2$.}
\label{Q1_Bnt_full_t2}
\end{figure}
\begin{figure}
\centering
  \includegraphics[width=\textwidth]{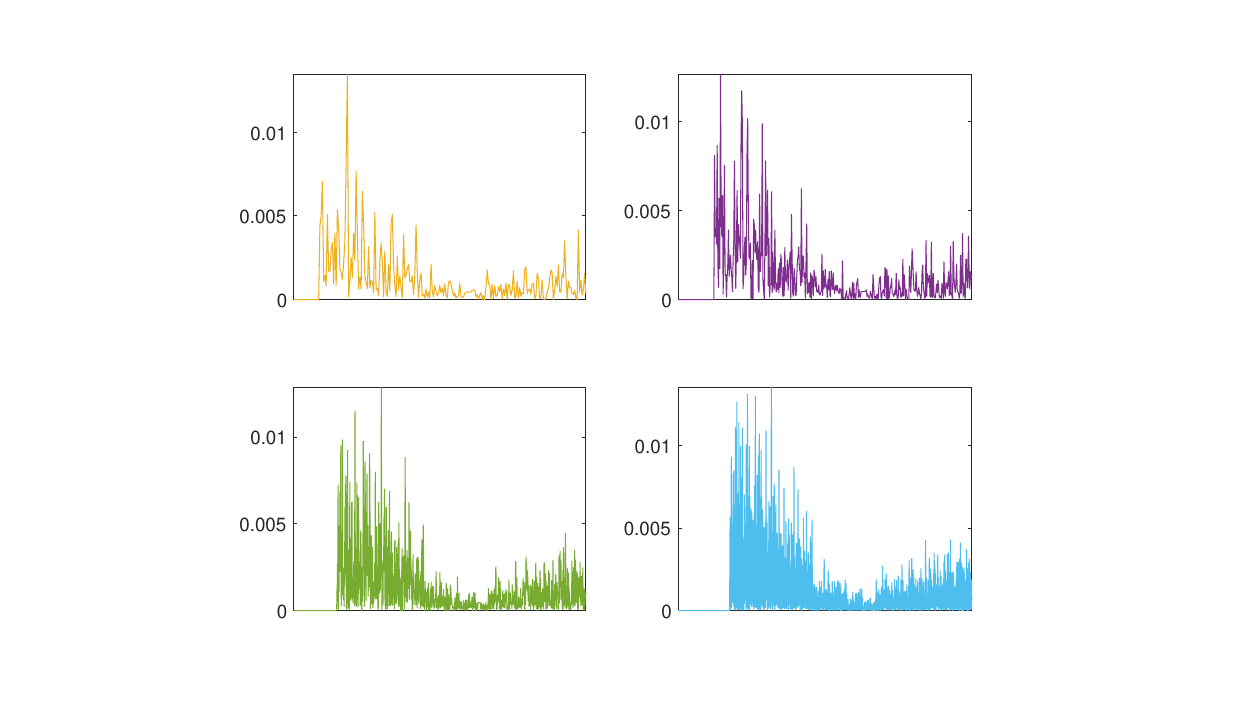} \vskip -0.5cm
  \caption{Minimal distance of eigenvalues of $B_{n,t}$ from $f(\theta_1,\theta_2)=(8 -2\cos\theta_1-2\cos\theta_2-4\cos\theta_1 \cos\theta_2)/3$ and $t=2$ for different $h$ values.}
\label{Q1_Bnt_full_t2_errore}
\end{figure}
%--------------------------------------------------------------\begin{figure}
\begin{figure}
\centering
  \includegraphics[width=\textwidth]{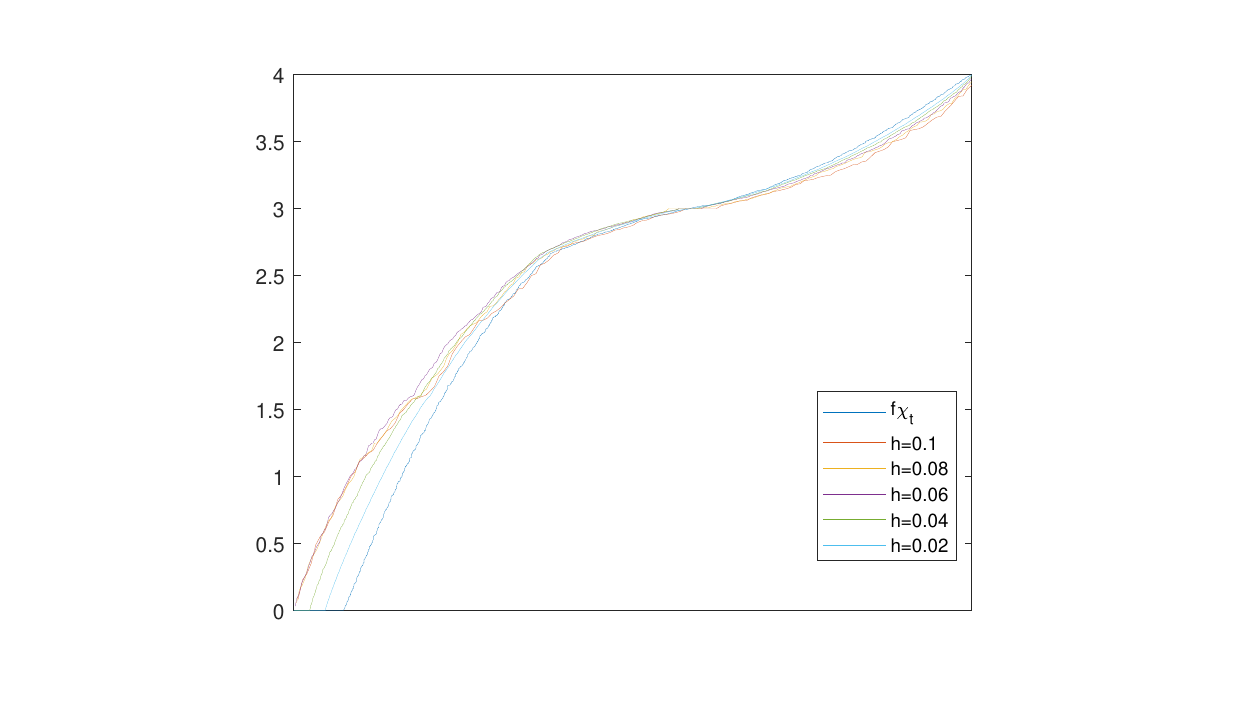} \vskip -0.5cm
  \caption{Eigenvalues distribution of $B_{n,t}$ for different $h$ values  together with the sampling of $f(\theta_1,\theta_2)=(8 -2\cos\theta_1-2\cos\theta_2-4\cos\theta_1 \cos\theta_2)/3$ and $t=4$.}
\label{Q1_Bnt_full_t4}
\end{figure}
\begin{figure}
\centering
  \includegraphics[width=\textwidth]{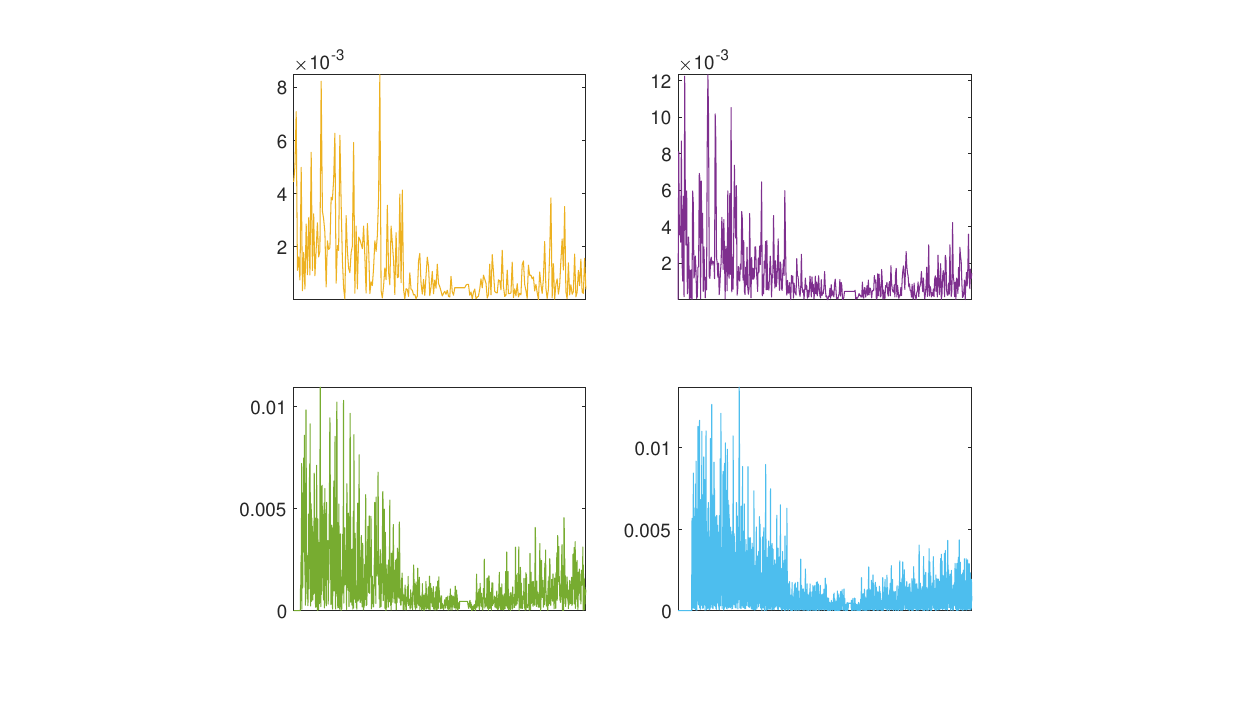} \vskip -0.5cm
  \caption{Minimal distance of eigenvalues of $B_{n,t}$ from $f(\theta_1,\theta_2)=(8 -2\cos\theta_1-2\cos\theta_2-4\cos\theta_1 \cos\theta_2)/3$ and $t=4$ for different $h$ values.}
\label{Q1_Bnt_full_t4_errore}
\end{figure}
%--------------------------------------------------------------
\begin{figure}
\centering
  \includegraphics[width=\textwidth]{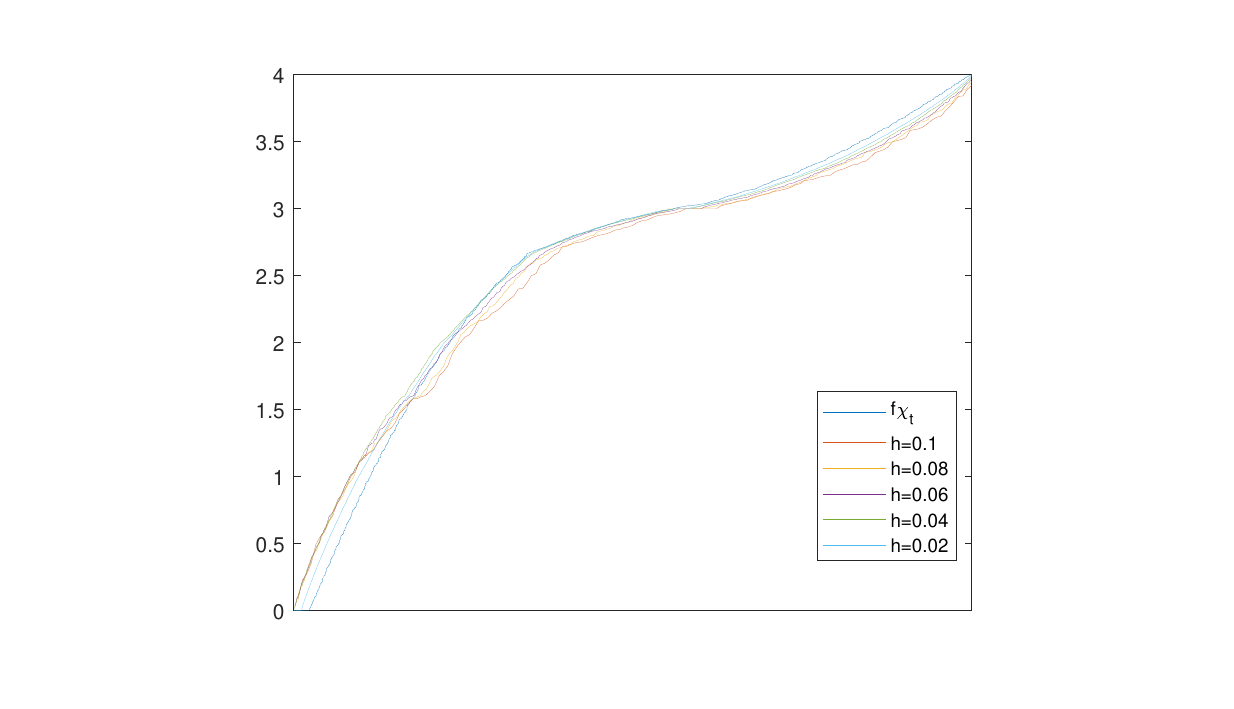} \vskip -0.5cm
  \caption{Eigenvalues distribution of $B_{n,t}$ for different $h$ values  together with the sampling of $f(\theta_1,\theta_2)=(8 -2\cos\theta_1-2\cos\theta_2-4\cos\theta_1 \cos\theta_2)/3$ and $t=6$.}
\label{Q1_Bnt_full_t6}
\end{figure}
\begin{figure}
\centering
  \includegraphics[width=\textwidth]{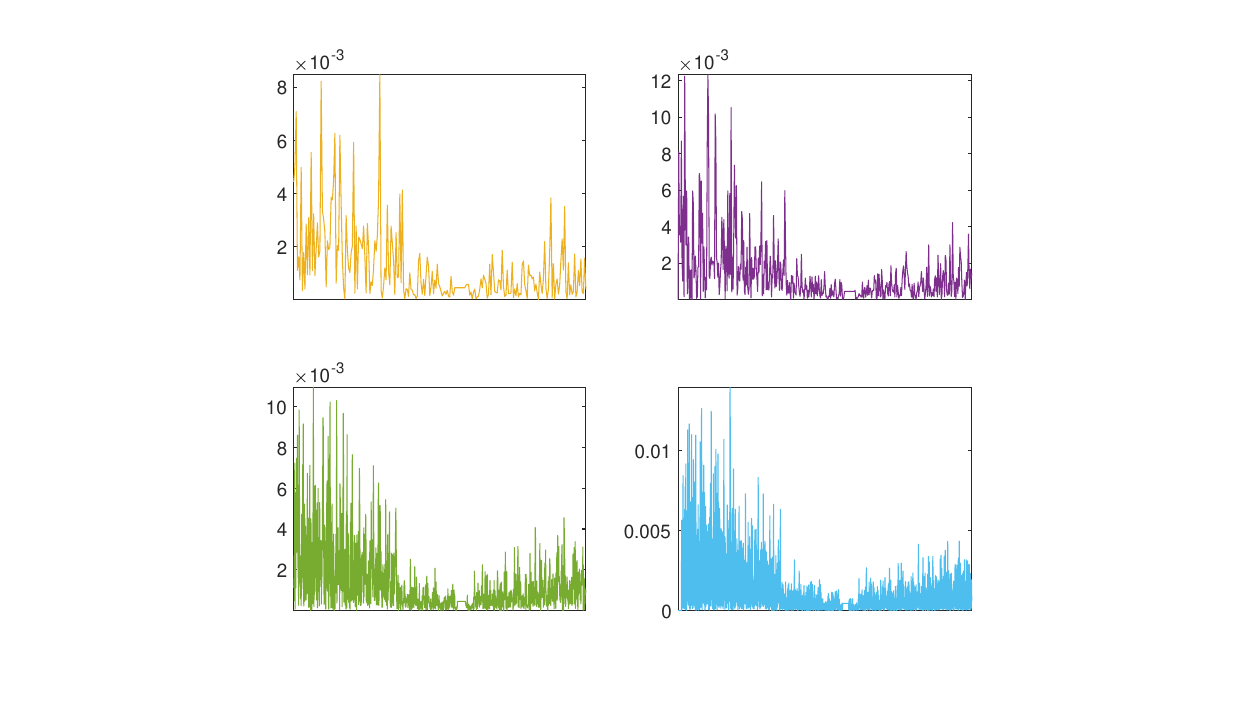} \vskip -0.5cm
  \caption{Minimal distance of eigenvalues of $B_{n,t}$ from $f(\theta_1,\theta_2)=(8 -2\cos\theta_1-2\cos\theta_2-4\cos\theta_1 \cos\theta_2)/3$ and $t=6$ for different $h$ values.}
\label{Q1_Bnt_full_t6_errore}
\end{figure}
%--------------------------------------------------------------
\begin{figure}
\centering
  \includegraphics[width=\textwidth]{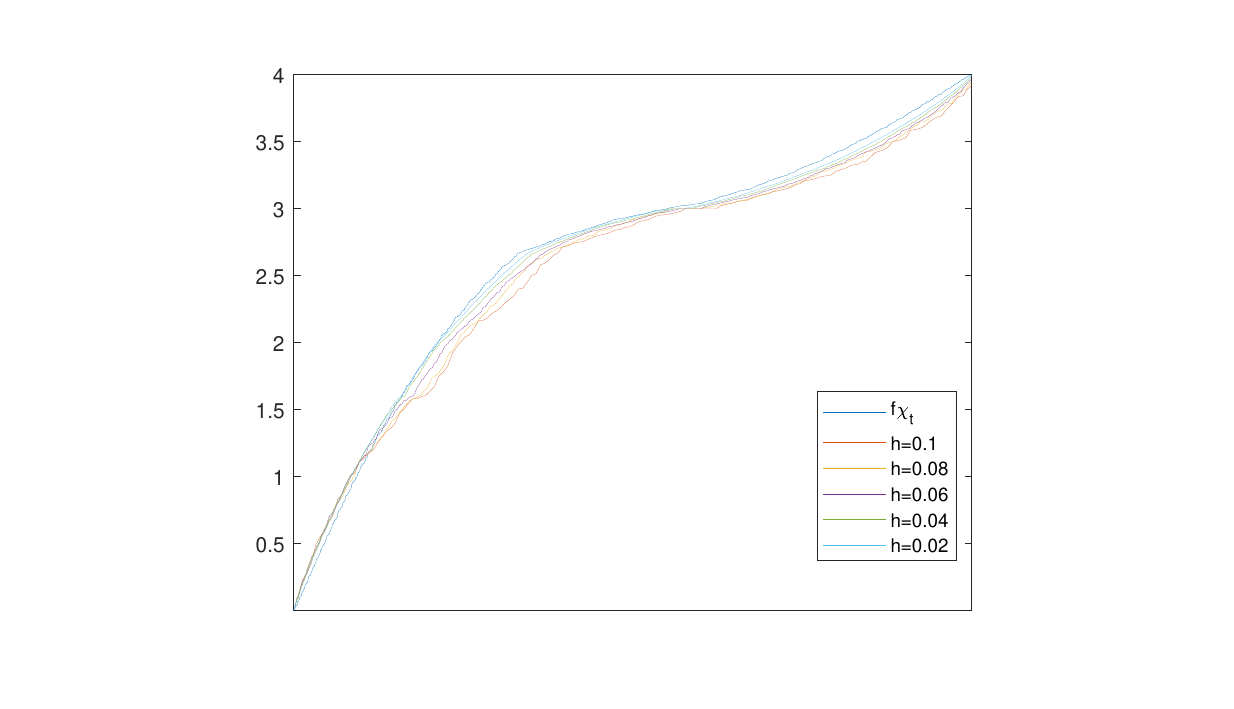} \vskip -0.5cm
  \caption{Eigenvalues distribution of $B_{n,t}$ for different $h$ values  together with the sampling of $f(\theta_1,\theta_2)=(8 -2\cos\theta_1-2\cos\theta_2-4\cos\theta_1 \cos\theta_2)/3$ and $t=8$.}
\label{Q1_Bnt_full_t8}
\end{figure}
\begin{figure}
\centering
  \includegraphics[width=\textwidth]{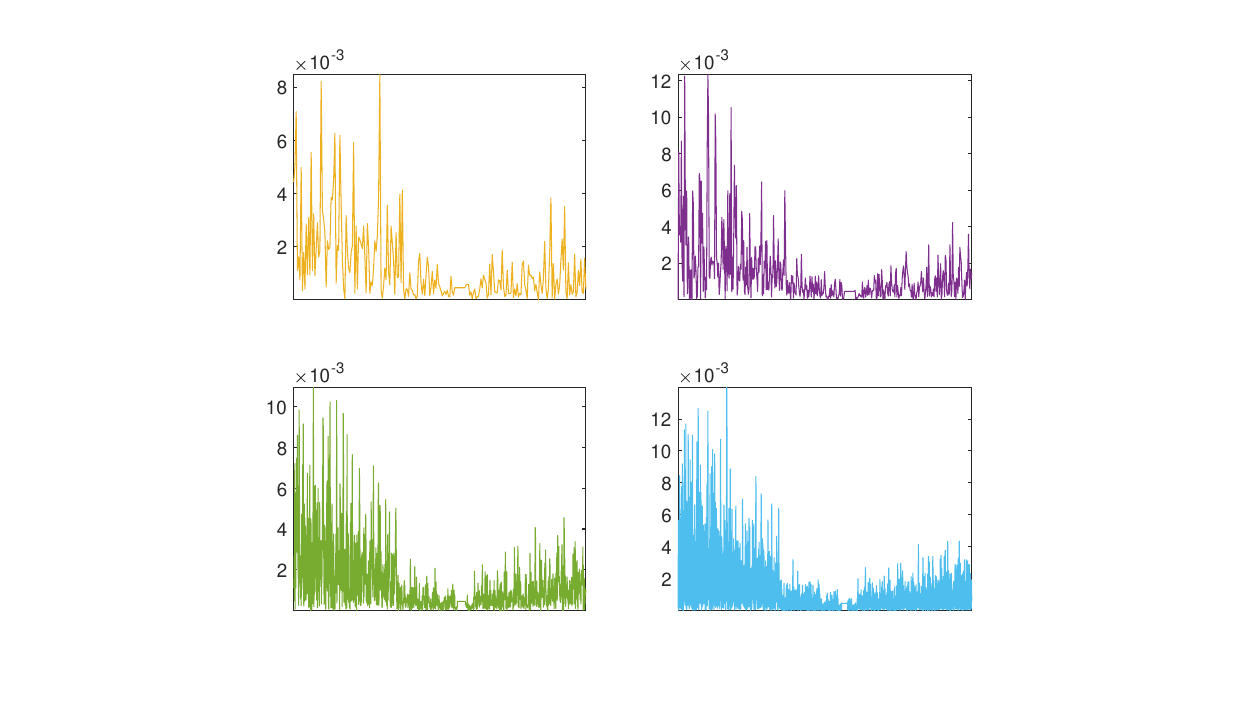} \vskip -0.5cm
  \caption{Minimal distance of eigenvalues of $B_{n,t}$ from $f(\theta_1,\theta_2)=(8 -2\cos\theta_1-2\cos\theta_2-4\cos\theta_1 \cos\theta_2)/3$ and $t=8$ for different $h$ values.}
\label{Q1_Bnt_full_t8_errore}
\end{figure}
%--------------------------------------------------------------
%-------------------------------------------------------------- %

%///////////////////////////////////////////////////////////
\subsubsection{$P_2$ - $f_{{P}_2}: [-\pi,\pi]^2 \longrightarrow \mathbb{C}^{4\times 4}$}
$P_2$: Matrix-valued symbol $f_{{P}_2}: [-\pi,\pi]^2 \longrightarrow \mathbb{C}^{4\times 4}$ with
\begin{equation}\label{eq:P2symbol}
f_{{P}_2}(\theta_{1},\theta_{2})  =  \left [
\begin{array}{rr|rr}
\alpha & -\beta(1+e^{\hat{i} \theta_{1}})  & -\beta(1+e^{\hat{i} \theta_{2}})  & 0 \\
-\beta(1+e^{-\hat{i} \theta_{1}}) & \alpha                 & 0                      & -\beta(1+e^{\hat{i} \theta_{2}})\\
\hline
-\beta(1+e^{-\hat{i} \theta_{2}}) & 0                     & \alpha                 & -\beta(1+e^{\hat{i} \theta_{1}})\\
0                      & -\beta(1+e^{-\hat{i} \theta_{2}}) & -\beta(1+e^{-\hat{i} \theta_{1}}) & \gamma +\frac{\beta}{2} (\cos(\theta_{1})+\cos(\theta_{2}))\\
\end{array}
\right ]
\end{equation}
with $\alpha={16}/{3}$, $\beta={4}/{3}$, and $\gamma=4$.\\

%-------------------------------------------------------------- %
%-------------------------------------------------------------- %
\begin{figure}
\centering
  \includegraphics[width=\textwidth]{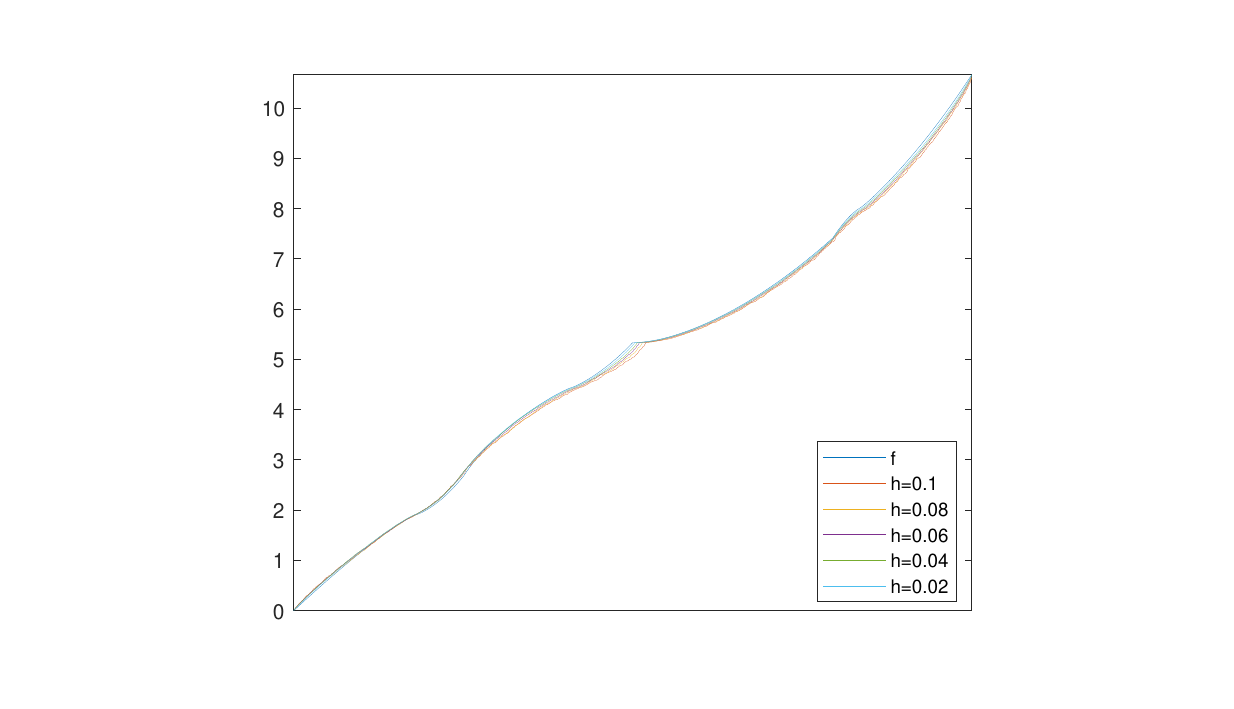} \vskip -0.5cm
  \caption{Eigenvalues distribution of $A_n$ for different $h$ values  together with the sampling of $f(\theta_1,\theta_2)=f_{{P}_2}(\theta_{1},\theta_{2})$.}
\label{P2_A}
\end{figure}
\begin{figure}
\centering
  \includegraphics[width=\textwidth]{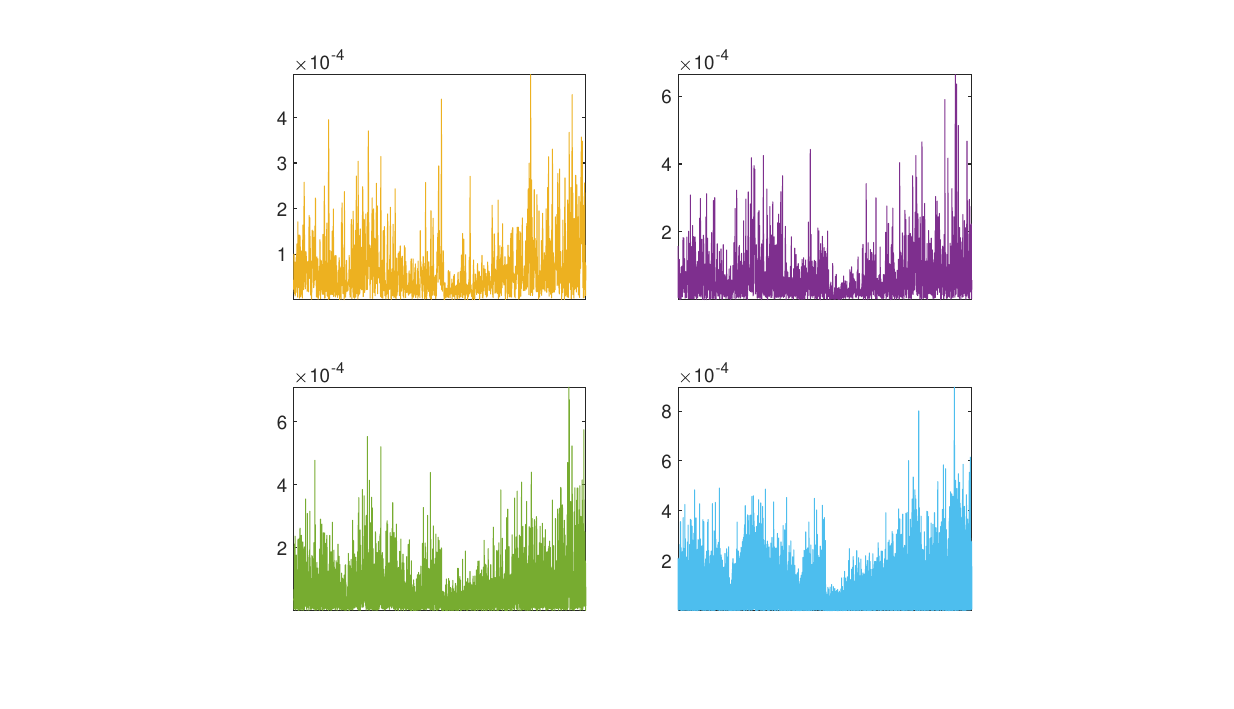} \vskip -0.5cm
  \caption{Minimal distance of eigenvalues of $A_n$ from $f(\theta_1,\theta_2)=f_{{P}_2}(\theta_{1},\theta_{2})$ for different $h$ values.}
\label{P2_A_errore}
\end{figure}
%-------------------------------------------------------------- %
%-------------------------------------------------------------- %
\begin{figure}
\centering
  \includegraphics[width=\textwidth]{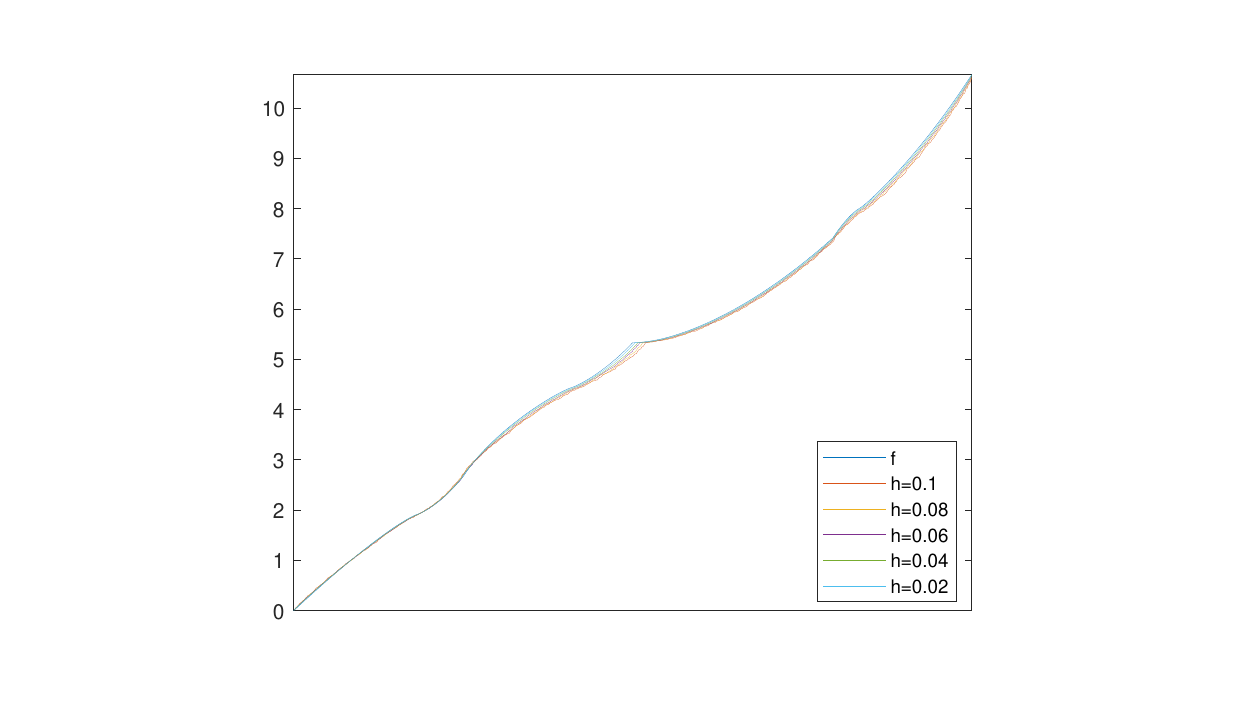} \vskip -0.5cm
  \caption{Eigenvalues distribution of $C_{n,t}$ for different $h$ values  together with the sampling of $f(\theta_1,\theta_2)=f_{{P}_2}(\theta_{1},\theta_{2})$ and $t=2$.}
\label{P2_Bnt_small_t2}
\end{figure}
\begin{figure}
\centering
  \includegraphics[width=\textwidth]{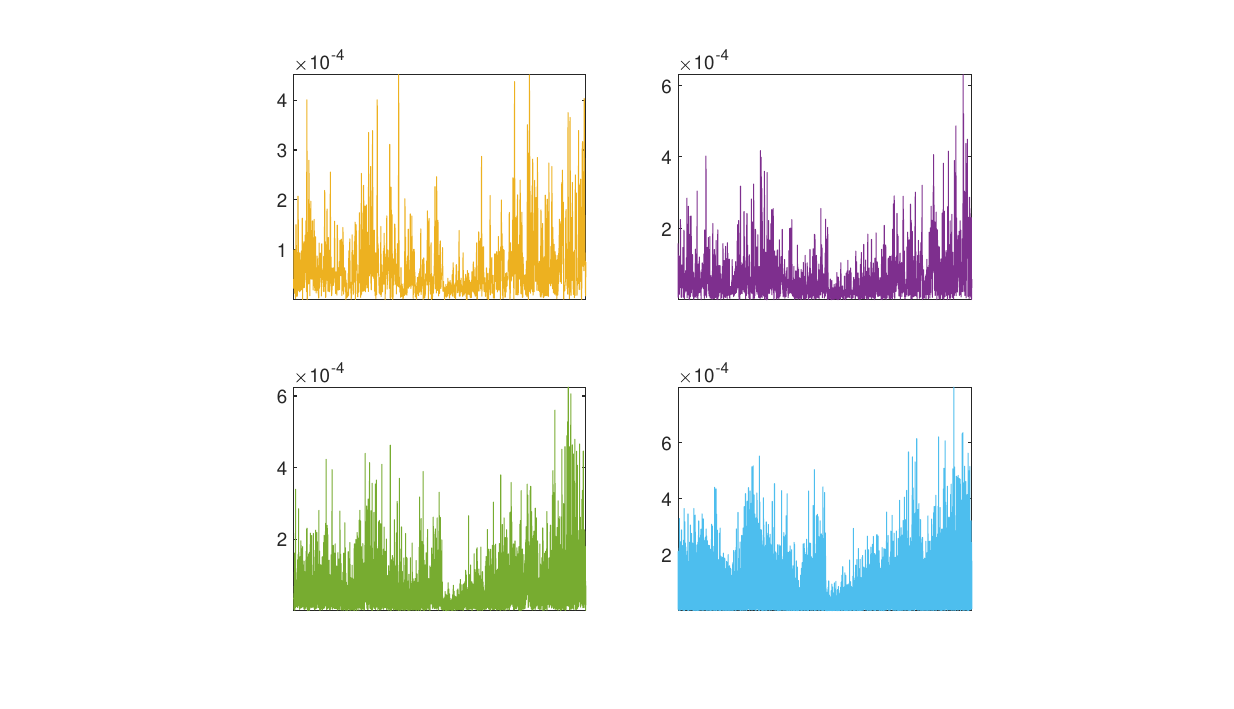} \vskip -0.5cm
  \caption{Minimal distance of eigenvalues of $C_{n,t}$ from $f(\theta_1,\theta_2)=f_{{P}_2}(\theta_{1},\theta_{2})$ and $t=2$ for different $h$ values.}
\label{P2_Bnt_small_t2_errore}
\end{figure}
%--------------------------------------------------------------\begin{figure}
\begin{figure}
\centering
  \includegraphics[width=\textwidth]{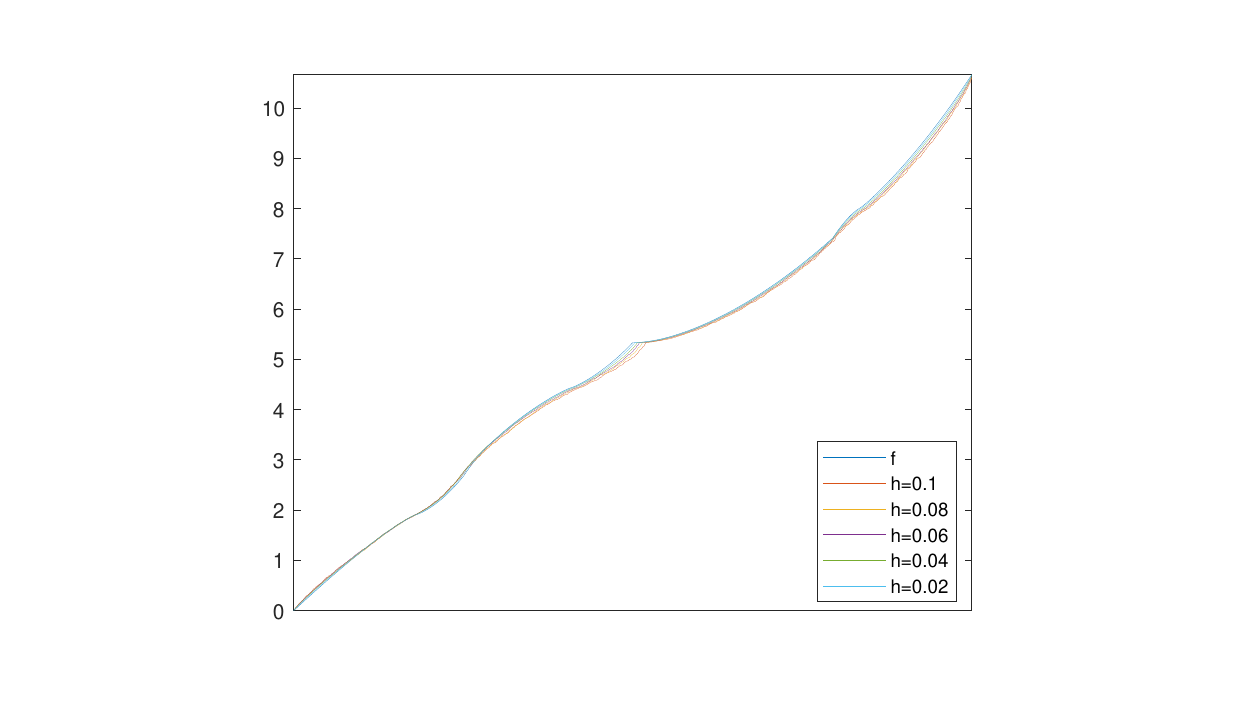} \vskip -0.5cm
  \caption{Eigenvalues distribution of $C_{n,t}$ for different $h$ values  together with the sampling of $f(\theta_1,\theta_2)=f_{{P}_2}(\theta_{1},\theta_{2})$ and $t=4$.}
\label{P2_Bnt_small_t4}
\end{figure}
\begin{figure}
\centering
  \includegraphics[width=\textwidth]{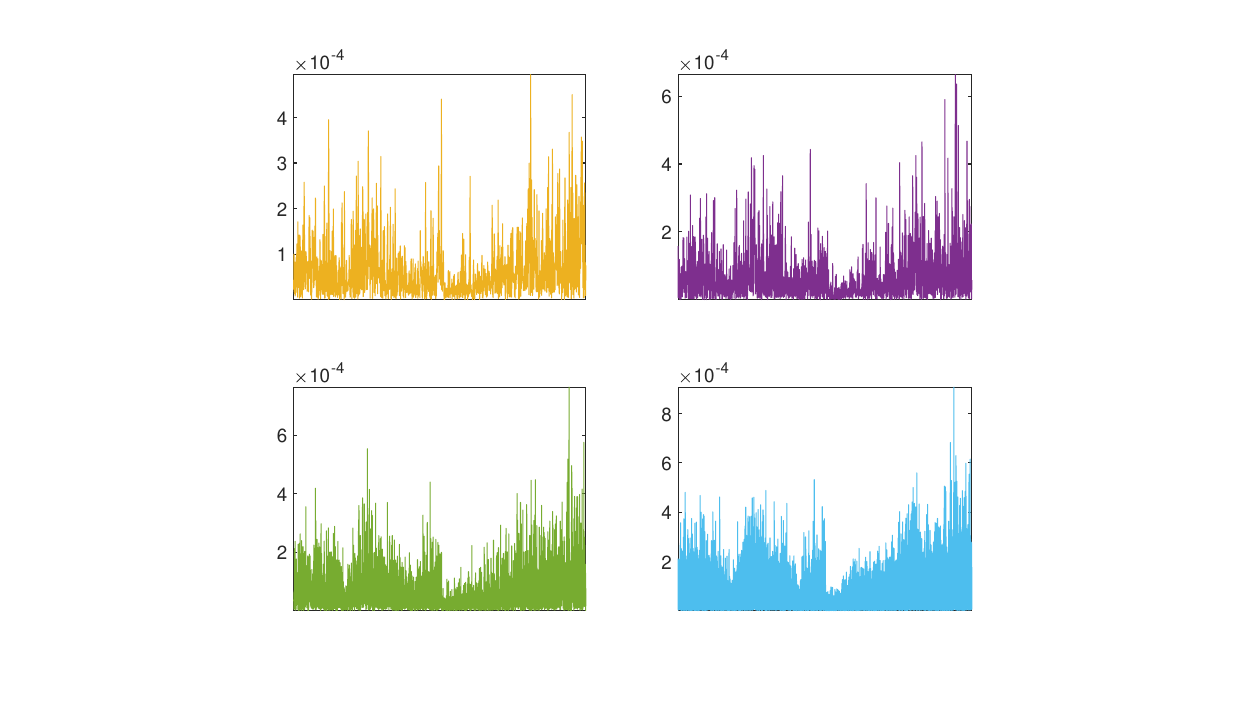} \vskip -0.5cm
  \caption{Minimal distance of eigenvalues of $C_{n,t}$ from $f(\theta_1,\theta_2)=f_{{P}_2}(\theta_{1},\theta_{2})$ and $t=4$ for different $h$ values.}
\label{P2_Bnt_small_t4_errore}
\end{figure}
%--------------------------------------------------------------
\begin{figure}
\centering
  \includegraphics[width=\textwidth]{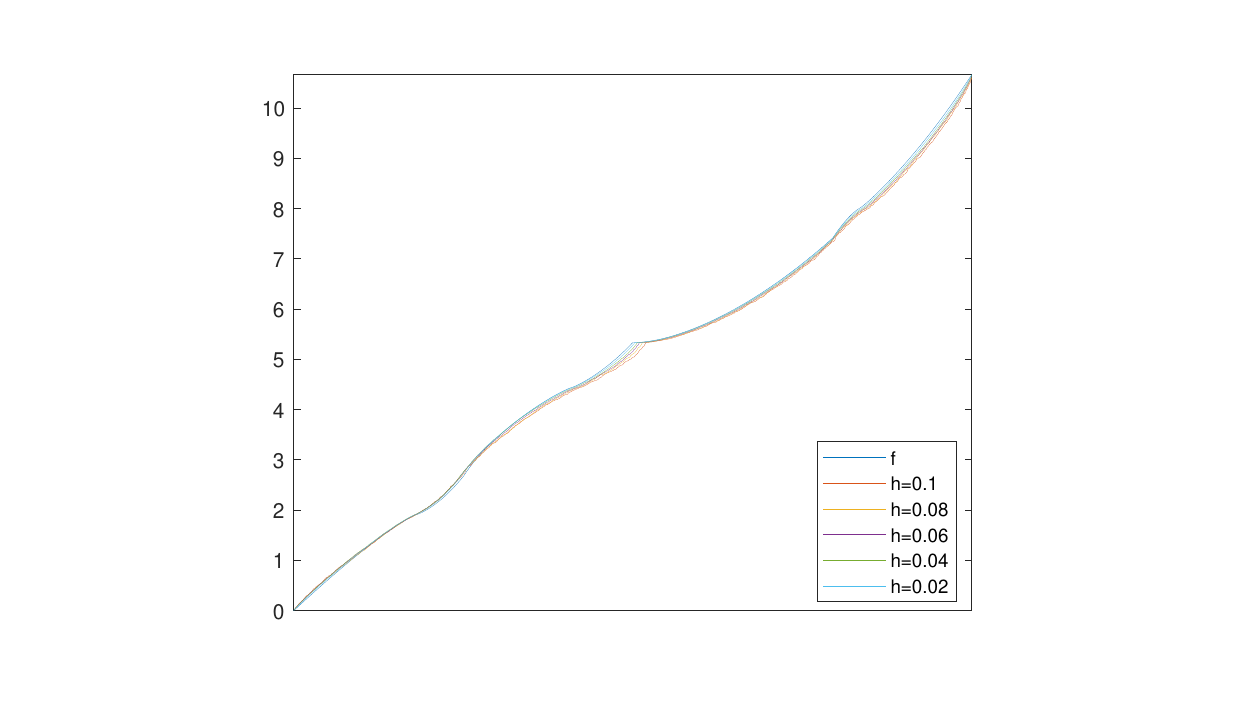} \vskip -0.5cm
  \caption{Eigenvalues distribution of $C_{n,t}$ for different $h$ values  together with the sampling of $f(\theta_1,\theta_2)=f_{{P}_2}(\theta_{1},\theta_{2})$ and $t=6$.}
\label{P2_Bnt_small_t6}
\end{figure}
\begin{figure}
\centering
  \includegraphics[width=\textwidth]{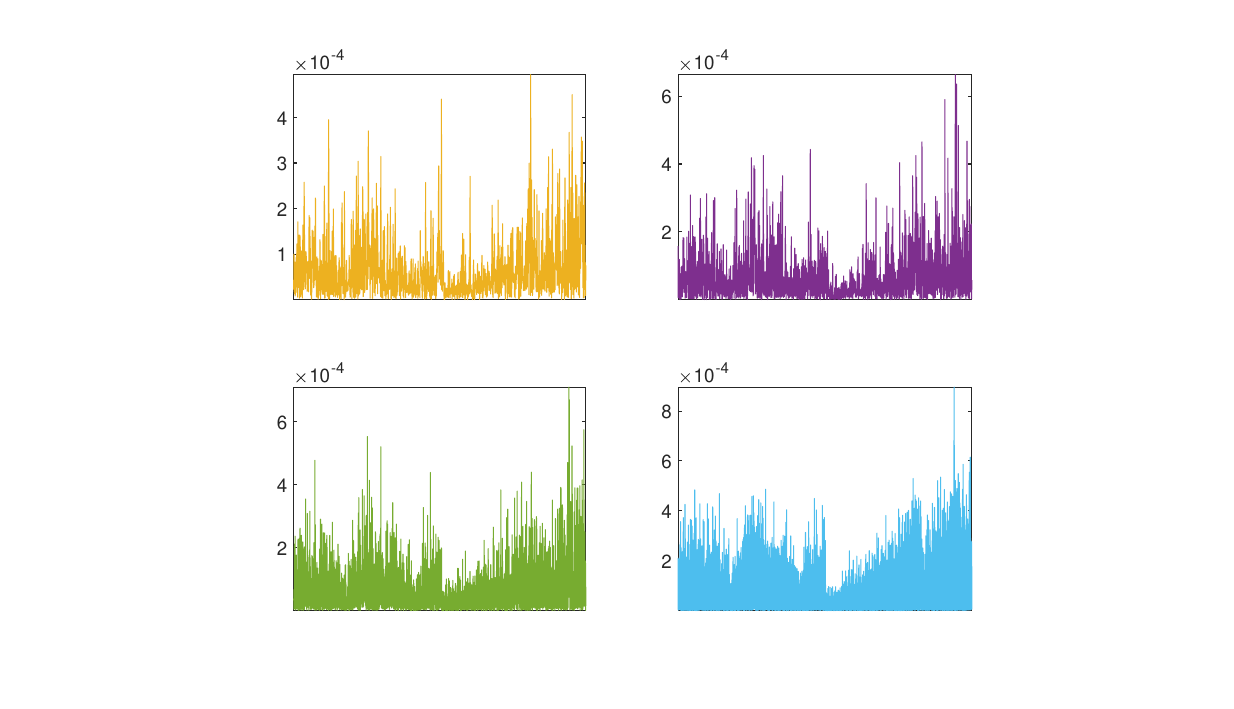} \vskip -0.5cm
  \caption{Minimal distance of eigenvalues of $C_{n,t}$ from $f(\theta_1,\theta_2)=f_{{P}_2}(\theta_{1},\theta_{2})$ and $t=6$ for different $h$ values.}
\label{P2_Bnt_small_t6_errore}
\end{figure}
%--------------------------------------------------------------
%-------------------------------------------------------------- %
\begin{figure}
\centering
  \includegraphics[width=\textwidth]{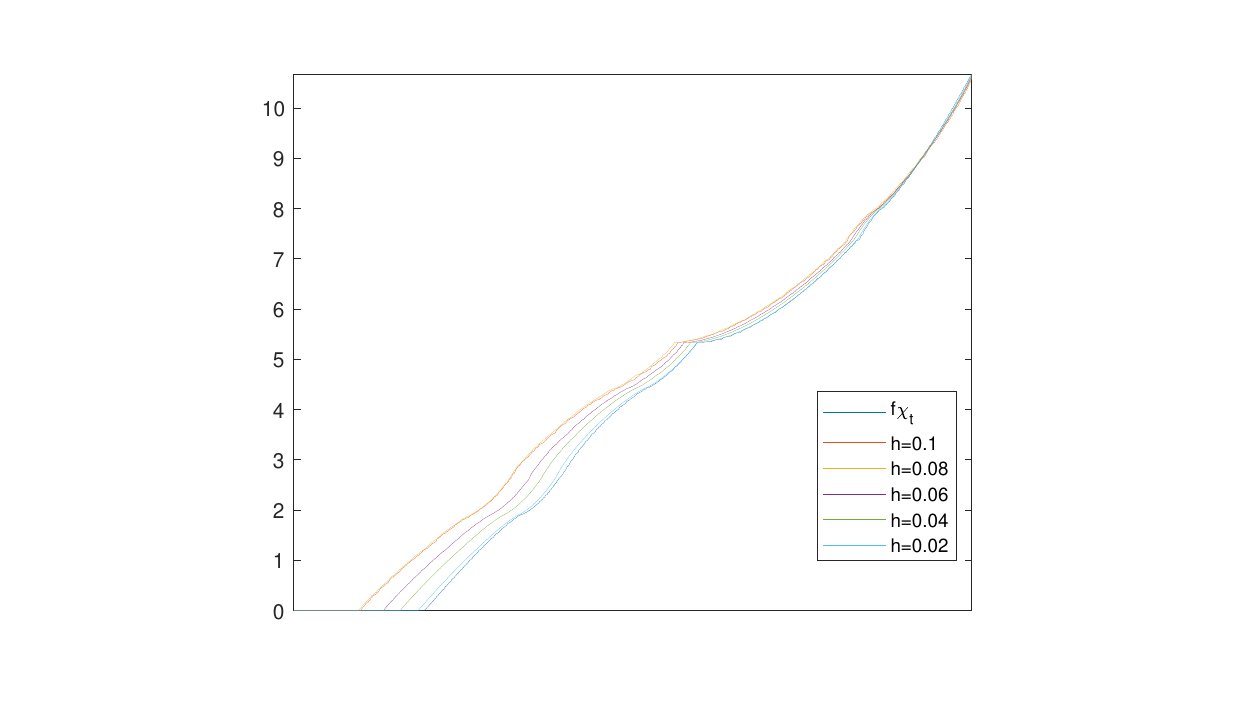} \vskip -0.5cm
  \caption{Eigenvalues distribution of $B_{n,t}$ for different $h$ values  together with the sampling of $f(\theta_1,\theta_2)=f_{{P}_2}(\theta_{1},\theta_{2})$ and $t=2$.}
\label{P2_Bnt_full_t2}
\end{figure}
\begin{figure}
\centering
  \includegraphics[width=\textwidth]{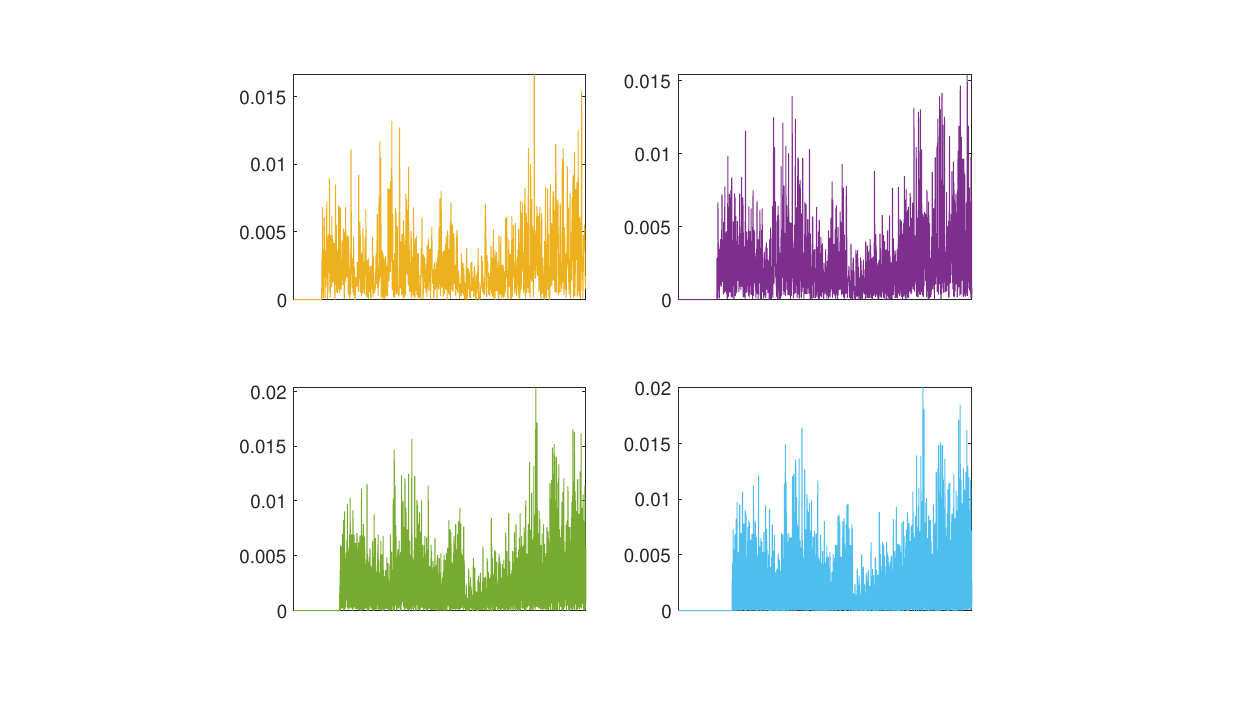} \vskip -0.5cm
  \caption{Minimal distance of eigenvalues of $B_{n,t}$ from $f(\theta_1,\theta_2)=f_{{P}_2}(\theta_{1},\theta_{2})$ and $t=2$ for different $h$ values.}
\label{P2_Bnt_full_t2_errore}
\end{figure}
%--------------------------------------------------------------\begin{figure}
\begin{figure}
\centering
  \includegraphics[width=\textwidth]{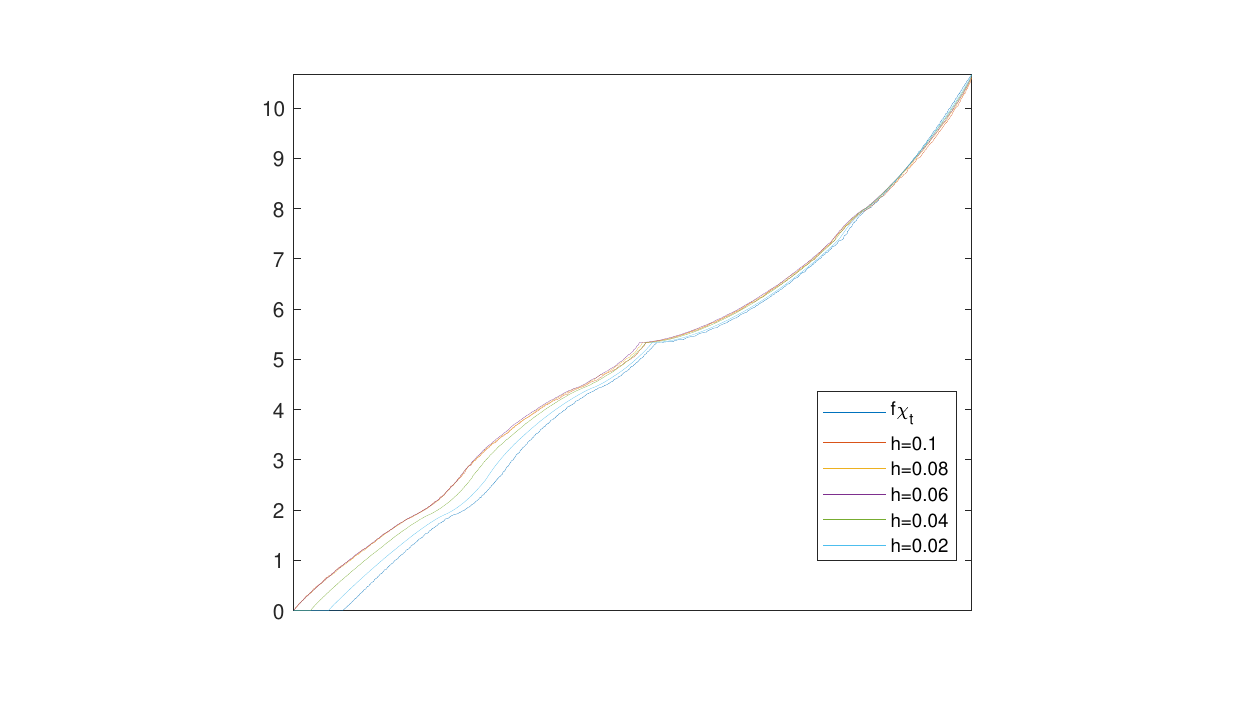} \vskip -0.5cm
  \caption{Eigenvalues distribution of $B_{n,t}$ for different $h$ values  together with the sampling of $f(\theta_1,\theta_2)=f_{{P}_2}(\theta_{1},\theta_{2})$ and $t=4$.}
\label{P2_Bnt_full_t4}
\end{figure}
\begin{figure}
\centering
  \includegraphics[width=\textwidth]{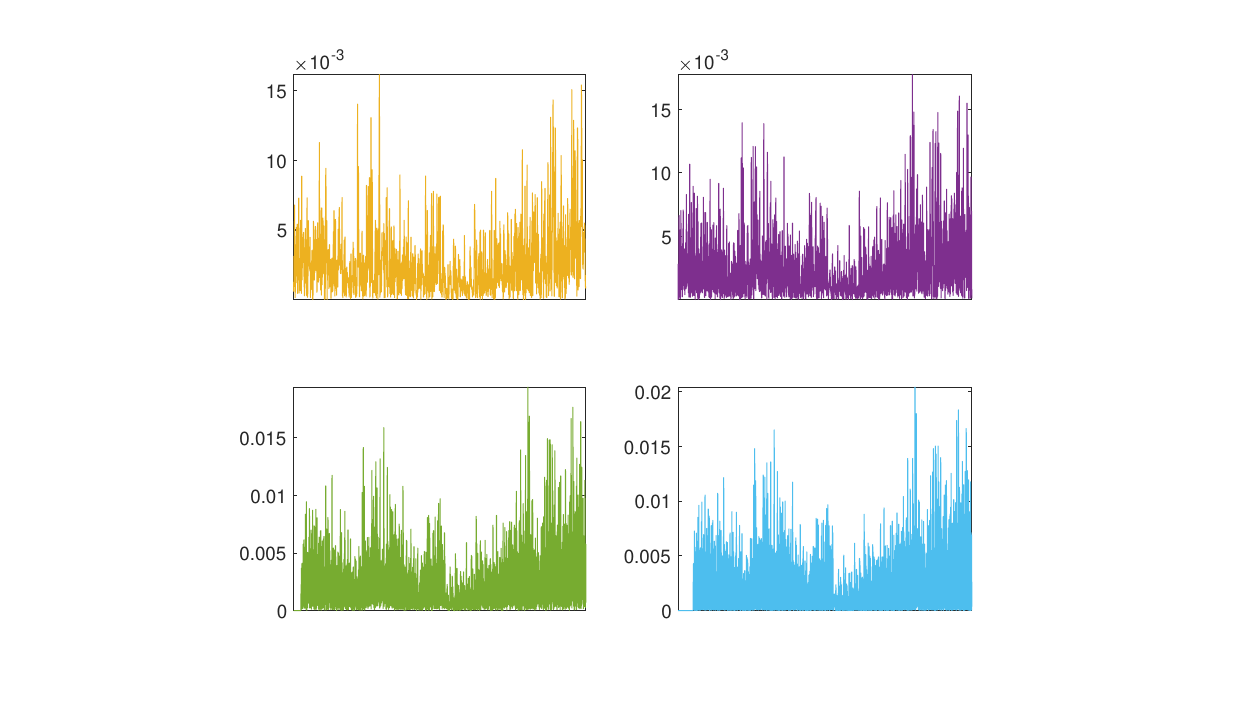} \vskip -0.5cm
  \caption{Minimal distance of eigenvalues of $B_{n,t}$ from $f(\theta_1,\theta_2)=f_{{P}_2}(\theta_{1},\theta_{2})$ and $t=4$ for different $h$ values.}
\label{P2_Bnt_full_t4_errore}
\end{figure}
%--------------------------------------------------------------
\begin{figure}
\centering
  \includegraphics[width=\textwidth]{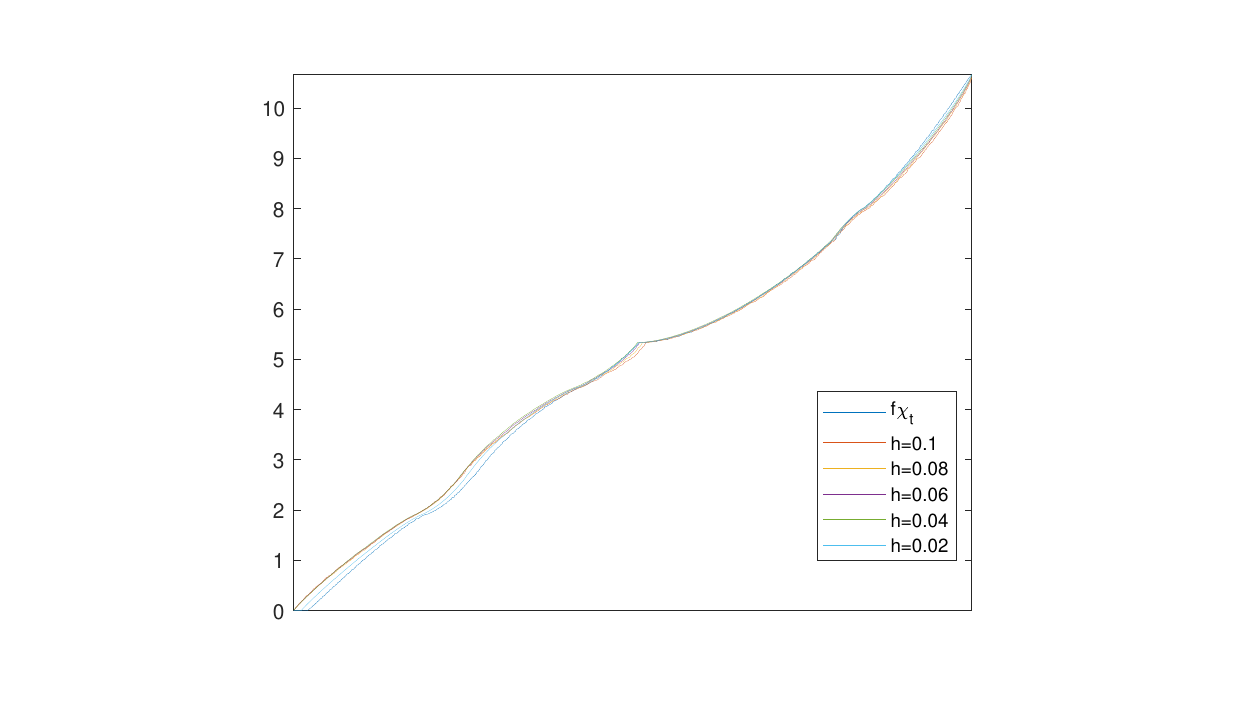} \vskip -0.5cm
  \caption{Eigenvalues distribution of $B_{n,t}$ for different $h$ values  together with the sampling of $f(\theta_1,\theta_2)=f_{{P}_2}(\theta_{1},\theta_{2})$ and $t=6$.}
\label{P2_Bnt_full_t6}
\end{figure}
\begin{figure}
\centering
  \includegraphics[width=\textwidth]{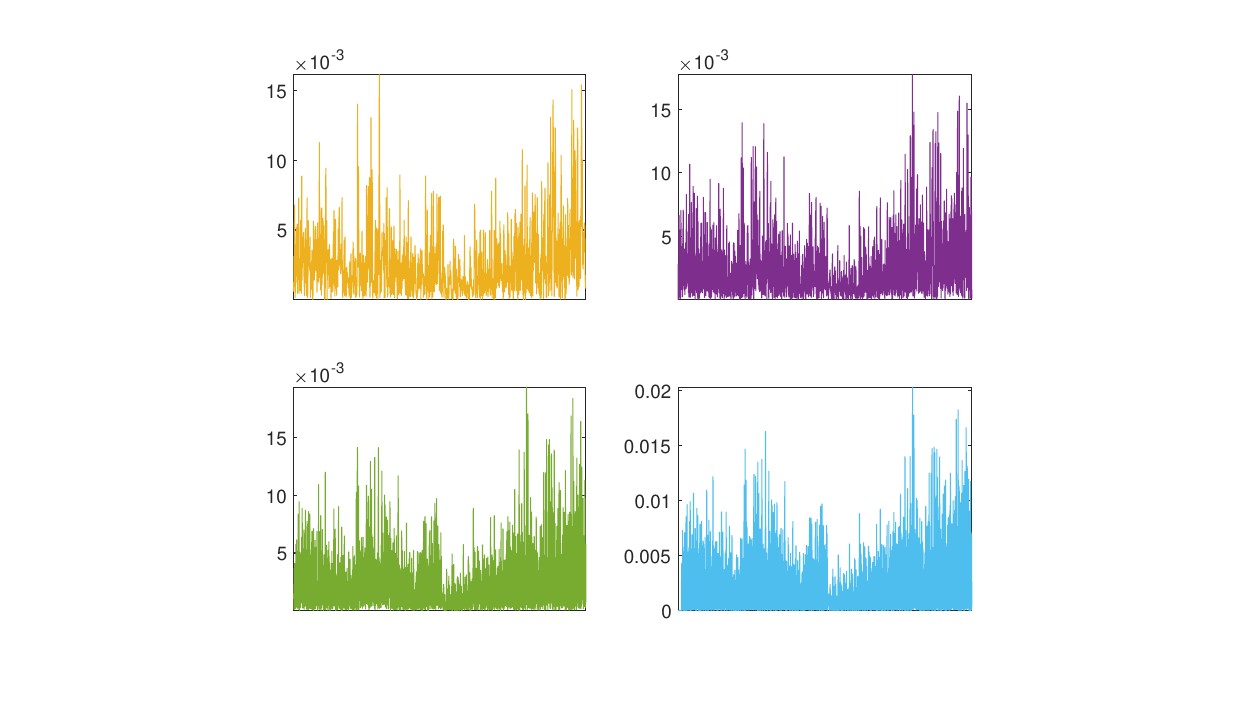} \vskip -0.5cm
  \caption{Minimal distance of eigenvalues of $B_{n,t}$ from $f(\theta_1,\theta_2)=f_{{P}_2}(\theta_{1},\theta_{2})$ and $t=6$ for different $h$ values.}
\label{P2_Bnt_full_t6_errore}
\end{figure}
%--------------------------------------------------------------
\begin{figure}
\centering
  \includegraphics[width=\textwidth]{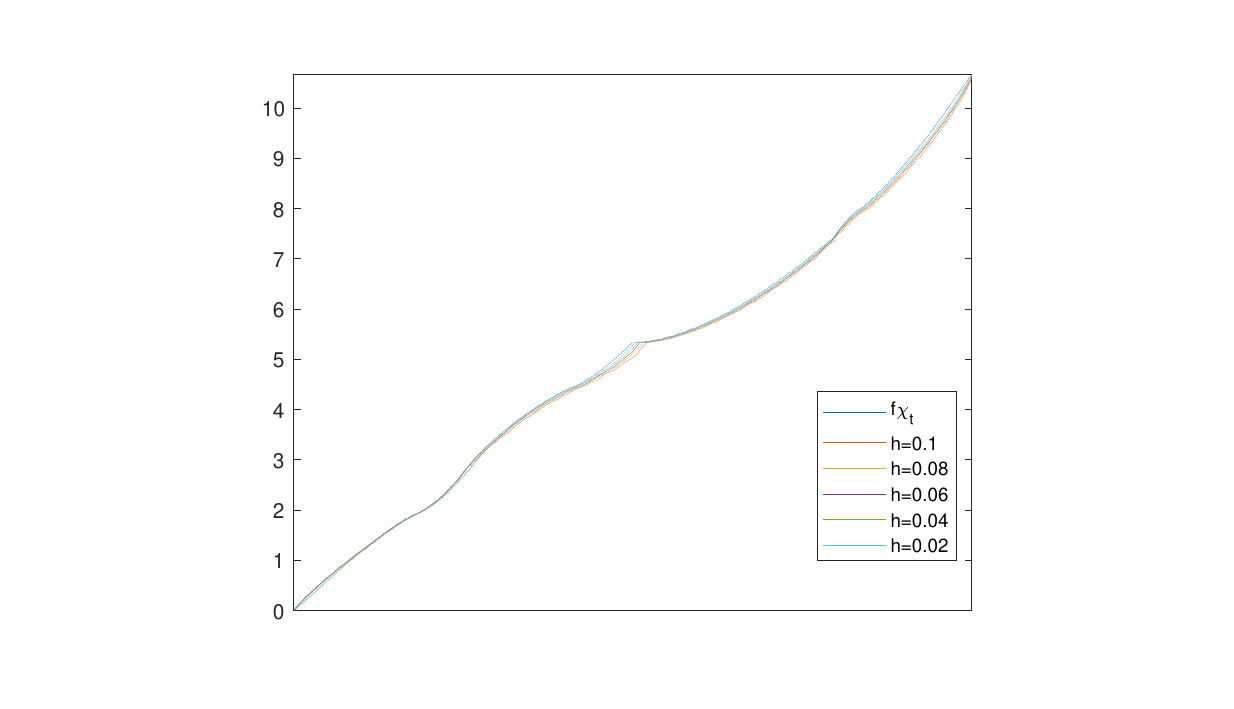} \vskip -0.5cm
  \caption{Eigenvalues distribution of $B_{n,t}$ for different $h$ values  together with the sampling of $f(\theta_1,\theta_2)=f_{{P}_2}(\theta_{1},\theta_{2})$ and $t=8$.}
\label{P2_Bnt_full_t8}
\end{figure}
\begin{figure}
\centering
  \includegraphics[width=\textwidth]{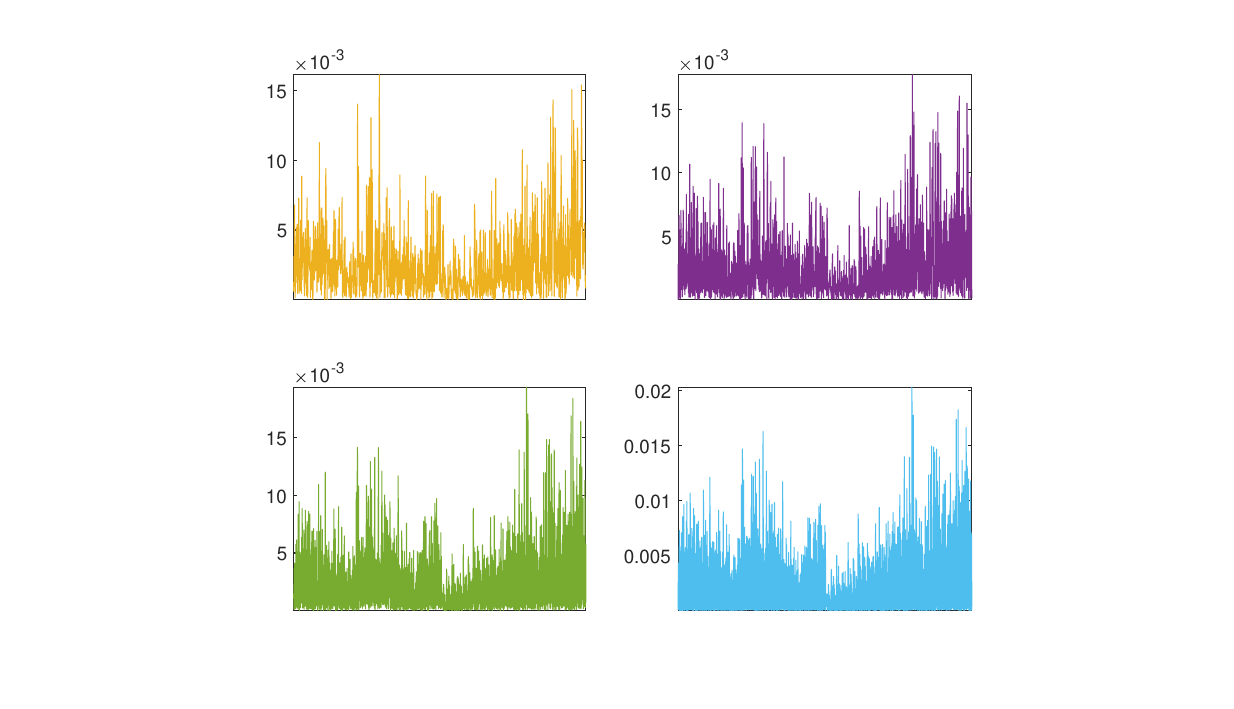} \vskip -0.5cm
  \caption{Minimal distance of eigenvalues of $B_{n,t}$ from $f(\theta_1,\theta_2)=f_{{P}_2}(\theta_{1},\theta_{2})$ and $t=8$ for different $h$ values.}
\label{P2_Bnt_full_t8_errore}
\end{figure}
%--------------------------------------------------------------
%-------------------------------------------------------------- %

%//////////////////////////////////////////////////

\section{$P_1$ - $f(\theta_1,\theta_2)=(2-2\cos \theta_1)+(2-2\cos \theta_2)$ variable coefficient}
$P_1$: $f(\theta_1,\theta_2,x,y)=[(2-2\cos \theta_1)+(2-2\cos \theta_2)]a(x,y)$\\
$a(x,y)= ( 10 + x^2 + 2y^2+ \sin^2(x+y)) / ( 1 + x^2 + y^2)$ \\

\newpage
%-------------------------------------------------------------- %
%-------------------------------------------------------------- %
\begin{figure}
\centering
  \includegraphics[width=\textwidth]{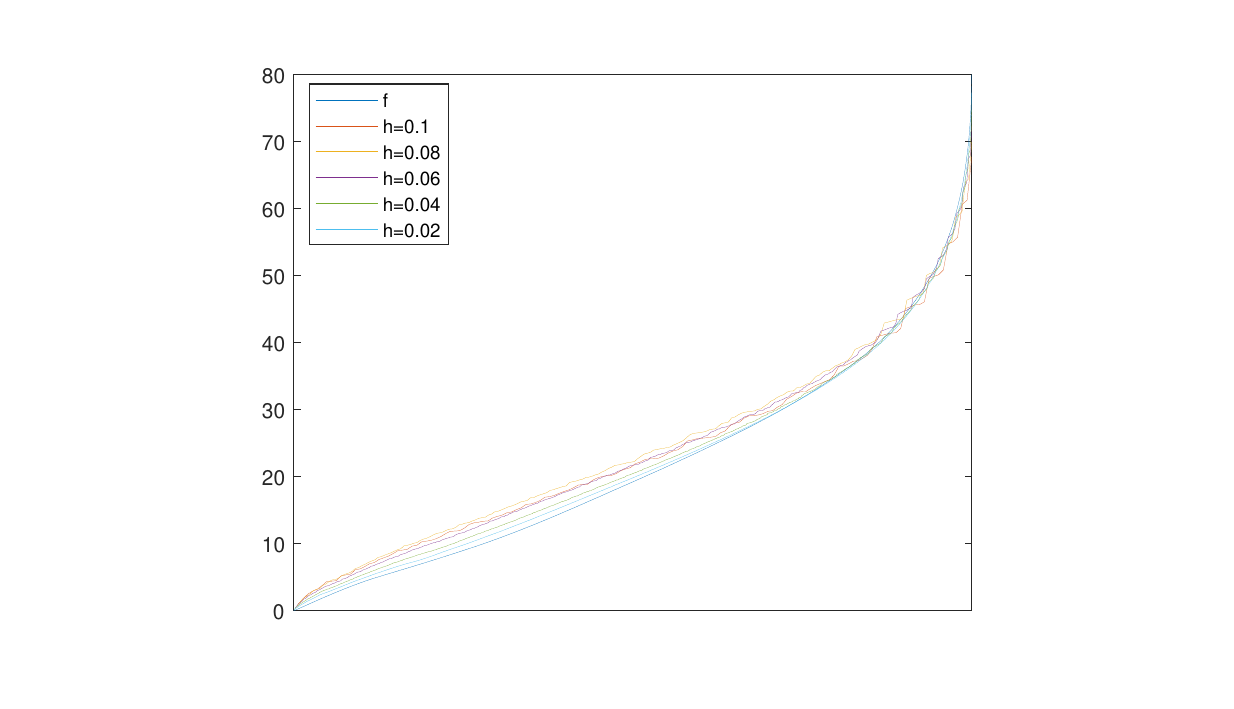} \vskip -0.5cm
  \caption{Eigenvalues distribution of $A_n(a)$ for different $h$ values  together with the sampling of $f(\theta_1,\theta_2,x,y)=[(2-2\cos \theta_1)+(2-2\cos \theta_2)]a(x,y)$.}
\label{FDcoeff_A}
\end{figure}
\begin{figure}
\centering
  \includegraphics[width=\textwidth]{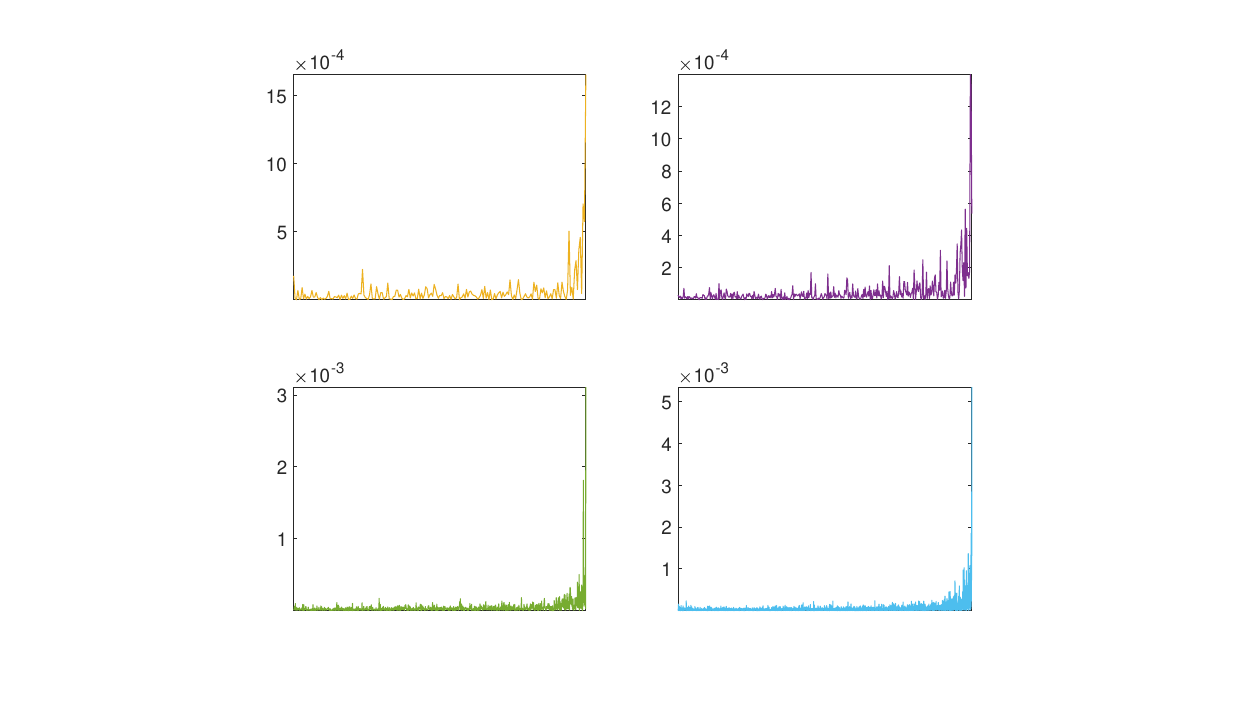} \vskip -0.5cm
  \caption{Minimal distance of eigenvalues of $A_n(a)$ from $f(\theta_1,\theta_2,x,y)=[(2-2\cos \theta_1)+(2-2\cos \theta_2)]a(x,y)$ for different $h$ values.}
\label{FDcoeff_A_errore}
\end{figure}
%
%-------------------------------------------------------------- %
%-------------------------------------------------------------- %
\begin{figure}
\centering
  \includegraphics[width=\textwidth]{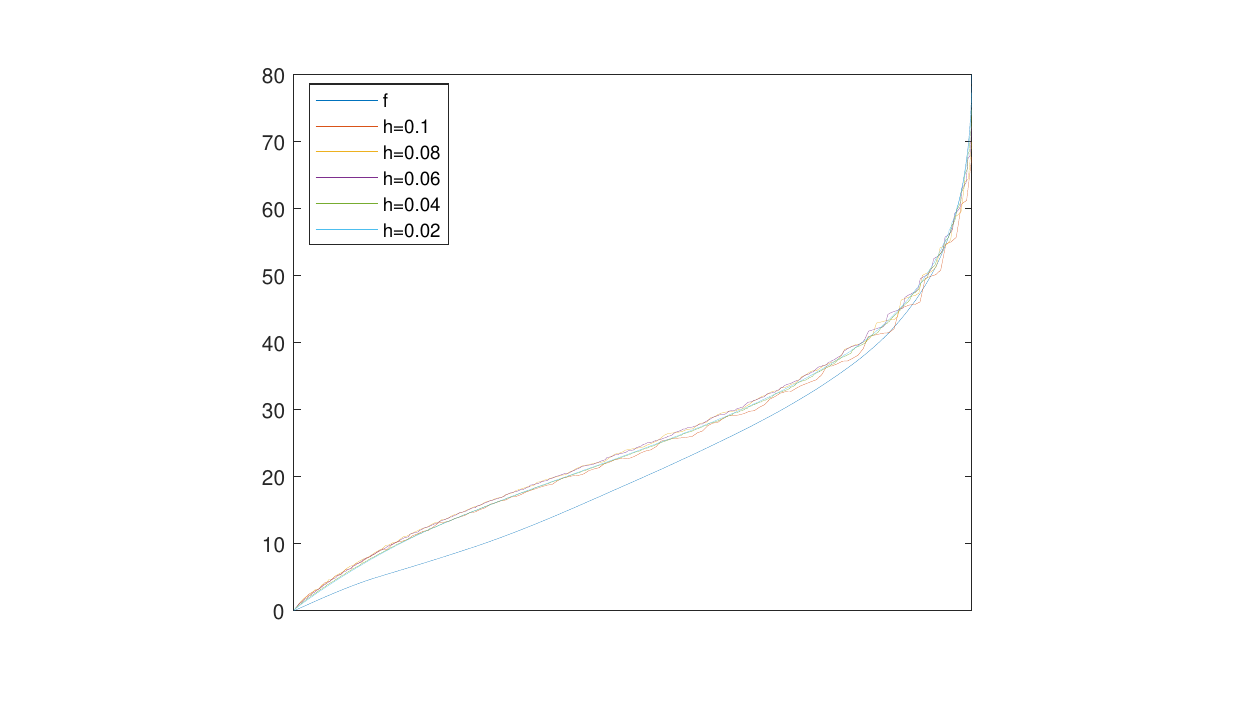} \vskip -0.5cm
  \caption{Eigenvalues distribution of $C_{n,t}(a)$ for different $h$ values  together with the sampling of $f(\theta_1,\theta_2,x,y)=[(2-2\cos \theta_1)+(2-2\cos \theta_2)]a(x,y)$ and $t=2$.}
\label{FDcoeff_Bnt_small_t2}
\end{figure}
\begin{figure}
\centering
  \includegraphics[width=\textwidth]{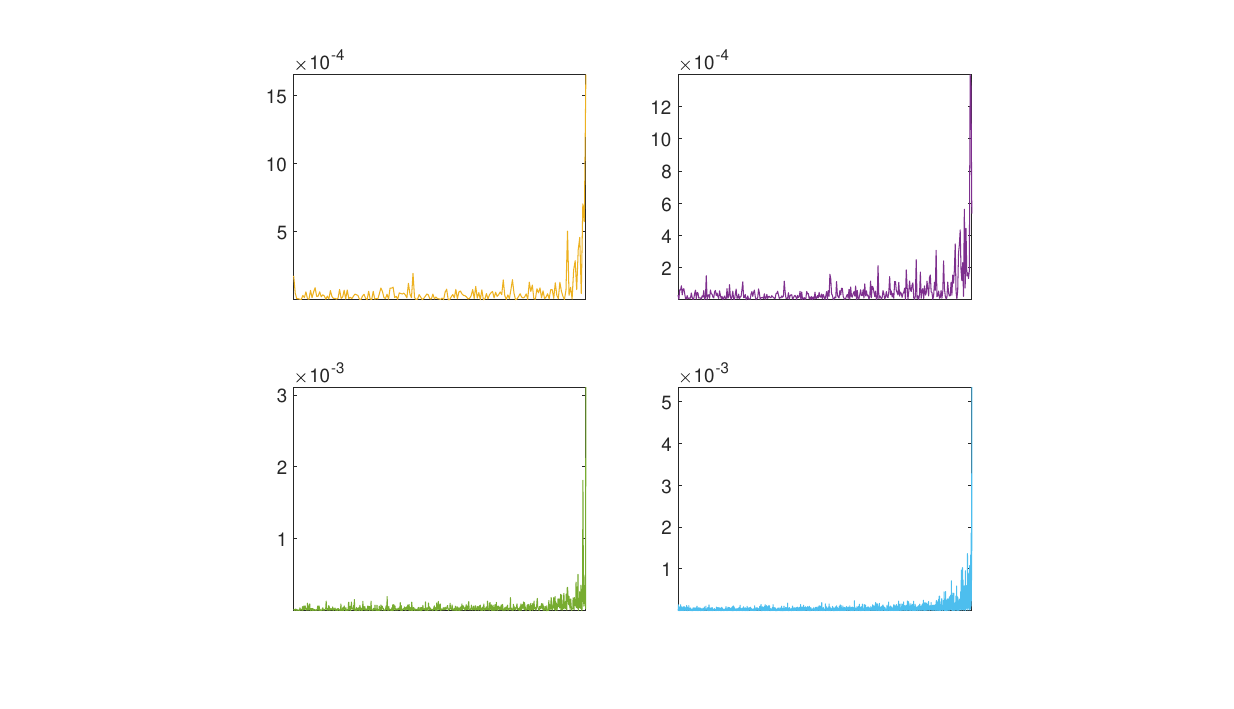} \vskip -0.5cm
  \caption{Minimal distance of eigenvalues of $C_{n,t}(a)$ from $f(\theta_1,\theta_2,x,y)=[(2-2\cos \theta_1)+(2-2\cos \theta_2)]a(x,y)$ and $t=2$ for different $h$ values.}
\label{FDcoeff_Bnt_small_t2_errore}
\end{figure}
%--------------------------------------------------------------\begin{figure}
\begin{figure}
\centering
  \includegraphics[width=\textwidth]{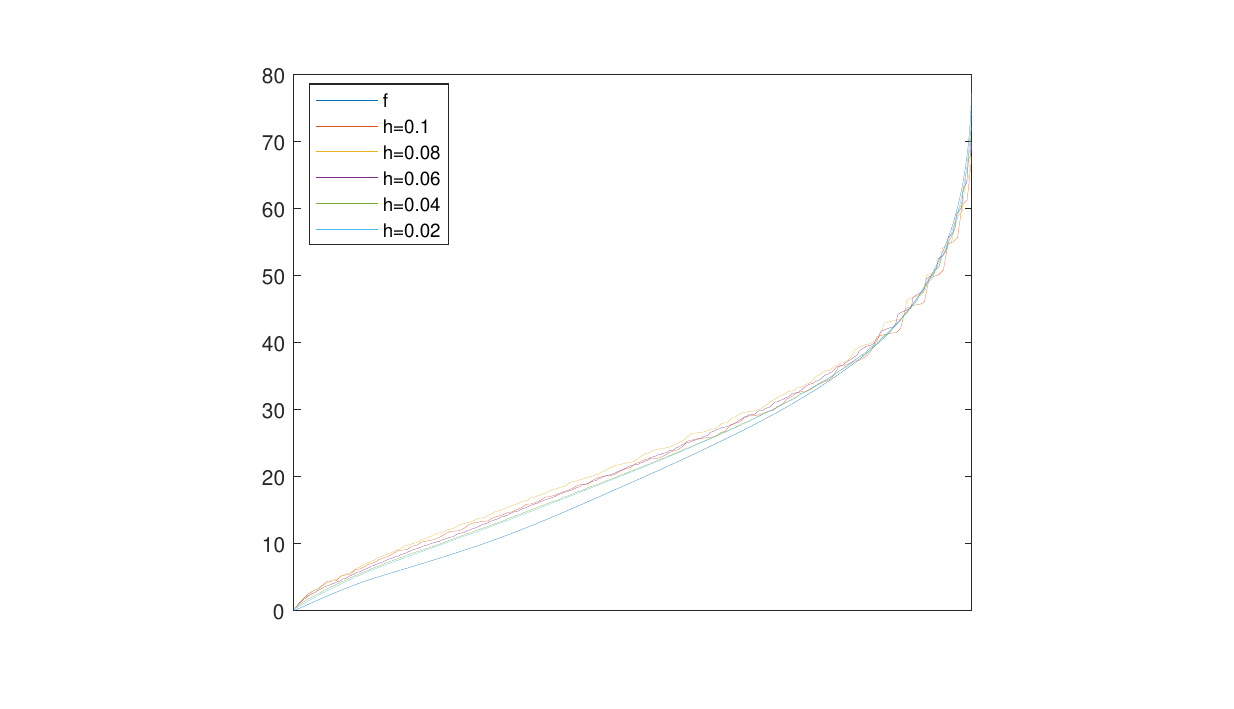} \vskip -0.5cm
  \caption{Eigenvalues distribution of $C_{n,t}(a)$ for different $h$ values  together with the sampling of $f(\theta_1,\theta_2,x,y)=[(2-2\cos \theta_1)+(2-2\cos \theta_2)]a(x,y)$ and $t=4$.}
\label{FDcoeff_Bnt_small_t4}
\end{figure}
\begin{figure}
\centering
  \includegraphics[width=\textwidth]{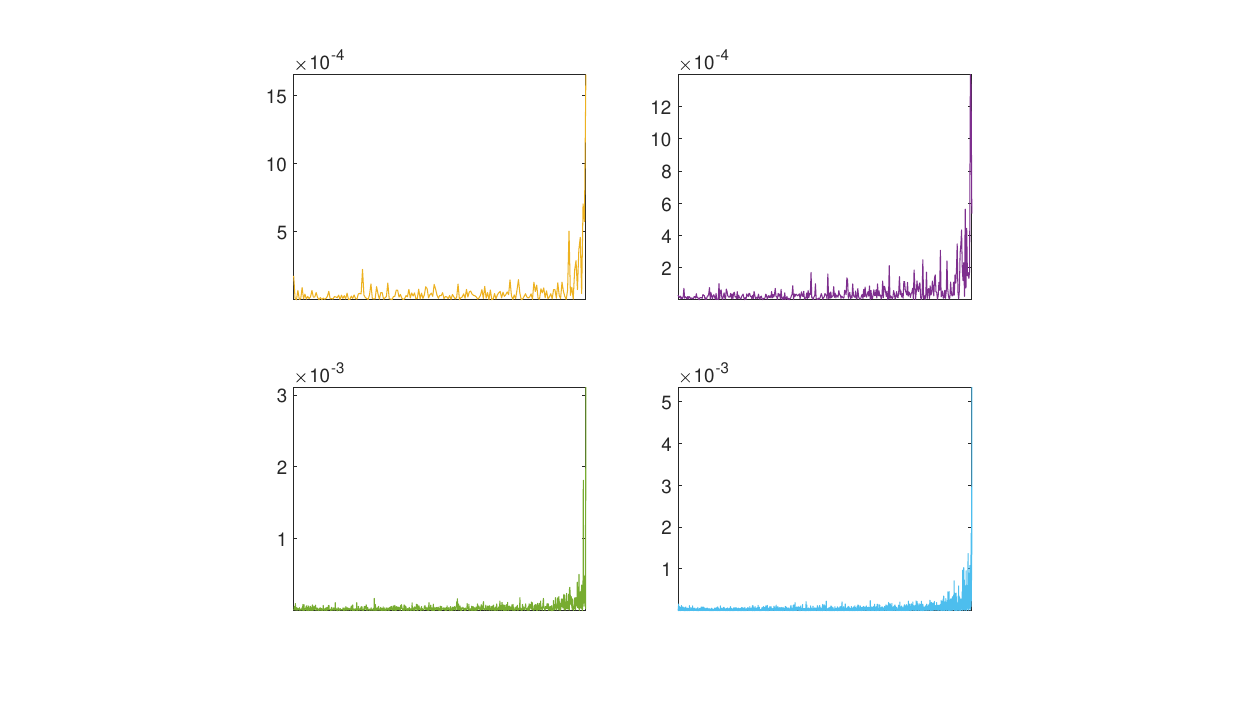} \vskip -0.5cm
  \caption{Minimal distance of eigenvalues of $C_{n,t}(a)$ from $f(\theta_1,\theta_2,x,y)=[(2-2\cos \theta_1)+(2-2\cos \theta_2)]a(x,y)$ and $t=4$ for different $h$ values.}
\label{FDcoeff_Bnt_small_t4_errore}
\end{figure}
%--------------------------------------------------------------
\begin{figure}
\centering
  \includegraphics[width=\textwidth]{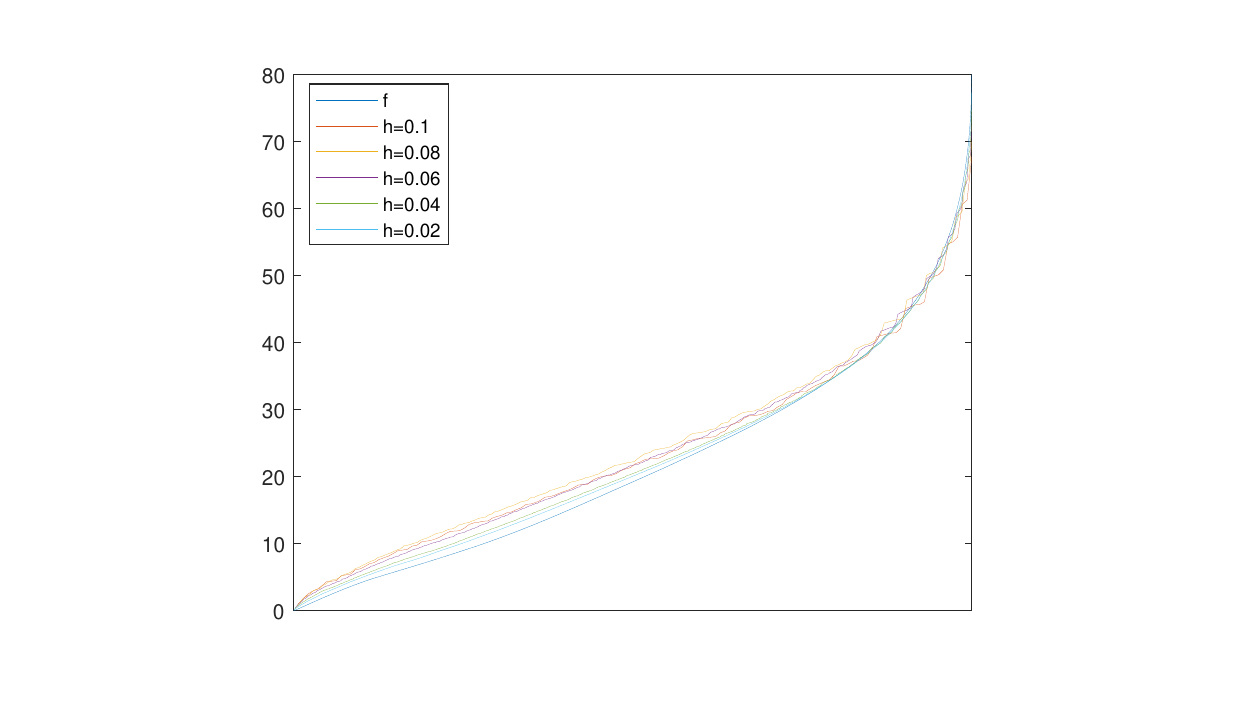} \vskip -0.5cm
  \caption{Eigenvalues distribution of $C_{n,t}(a)$ for different $h$ values  together with the sampling of $f(\theta_1,\theta_2,x,y)=[(2-2\cos \theta_1)+(2-2\cos \theta_2)]a(x,y)$ and $t=6$.}
\label{FDcoeff_Bnt_small_t6}
\end{figure}
\begin{figure}
\centering
  \includegraphics[width=\textwidth]{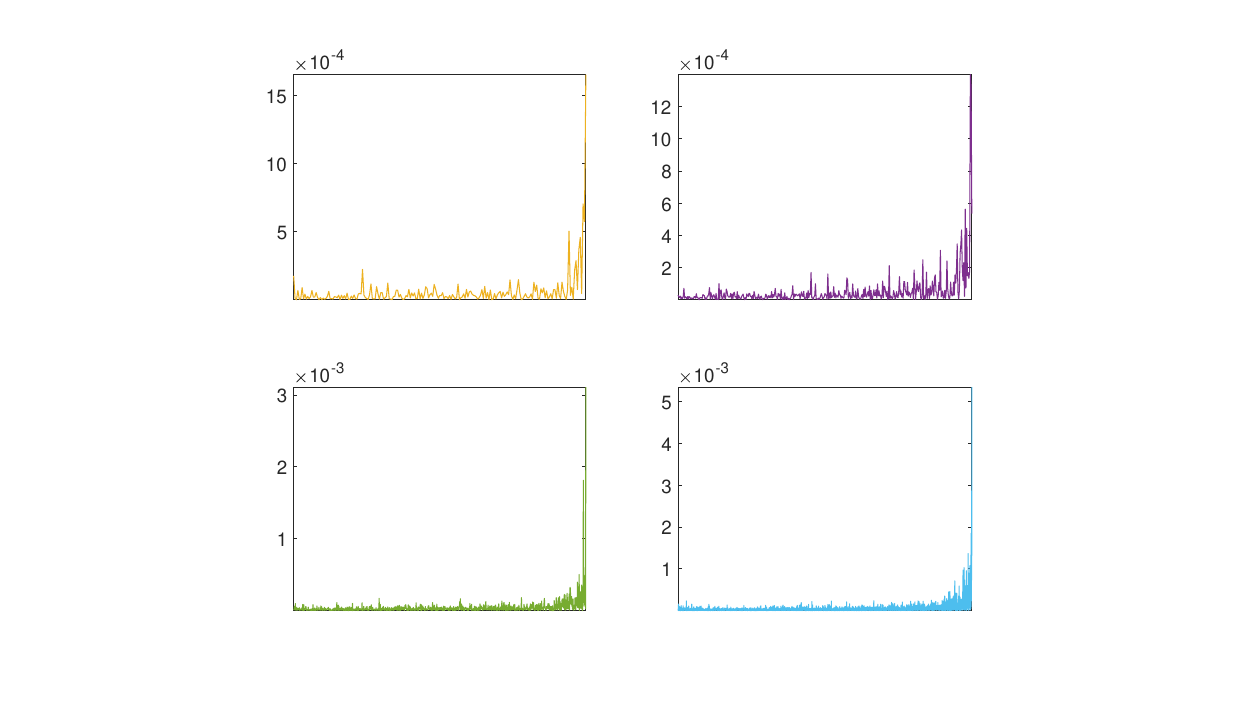} \vskip -0.5cm
  \caption{Minimal distance of eigenvalues of $C_{n,t}(a)$ from $f(\theta_1,\theta_2,x,y)=[(2-2\cos \theta_1)+(2-2\cos \theta_2)]a(x,y)$ and $t=6$ for different $h$ values.}
\label{FDcoeff_Bnt_small_t6_errore}
\end{figure}
%--------------------------------------------------------------
%-------------------------------------------------------------- %
\begin{figure}
\centering
  \includegraphics[width=\textwidth]{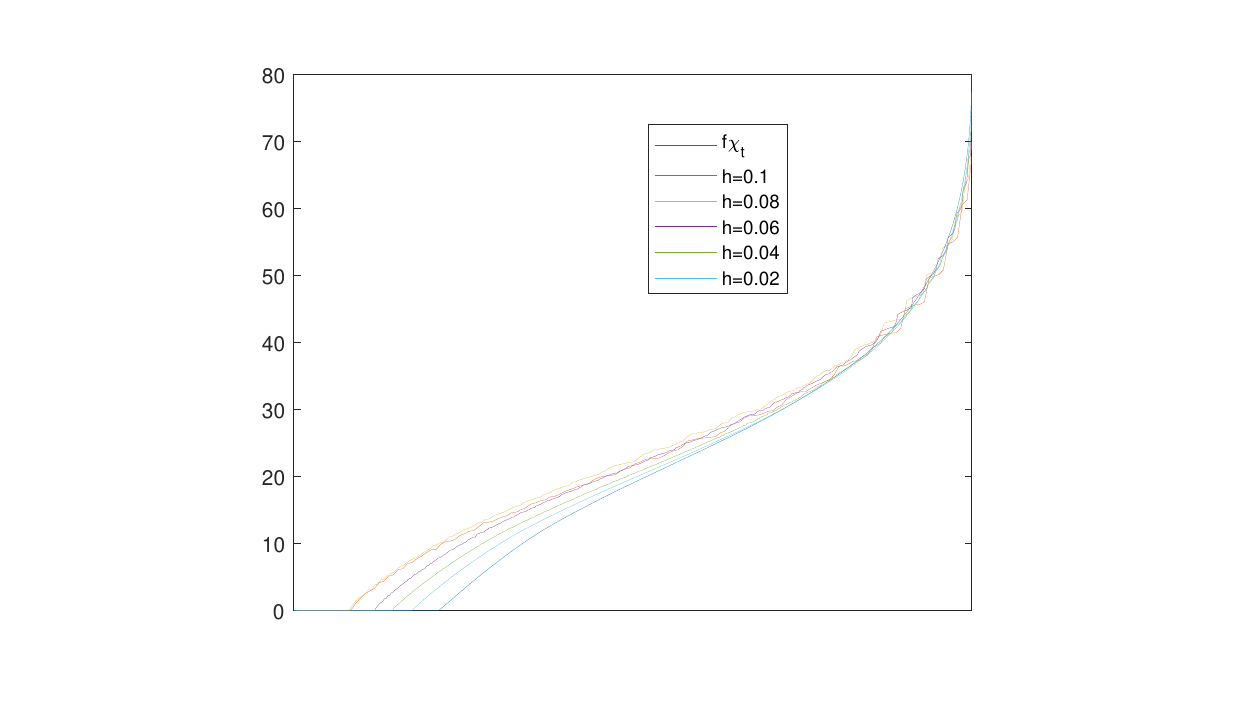} \vskip -0.5cm
  \caption{Eigenvalues distribution of $B_{n,t}(a)$ for different $h$ values  together with the sampling of $f(\theta_1,\theta_2,x,y)=[(2-2\cos \theta_1)+(2-2\cos \theta_2)]a(x,y)$ and $t=2$.}
\label{FDcoeff_Bnt_full_t2}
\end{figure}
\begin{figure}
\centering
  \includegraphics[width=\textwidth]{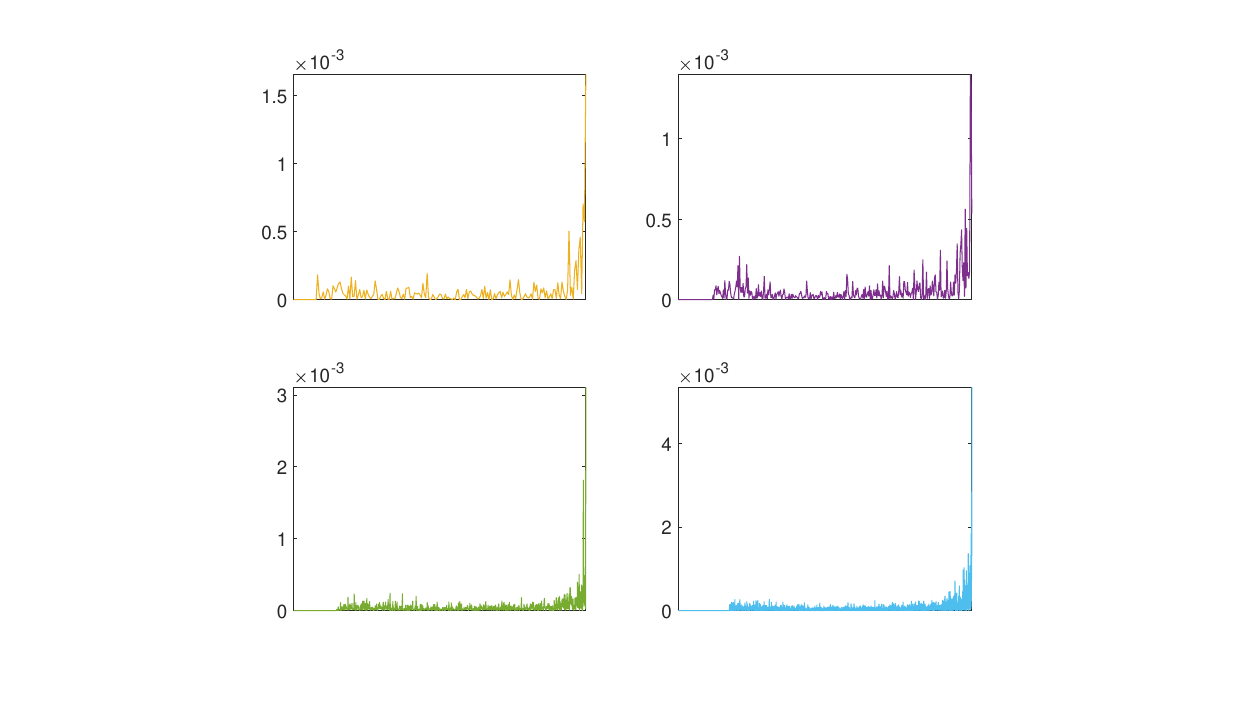} \vskip -0.5cm
  \caption{Minimal distance of eigenvalues of $B_{n,t}(a)$ from $f(\theta_1,\theta_2,x,y)=[(2-2\cos \theta_1)+(2-2\cos \theta_2)]a(x,y)$ and $t=2$ for different $h$ values.}
\label{FDcoeff_Bnt_full_t2_errore}
\end{figure}
%--------------------------------------------------------------\begin{figure}
\begin{figure}
\centering
  \includegraphics[width=\textwidth]{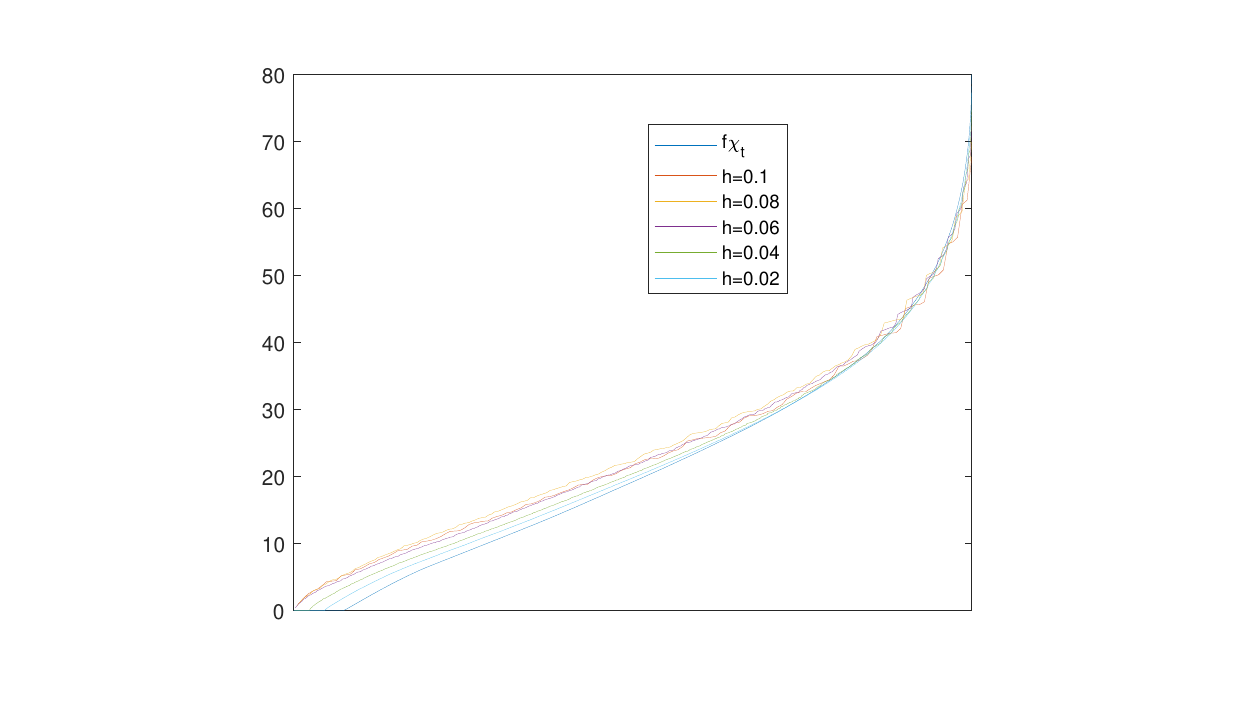} \vskip -0.5cm
  \caption{Eigenvalues distribution of $B_{n,t}(a)$ for different $h$ values  together with the sampling of $f(\theta_1,\theta_2,x,y)=[(2-2\cos \theta_1)+(2-2\cos \theta_2)]a(x,y)$ and $t=4$.}
\label{FDcoeff_Bnt_full_t4}
\end{figure}
\begin{figure}
\centering
  \includegraphics[width=\textwidth]{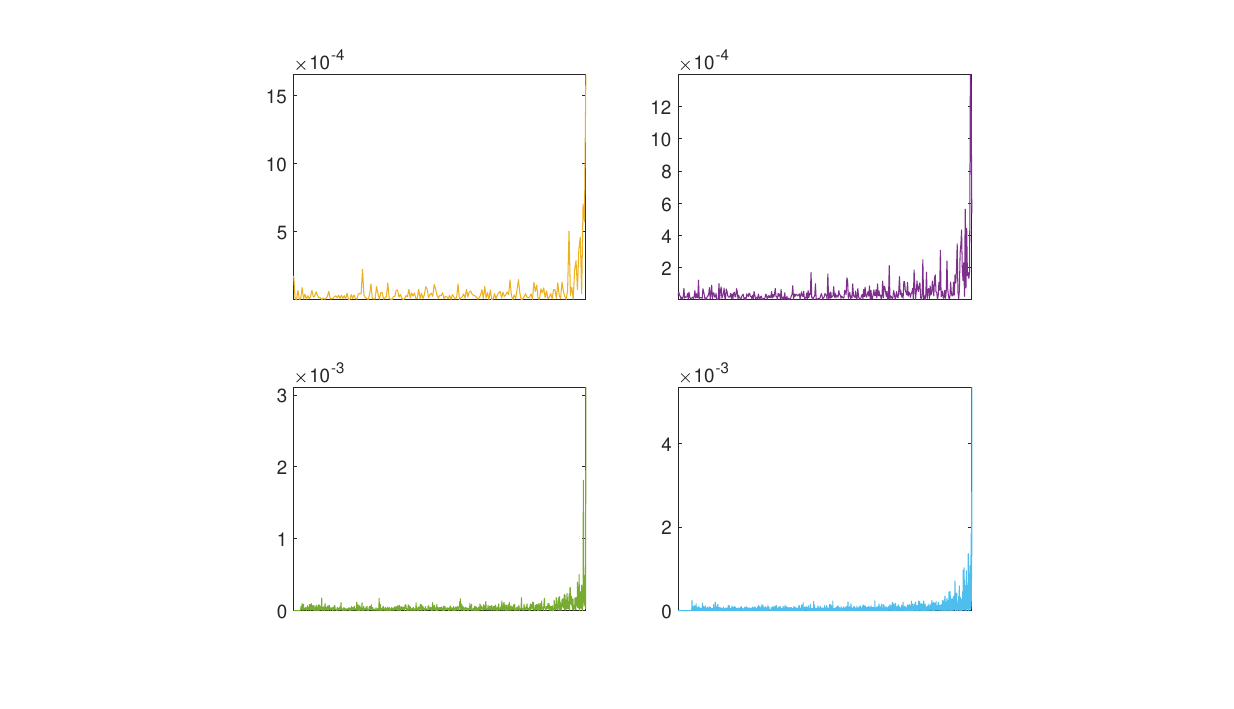} \vskip -0.5cm
  \caption{Minimal distance of eigenvalues of $B_{n,t}(a)$ from $f(\theta_1,\theta_2,x,y)=[(2-2\cos \theta_1)+(2-2\cos \theta_2)]a(x,y)$ and $t=4$ for different $h$ values.}
\label{FDcoeff_Bnt_full_t4_errore}
\end{figure}
%--------------------------------------------------------------
\begin{figure}
\centering
  \includegraphics[width=\textwidth]{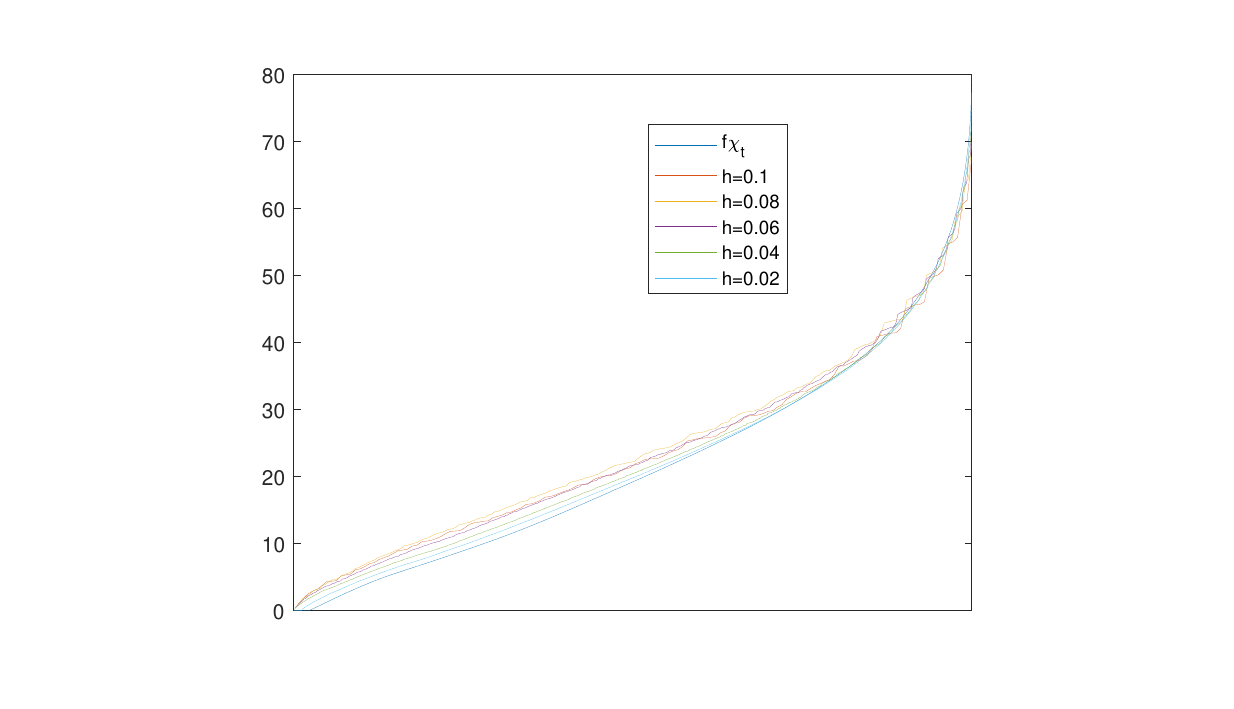} \vskip -0.5cm
  \caption{Eigenvalues distribution of $B_{n,t}(a)$ for different $h$ values  together with the sampling of$f(\theta_1,\theta_2,x,y)=[(2-2\cos \theta_1)+(2-2\cos \theta_2)]a(x,y)$ and $t=6$.}
\label{FDcoeff_Bnt_full_t6}
\end{figure}
\begin{figure}
\centering
  \includegraphics[width=\textwidth]{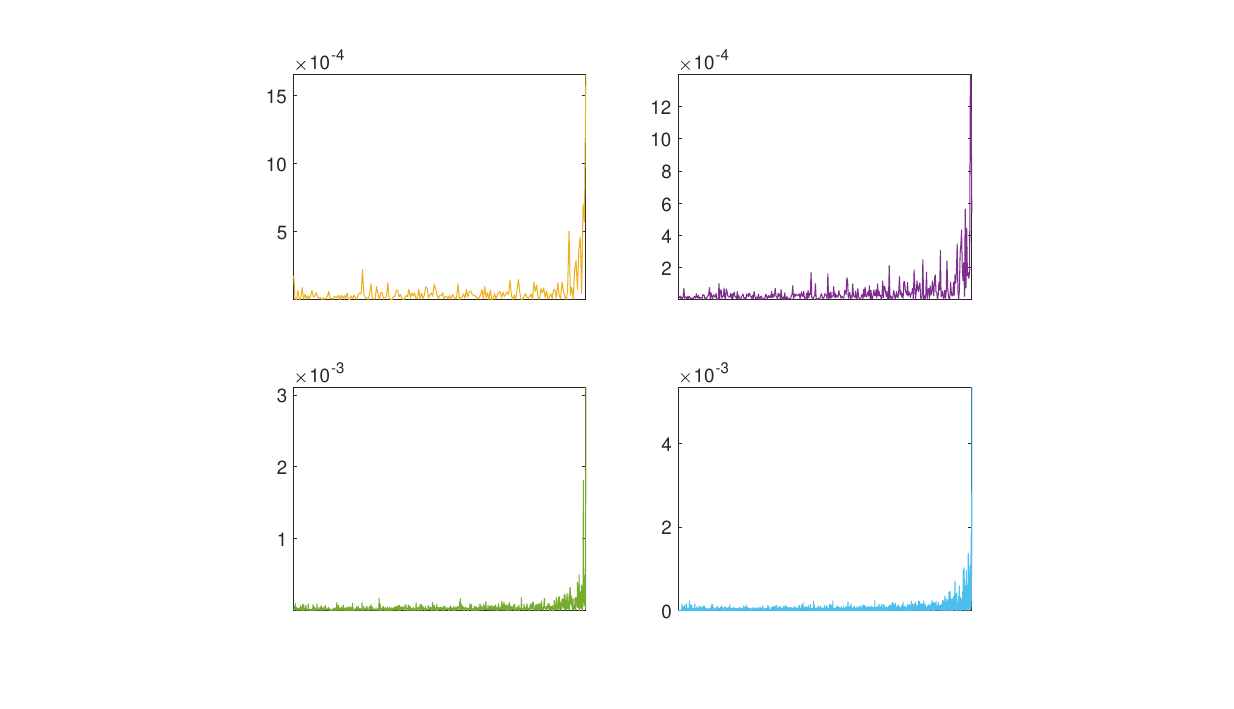} \vskip -0.5cm
  \caption{Minimal distance of eigenvalues of $B_{n,t}(a)$ from $f(\theta_1,\theta_2,x,y)=[(2-2\cos \theta_1)+(2-2\cos \theta_2)]a(x,y)$ and $t=6$ for different $h$ values.}
\label{FDcoeff_Bnt_full_t6_errore}
\end{figure}
%--------------------------------------------------------------
\begin{figure}
\centering
  \includegraphics[width=\textwidth]{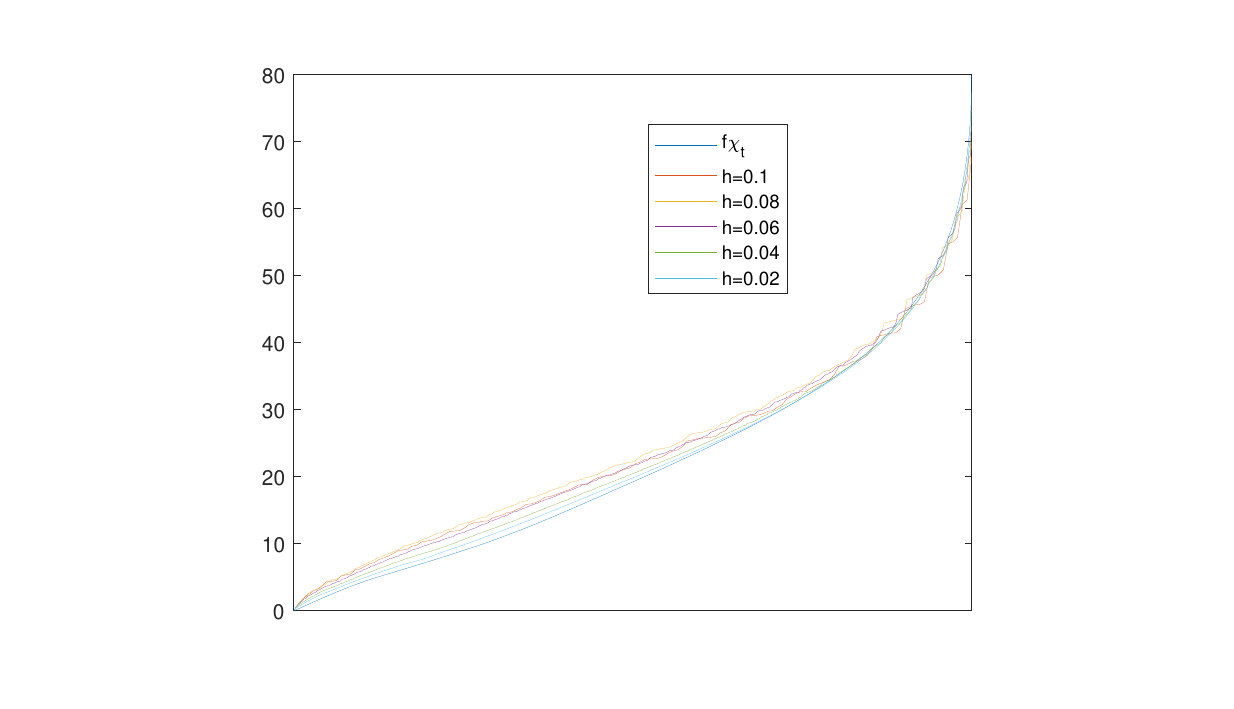} \vskip -0.5cm
  \caption{Eigenvalues distribution of $B_{n,t}(a)$ for different $h$ values  together with the sampling of $f(\theta_1,\theta_2,x,y)=[(2-2\cos \theta_1)+(2-2\cos \theta_2)]a(x,y)$ and $t=8$.}
\label{FDcoeff_Bnt_full_t8}
\end{figure}
\begin{figure}
\centering
  \includegraphics[width=\textwidth]{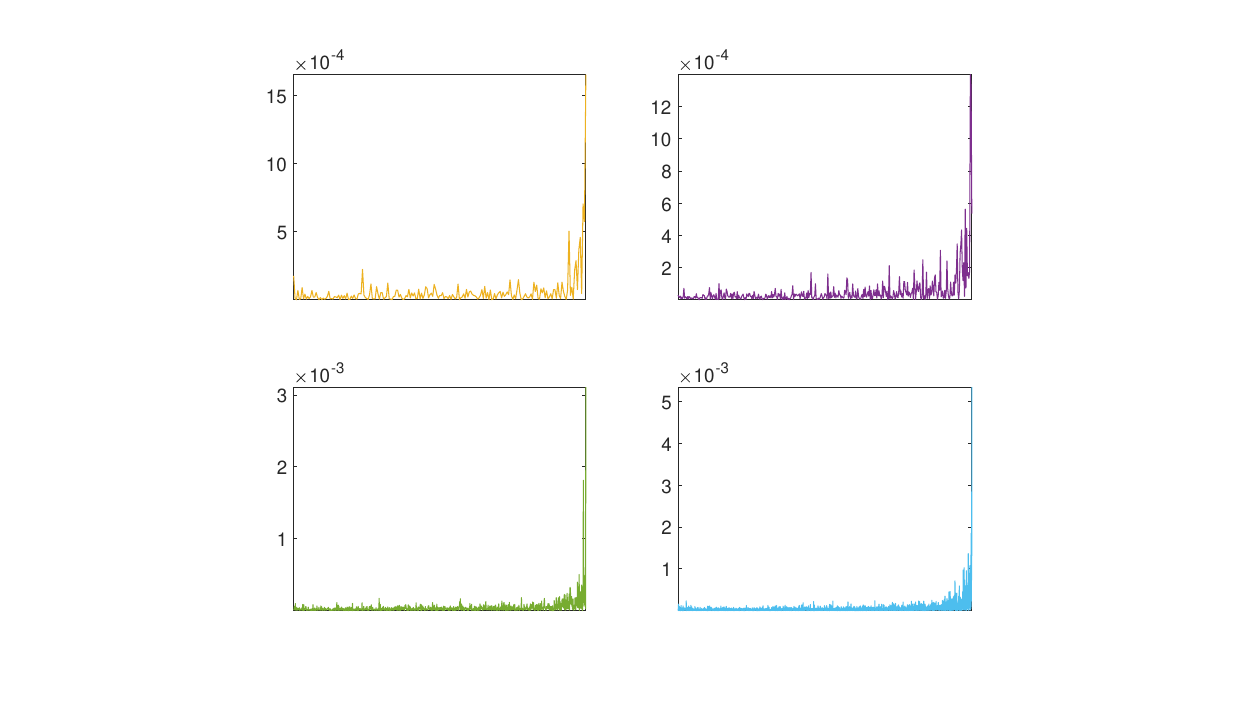} \vskip -0.5cm
  \caption{Minimal distance of eigenvalues of $B_{n,t}(a)$ from $f(\theta_1,\theta_2,x,y)=[(2-2\cos \theta_1)+(2-2\cos \theta_2)]a(x,y)$ and $t=8$ for different $h$ values.}
\label{FDcoeff_Bnt_full_t8_errore}
\end{figure}
%--------------------------------------------------------------
%-------------------------------------------------------------- %

%//////////////////////////////////////////////////

\section{Conclusions}\label{sec:end}

We considered matrix-sequences $\{B_{n,t}\}_n$ of increasing sizes depending on $n$ and equipped with a parameter $t>0$. For every fixed $t>0$, we assumed that each $\{B_{n,t}\}_n$ possesses a canonical spectral/singular values symbol $f_t$ defined on $D_t\subset \R^{d}$ of finite measure, $d\ge 1$. Furthermore we assumed that $ \{ \{ B_{n,t}\} : \, t > 0 \} $ is an a.c.s. for $ \{ A_n \} $ and that $ \bigcup_{t > 0} D_t = D $ with $ D_{t + 1} \supset D_t $. Under such assumptions and via the a.c.s. notion, we proved general distribution results on the canonical distributions of $ \{ A_n \} $, whose symbol, when it exists, can be defined on the possibly unbounded domain $D$ of finite or even infinite measure. In a second theoretical part, we concentrated out attention on the case of unbounded domains of finite measure and we introduce a new concept, the g.a.c.s., which is suited particularly when moving or unbounded domains have to be treated.

Beside the notions of a.c.s and g.a.c.s, the main tool in concrete applications is the theory of GLT matrix-sequences to which usually all the basic matrix-sequences $\{B_{n,t}\}_n$ belong for every $t>0$, as shown in the examples stemming from the numerical approximation of PDEs/FDEs with either moving or unbounded domains.
Several numerical evidences have been given in order to corroborate the analysis.

As open questions we can mention the following main items:
\begin{enumerate}
\item In the present work we focused on the case of unbounded domains with finite measure. However, as already mentioned, the problem of dealing with discretizations on domains of infinite measure remains interesting, both from the point of view of approximation and from the distributional point of view. In this context, also using the simpler a.c.s. notion, there is probably not a distributional symbol for the sequence as in the classical sense, but rather a limit operator $\alpha $ which exploits some features of the discretized operator.
\item The a.c.s. notion has been used as one of the main tools at the foundation of the GLT theory. It is possible that a similar role could be played by the notion of g.a.c.s. for the construction of a new large class of sequences. This will be subject of further study in the future.
\item The study of the eigenvalue distribution in non-normal cases in which the matrix-sequence cannot be viewed as a small perturbation of a Hermitian matrix-sequence is still very intricate (see \cite{Alec1,Alec2,Eric} and the use of potential theory in \cite{saff-book} and references therein). This type of research is very challenging/difficult and it is worth to be considered.
\end{enumerate}

\section*{Declarations}

\noindent
\textbf{Conflict of Interests} The authors declare that they have no conflict of interest.\\
\textbf{Contribution} The authors have all contributed equally to the present work.\\

\section*{Acknowledgments}
\noindent
The research of  all the authors was partly supported by the Italian INdAM-GNCS agency. Furthermore, the work of Stefano Serra-Capizzano was funded from the European High-Performance Computing Joint Undertaking  (JU) under grant agreement No 955701. The JU receives support from the European Union’s Horizon 2020 research and innovation programme and Belgium, France, Germany, Switzerland. Furthermore Stefano Serra-Capizzano is grateful for the support of the Laboratory of Theory, Economics and Systems – Department of Computer Science at Athens University of Economics and Business.

{}

\end{document}